\documentclass{amsart}
\usepackage{amsmath,amscd}
\usepackage{amssymb}
\usepackage{bbm}
\usepackage{enumerate}
\usepackage{amssymb}
\usepackage{array}
\usepackage{graphicx}
\usepackage{amsthm}
\usepackage{units}
\usepackage{dsfont}
\usepackage[arrow, matrix, curve]{xy}
\usepackage[left,modulo]{lineno}

\newtheorem{theorem}{Theorem}[section] 
\newtheorem{lemma}[theorem]{Lemma}
\newtheorem{corollary}[theorem]{Corollary}
\newtheorem{proposition}[theorem]{Proposition}

\theoremstyle{definition}
\newtheorem{definition}[theorem]{Definition}
\newtheorem{example}[theorem]{Example}

\theoremstyle{remark}
\newtheorem{remark}[theorem]{Remark}

\numberwithin{equation}{section}

\newcommand{\op}{\operatorname}

\begin{document}

\title{Wave Front Sets of Reductive Lie Group Representations III}

\author{Benjamin Harris}
\address{Bard College at Simon's Rock, 84 Alford Road, Great Barrington, MA 01230}
\email{Benjamin.Harris@math.okstate.edu}

\author{Tobias Weich}
\address{Institut f\"ur Mathematik, Universitat Paderborn, Paderborn, Germany}
\email{weich@math.uni-paderborn.de}


\subjclass[2010]{43A30,22E45,22E46,43A85}

\date{April 12, 2015}


\keywords{Wave Front Set, Singular Spectrum, Analytic Wave Front Set, Lie Group, Real Linear Algebraic Group, Real Reductive Algebraic Group, Induced Representation, Plancherel Measure, Orbit Method, Tempered Representation}

\begin{abstract}
Let $G$ be a real, reductive algebraic group, and let $X$ be a homogeneous space for $G$ with a non-zero invariant density. We give an explicit description of a Zariski open, dense subset of the asymptotics of the tempered support of $L^2(X)$. Under additional hypotheses, this result remains true for vector bundle valued harmonic analysis on $X$. These results follow from an upper bound on the wave front set of an induced Lie group representation under a uniformity condition.
\end{abstract}

\maketitle

\section{Introduction}
Let $G$ be a real, linear algebraic group with Lie algebra $\mathfrak{g}$. Denote the collection of purely imaginary linear functionals on $\mathfrak{g}$ by $i\mathfrak{g}^*:=\op{Hom}_{\mathbb{R}}(\mathfrak{g},i\mathbb{R})$. The Kirillov-Kostant orbit method seeks to associate a coadjoint orbit $\mathcal{O}_{\pi}\subset i\mathfrak{g}^*$ to every irreducible, unitary representation $\pi$ of $G$. This program is not possible in full generality; some irreducible, unitary representations do not naturally correspond to coadjoint orbits. However, it has been carried out for irreducible, unitary representations occurring in the Plancherel formula for $G$ \cite{Du82}, and there are encouraging signs of success for other important families of representations (see for instance \cite{Vog00}). The passage from $\mathcal{O}_{\pi}$ to $\pi$ is often referred to as \emph{quantization}.
\bigskip

Let $X$ be a homogeneous space for $G$ with a non-zero invariant density. For each $x\in X$, let $G_x$ denote the stabilizer of $x$ in $G$, and let $\mathfrak{g}_x$ denote the Lie algebra of $G_x$. Abstract harmonic analysis seeks to decompose $L^2(X)$ into irreducible, unitary representations of $G$. Let $\op{supp}L^2(X)$ denote the support of the Plancherel formula for $L^2(X)$. Loosely speaking, it is the collection of irreducible, unitary representations of $G$ that occur in the direct integral decomposition of $L^2(X)$.
\bigskip

Folklore in this field posits that the functional analysis question of decomposing $L^2(X)$ under the action of $G$ should be related to the geometry of the action of $G$ on $iT^*X$. One way to make the connection is via the \emph{momentum map}
\begin{equation}\label{eq:momentum_map1}
\mu\colon iT^*X\longrightarrow i\mathfrak{g}^*
\end{equation}
by 
\begin{equation}\label{eq:momentum_map2}
(x,\xi)\mapsto \xi\in iT_x^*X\simeq i(\mathfrak{g}/\mathfrak{g}_x)^*\subset i\mathfrak{g}^*.
\end{equation} 

This leads to the naive conjecture

\begin{equation}\label{conj:eq_1}
\pi\in \op{supp}L^2(X) \Longleftrightarrow \mathcal{O}_{\pi}\subset \mu(iT^*X).
\end{equation}

Unfortunately, this conjecture is false even in simple examples. In this paper, we show that it is \emph{asymptotically true} when $G$ is a real, reductive algebraic group and $\pi$ is an irreducible, tempered representation with regular infinitesimal character. Let us make a more precise statement. 
\bigskip

Let $G$ be a real, reductive algebraic group. The support of the Plancherel measure for $L^2(G)$ is the collection of irreducible, tempered representations of $G$, which we denote by $\widehat{G}_{\text{temp}}$. Within this collection is the set of irreducible, unitary representations of $G$ with regular infinitesimal character, which we denote by $\widehat{G}_{\text{temp}}^{\text{\ }\prime}$. `Most' irreducible, tempered representations have regular, infinitesimal character; more precisely, the complementary set $\widehat{G}_{\text{temp}}\setminus \widehat{G}_{\text{temp}}^{\text{\ }\prime}$ has Plancherel measure zero. Rossmann \cite{Ro78}, \cite{Ro80} and Duflo \cite{Du82} (in greater generality) associate a finite union of coadjoint orbits $\mathcal{O}_{\pi}\subset i\mathfrak{g}^*$ to every irreducible, tempered representation $\pi\in \widehat{G}_{\text{temp}}$. When $\pi$ has regular infinitesimal character, $\mathcal{O}_{\pi}$ is a single coadjoint orbit. 

If $W$ is a finite dimensional real vector space, then $\mathcal{C}\subset W$ is a \emph{cone} in $W$ if $t\mathcal{C}=\mathcal{C}$ whenever $t>0$. If $S\subset W$ is a subset of a finite-dimensional, real vector space, then we define the \emph{asymptotic cone} of $S$ in $W$ to be 
$$\operatorname{AC}(S)=\left\{\xi\in W\Big|\ \xi\in \mathcal{C}\ \text{an\ open\ cone} \implies \mathcal{C}\cap S\ \text{is\ unbounded}\right\}\cup \{0\}.$$
Let $(i\mathfrak{g}^*)'\subset i\mathfrak{g}^*$ denote the Zariski open, dense subset of regular, semisimple elements.
 
\begin{theorem}\label{thm:main} Suppose $G$ is a real, reductive algebraic group, and suppose $X$ is a homogeneous space for $G$ with a non-zero invariant density. Then
\begin{equation}\label{eq:main}
\operatorname{AC}\left(\bigcup_{\substack{ \sigma\in \operatorname{supp}L^2(X)\\ \sigma\in \widehat{G}_{\mathrm{temp}}^{\text{\ }\prime}}}\mathcal{O}_{\sigma}\right)\cap (i\mathfrak{g}^*)'=\overline{\mu(iT^*X)}\cap (i\mathfrak{g}^*)'.
\end{equation}
\end{theorem}

Let us unpack this statement. On the spectral side, we associate a coadjoint orbit to every irreducible, tempered representation with regular infinitesimal character occurring in the decomposition of $L^2(X)$, and then we take the union of all of these orbits. On the geometric side, we take the closure of the image of the momentum map applied to $iT^*X$. Theorem \ref{thm:main} states that after taking asymptotics and intersecting with a suitable Zariski open, dense subset of $i\mathfrak{g}^*$, the spectral side and geometric side agree.

\bigskip

Let us consider an application of Theorem \ref{thm:main}. Let $G=\operatorname{Sp}(2n,\mathbb{R})$, and let
\[X_{l,m,n}:=\operatorname{Sp}(2n,\mathbb{R})/[\operatorname{Sp}(2l,\mathbb{Z})\times \operatorname{Sp}(2m,\mathbb{R})]\]
for $l+m\leq n$. It follows from Theorem~\ref{thm:main} that whenever $2m\leq n$, there exist infinitely many Harish-Chandra discrete series representations $\sigma$ of $\operatorname{Sp}(2n,\mathbb{R})$ such that
\[
\op{Hom}_{\op{Sp}(2n,\mathbb{R})}(\sigma,L^2(X_{l,m,n}))\neq \{0\}. 
\]

We can give a weak converse to this statement. Suppose $T\subset G=\operatorname{Sp}(2n,\mathbb{R})$ is a maximal torus with Lie algebra $\mathfrak{t}\subset \mathfrak{g}=\op{sp}(2n,\mathbb{R})$, and, using the decomposition $\mathfrak{g}=[\mathfrak{t},\mathfrak{g}]\oplus \mathfrak{t}$, identify $i\mathfrak{t}^*\subset i\mathfrak{g}^*$. Let $W=N_G(T)/T$ be the real Weyl group of $T$ with respect to $G$. Every Harish-Chandra discrete series representation $\sigma$ of $\operatorname{Sp}(2n,\mathbb{R})$ corresponds (via its Harish-Chandra parameter) to a single $W$ orbit $\lambda_{\sigma}=\mathcal{O}_{\sigma}\cap i\mathfrak{t}^*$. If $\mathcal{C}\subset \overline{\mathcal{C}}\subset (i\mathfrak{t}^*)'$ is an open cone whose closure is contained in the set of regular semisimple elements of $i\mathfrak{t}^*$ and $2m>n$, then there exist at most finitely many Harish-Chandra discrete series representations $\sigma$ of $\op{Sp(2n,\mathbb{R})}$ such that $\lambda_{\sigma}\cap \mathcal{C}\neq \emptyset$ and 
\[
\op{Hom}_{\op{Sp}(2n,\mathbb{R})}(\sigma,L^2(X_{l,m,n}))\neq \{0\}. 
\]
In the special case $m=0$, much stronger results are already known (see Proposition 10.5 on pages 117-118 of \cite{KK16}). In Section \ref{sec:examples}, we will give a version of Theorem \ref{thm:main} that is easier to compute in examples, and we will explain how to deduce the statements in this example from Theorem \ref{thm:main}.

Due to the recent work of Benoist-Kobayashi \cite{BK15}, one has a wealth of examples of homogeneous spaces $X$ for which one knows $\op{supp}L^2(X)\subset \widehat{G}_{\text{temp}}$.
 
\begin{theorem}\label{thm:tempered} Suppose $G$ is a real, reductive algebraic group, suppose $X$ is a homogeneous space for $G$ with a non-zero invariant density, and suppose 
\[\op{supp}L^2(X)\subset \widehat{G}_{\text{temp}}.\]
Then
\begin{equation}\label{eq:tempered}
\operatorname{AC}\left(\bigcup_{\substack{\sigma\in \operatorname{supp}L^2(X)\\ \sigma\in \widehat{G}_{\mathrm{temp}}}}\mathcal{O}_{\sigma}\right)=\overline{\mu(iT^*X)}.
\end{equation}
\end{theorem}

Philosophically, the reason that we have to intersect both sides of (\ref{eq:main}) with $(i\mathfrak{g}^*)'$ is that there may be non-tempered representations occurring in the decomposition of $L^2(X)$ which should contribute to $i\mathfrak{g}^*\setminus (i\mathfrak{g}^*)'$ on the spectral side. In Theorem \ref{thm:tempered}, we assume that such representations do not occur in the decomposition of $L^2(X)$ yielding a sharper statement. 

Under an additional assumption, we give a bundle valued version of these results in Corollary \ref{cor:dense semisimple cor}.

All of these results are corollaries of new results on wave front sets of induced Lie group representations. In the next section, we outline these results, which the authors believe to be fundamental in their own right. Wave front sets of Lie group representations have other applications beyond those explored in this introduction. For instance, they play a fundamental role in the seminal series of papers \cite{Ko94}, \cite{Ko98b}, \cite{Ko98c}, which give necessary and sufficient conditions for discrete decomposability of Lie group representations. This is the third in a series of papers on wave front sets of Lie group representations \cite{HHO16}, \cite{Ha}. The ideas in this series are heavily influenced by the earlier work of Kobayashi.

\textbf{\emph{Acknowledgements:}}

B. Harris is indebted to Hongyu He and Gestur \'{O}lafsson, his former postdoctoral advisers at Louisiana State University, for their support, encouragement, knowledge, and wisdom. The ideas in this paper are a continuation of the work begun in Baton Rouge.

B. Harris would like to thank Joachim Hilgert and Bernhard Kr\"otz for their hospitality during a visit to Universit\"at Paderborn in the summer of 2014. It was a pleasure to learn from them and their myriad bright postdoctoral colleagues. In particular, the authors started this project during the first author's stay.

B. Harris would like to thank Toshiyuki Kobayashi for his constant encouragement and for several conversations.

T. Weich acknowledges financial support via the grant DFG HI 412 12-1 and he would like to thank Joachim Hilgert, S\"onke Hansen and Bernhard Kr\"otz for helpful remarks and motivating discussions.

\section{The Wave Front Set of an Induced Lie Group Representation}

If $f$ is a continuous function on a Lie group $G$, then $\operatorname{WF}_e(f)=\operatorname{WF}(f)\cap iT^*_eG$ (resp. $\op{SS}_e(f)=\op{SS}(f)\cap iT_e^*G$) denotes the wave front set of $f$ (resp. singular spectrum of $f$) intersected with the fiber of $iT^*G$ over the identity. If $G$ is a Lie group and $(\pi,V)$ is a unitary representation of $G$, then the \emph{wave front set} and \emph{singular spectrum} of $\pi$ are defined by
\[
\operatorname{WF}(\pi)=\overline{\bigcup_{u,v\in V}\operatorname{WF}_e(\pi(g)u,v)},\ \operatorname{SS}(\pi)=\overline{\bigcup_{u,v\in V}\operatorname{SS}_e(\pi(g)u,v)}.
\]
In words, the wave front set of $\pi$ is the closure of the unions of the wave front sets at the identity of the matrix coefficients of $\pi$. These notions were first introduced by Kashiwara-Vergne \cite{KV79} and Howe \cite{Ho81}. It is a well known fact that $\operatorname{WF}(\pi)$ and $\operatorname{SS}(\pi)$ are $\op{Ad}^*(G)$-invariant, closed cones in $i\mathfrak g^* = iT^*_e G$. For precise definitions of the wave front set and singular spectrum which are given in a way that is compatible with the exposition in this paper see Section 2 of the first paper in this series \cite{HHO16}. 

Let $G$ be a Lie group, let $H\subset G$ be a closed subgroup, and let $(\tau,V_{\tau})$ be a unitary representation of $H$. Let $X=G/H$, let $\mathcal{D}^{1/2}\rightarrow X$ denote the bundle of complex half densities on $X$ (see Appendix A), and let $\mathcal{V}_{\tau}=G\times_H V_{\tau}\rightarrow X$ denote the unique $G$ equivariant vector bundle on $X$ whose fiber over $\{H\}$ is the unitary representation $\tau$. Then the representation of $G$ \emph{induced} from the representation $(\tau,V_{\tau})$ of $H$ is defined by \[
\operatorname{Ind}_H^G\tau:=L^2(X,\mathcal{D}^{1/2}\otimes \mathcal{V}_{\tau})
\]
with the natural action of $G$ on the square integrable sections on the vector bundle $\mathcal D^{1/2}\otimes \mathcal V_\tau\to X$.

In this article, we are interested in relating $\operatorname{WF}(\tau)$ and $\operatorname{WF}(\op{Ind}_H^G(\tau))$ (resp. $\operatorname{SS}(\tau)$ and $\operatorname{SS}(\op{Ind}_H^G(\tau))$). Let $\mathfrak{g}$ (resp. $\mathfrak{h}$) denote the Lie algebra of $G$ (resp. $H$), let $i\mathfrak{g}^*$ (resp. $i\mathfrak{h}^*$) be the set of purely imaginary linear functionals on $\mathfrak{g}$ (resp. $\mathfrak{h}$). Then recall that $\operatorname{WF}(\tau),\operatorname{SS}(\tau)\subset i\mathfrak h^*$ and $\operatorname{WF}(\op{Ind}_H^G(\tau)),\operatorname{SS}(\op{Ind}_H^G(\tau))\subset i\mathfrak g^*$ are closed, $\op{Ad}^*(G)$ invariant cones. Using the natural projection $q\colon i\mathfrak{g}^*\rightarrow i\mathfrak{h}^*$ we define a natural way of inducing closed $\op{Ad}^*(G)$ invariant cones in $i\mathfrak g^*$ from closed, $\op{Ad}^*(H)$-invariant cones in $i\mathfrak h^*$: If $S\subset i\mathfrak{h}^*$ is a subset of $i\mathfrak{h}^*$, then we may form the \emph{induced} subset of $i\mathfrak{g}^*$ by
\[
\operatorname{Ind}_{H}^{G}S:=\overline{\operatorname{Ad}^*(G)\cdot q^{-1}(S)}. 
\]
By construction, this is a closed $\op{Ad}^*(G)$-invariant cone in $i\mathfrak g^*$ if $S$ is a closed, $\op{Ad}^*(H)$-invariant cone in $i\mathfrak h^*$.

In Theorem 1.1 of \cite{HHO16}, it is shown that
$$\operatorname{WF}(\operatorname{Ind}_H^G\tau)\supset \operatorname{Ind}_H^G\operatorname{WF}(\tau)$$
and
$$\operatorname{SS}(\operatorname{Ind}_H^G\tau)\supset \operatorname{Ind}_H^G\operatorname{SS}(\tau).$$

In this paper, we address the converse statement. First, we address the case where $X=G/H$ has an invariant measure and $\tau=\mathbbm{1}$ is the trivial representation of $H$. 

\begin{theorem} \label{thm:L^2 space} Suppose $X=G/H$ is a homogeneous space for a Lie group $G$ equipped with a nonzero invariant density. Then 
$$\operatorname{WF}(L^2(X))=\operatorname{WF}(\operatorname{Ind}_H^G\mathbbm{1})=\operatorname{Ind}_H^G\operatorname{WF}(\mathbbm{1})=\overline{\operatorname{Ad}^*(G)\cdot i(\mathfrak{g}/\mathfrak{h})^*}$$
and 
$$\operatorname{SS}(L^2(X))=\operatorname{SS}(\operatorname{Ind}_H^G\mathbbm{1})=\operatorname{Ind}_H^G\operatorname{SS}(\mathbbm{1})=\overline{\operatorname{Ad}^*(G)\cdot i(\mathfrak{g}/\mathfrak{h})^*}.$$
\end{theorem}

This result is proven in Section \ref{sec:scalar}. Observe 
\[\operatorname{Ad}^*(G)\cdot i(\mathfrak{g}/\mathfrak{h})^*=\mu(iT^*(G/H))\]
(see (\ref{eq:momentum_map1}) and (\ref{eq:momentum_map2})). Next, we address the case where $\tau$ is a finite dimensional, unitary representation of $H$. 

\begin{theorem} \label{thm:dense semisimple} Suppose $G$ is a real, linear algebraic group, suppose $H\subset G$ is a closed subgroup with Lie algebra $\mathfrak{h}$, and suppose $\tau$ is a finite dimensional, unitary representation of $H$. Assume the existence of a closed, real algebraic subgroup $H_1\subset G$ with Lie algebra $\mathfrak{h}$ and assume that the set of semisimple elements in $\mathfrak{h}$ is dense in $\mathfrak{h}$. If $\tau$ is a finite dimensional, unitary representation of $H$, then
\[
\operatorname{WF}(\operatorname{Ind}_H^G\tau)=\operatorname{Ind}_H^G\operatorname{WF}(\tau) 
\]
and 
\[
\operatorname{SS}(\operatorname{Ind}_H^G\tau)=\operatorname{Ind}_H^G\operatorname{SS}(\tau).
\]
\end{theorem} 
Observe that the hypotheses of Theorem \ref{thm:dense semisimple} hold whenever $H$ is a real, reductive algebraic group or a parabolic subgroup of a real, reductive algebraic group. However, the hypotheses fail when $H$ is a unipotent group. We want to emphasize that the assumption of dense semisimple elements is not only of technical nature in our proof, but that one cannot hope to prove the upper bounds on the wavefront sets without further assumptions. Already for the classical example of $G=\operatorname{SL}(2,\mathbb{R})$ and $H=N$ being the standard unipotent subgroup (thus not fulfilling the dense semisimple condition) one finds counterexamples to the equality $\operatorname{WF}(\operatorname{Ind}_H^G\tau)=\operatorname{Ind}_H^G\operatorname{WF}(\tau)$. A family of such counterexamples is presented in Section~\ref{sec:counterexamples} utilizing Matumoto's work on the theory of Whittaker functionals for real, reductive algebraic groups \cite{Ma92}. 

Next, we consider the case where $\tau$ is an arbitrary, unitary representation of $H$ and $X=G/H$ is compact. 

\begin{theorem} \label{thm:compact} Suppose $G$ is a Lie group, $H\subset G$ is 
a closed subgroup, $X=G/H$ is a compact homogeneous space, and $\tau$ is a unitary representation of $H$. Then
$$\operatorname{WF}(\operatorname{Ind}_H^G\tau)=\operatorname{Ind}_H^G\operatorname{WF}(\tau)$$
and 
$$\operatorname{SS}(\operatorname{Ind}_H^G\tau)=\operatorname{Ind}_H^G\operatorname{SS}(\tau).$$ 
\end{theorem}

This result was obtained by Kashiwara-Vergne in the case where $G$ is compact \cite{KV79}.  Suppose $G$ is a real, reductive algebraic group, $H=P\subset G$ is a parabolic subgroup, and $P=MAN$ a Langlands decomposition of $P$. When combined with work of Rossmann \cite{Ro95}, the work of Barbasch and Vogan implies Theorem \ref{thm:compact} in the special case where $\tau$ is an irreducible, unitary representation of $MA$ extended to $P$ \cite{BV}. Although we do not give applications of Theorem \ref{thm:compact} in this paper, the authors believe that the case where $\tau$ is a highly reducible representation of $MA$ will likely play an important role in future work.
\bigskip


Note that these Theorems likely hold for an arbitrary reductive Lie group of Harish-Chandra class. We state our corollaries for $G$ a real, reductive algebraic group because arguments in the previous paper \cite{Ha} utilize results that have only been written down in that special case.

Let us end this second introduction with a short outline of the article. In Section~\ref{sec:scalar} we give a direct proof of Theorem~\ref{thm:L^2 space}. In Section \ref{sec:uniformity}, we formulate a general 
\emph{wavefront condition U} under which we can prove 
\[
\operatorname{WF}(\operatorname{Ind}_H^G\tau)=\operatorname{Ind}_H^G\operatorname{WF}(\tau) 
\]
and an analogous \emph{singular spectrum condition U} under which we show that
\[
\operatorname{SS}(\operatorname{Ind}_H^G\tau)=\operatorname{Ind}_H^G\operatorname{SS}(\tau). 
\]
In Section \ref{sec:CU_compact} and Section \ref{sec:CU_dense_semisimple}, we verify wavefront condition U and singular spectrum condition U in special cases thereby deducing Theorem \ref{thm:dense semisimple} and Theorem \ref{thm:compact}. Theorem~\ref{thm:L^2 space} could also be obtained by verifying wavefront condition U and singular spectrum condition U. We nevertheless wanted to include the direct proof of Theorem~\ref{thm:L^2 space} because it nicely illustrates the central ideas which are exploited in the following sections.
We therefore also end Section~\ref{sec:scalar} with a discussion of the challenges which one encounters in generalizing these ideas to nontrivial $H$-representations $\tau$ and how they are handled in the subsequent sections. In Section~\ref{sec:examples} we show how to use Theorem \ref{thm:L^2 space} to deduce Theorem \ref{thm:main} and Theorem \ref{thm:tempered}. We also utilize Theorem \ref{thm:dense semisimple} to deduce additional results, and we present concrete examples of these results. Finally in Section~\ref{sec:counterexamples} we attempt to justify the extra hypotheses in Theorem \ref{thm:dense semisimple} by providing counterexamples to a stronger statement.

\section{Proof of Theorem \ref{thm:L^2 space}}
\label{sec:scalar}

In this section, we give a proof of Theorem \ref{thm:L^2 space}.

As in the introduction, let $G$ be a Lie group with Lie algebra $\mathfrak{g}$, 
let $H\subset G$ be a closed subgroup with Lie algebra $\mathfrak{h}$, and assume 
$X=G/H$ has a nonzero invariant density. If 
$\eta_0 \notin \overline{\op{Ad}^*(G)\cdot i(\mathfrak{g}/\mathfrak{h})^*}$, 
then it suffices to show that $\eta_0\notin \op{SS}(\op{Ind}_H^G\mathds 1)$. 
This suffices for the wavefront case as well since 
$\op{WF}(\op{Ind}_H^G\mathds 1)\subset \op{SS}(\op{Ind}_H^G\mathds 1)$.

To handle the singular spectrum, we require some notation. Choose a basis 
$\{X_1,\ldots,X_n\}$ of $\mathfrak{g}$, and for every multi-index 
$\alpha=(\alpha_1,\ldots,\alpha_n)\in\mathbb{N}^n$, define the differential operator 
\[
D^{\alpha}=\partial_{X_1}^{\alpha_1}\partial_{X_2}^{\alpha_2}\cdots \partial_{X_n}^{\alpha_n}
\]
on $\mathfrak{g}$. If $0\in U_1\subset U_2\subset \mathfrak{g}$ are precompact, 
open sets with $\overline{U_1}\subset U_2$, then (see pages 25-26, 282 of \cite{Hor83}), 
we may find a sequence of functions $\{\varphi_{N,U_1,U_2}\}$ indexed by $N\in \mathbb{N}$ 
satisfying the following properties:
\begin{itemize}

\item $\varphi_{N,U_1,U_2}\in C_c^{\infty}(U_2)$  
\item $\varphi_{N,U_1,U_2}(x)=1\ \text{if}\ x\in U_1$
\item There exist constants $C_{\alpha}>0$ for every multi-index 
$\alpha=(\alpha_1,\ldots,\alpha_n)$ such that
\begin{equation}\label{eq:varphi_def}
|D^{\alpha+\beta}\varphi_{N,U_1,U_2}(x)|\leq C_{\alpha}^{|\beta|+1} (N+1)^{|\beta|}
\end{equation}
for every $x\in U_2$ and every multi-index $\beta=(\beta_1,\ldots,\beta_n)$. 
Here $|\beta|=\beta_1+\cdots+\beta_n$.
\end{itemize}

For every pair of subsets $0\in U_1\subset U_2\subset \mathfrak{g}$, we fix such 
a sequence of functions $\{\varphi_{N,U_1,U_2}\}$.

The strategy of the proof is as follows. Let 
$\eta_0 \notin \overline{\op{Ad}^*(G)\cdot i(\mathfrak{g}/\mathfrak{h})^*}$. 
We have to show (see for instance Definition 2.3 of \cite{HHO16})
that for fixed $f_1,f_2\in L^2(X)$, 
there exists an open neighborhood $\eta_0\in V_0\subset i\mathfrak{g}^* $ and open subsets 
$0\in U_1\Subset U_2\subset \mathfrak{g}$ such that
\begin{equation}\label{eq:SS_bound}
 (\star):=t^N\left|\int\limits_G\int\limits_X f_1(g^{-1} x) \overline{f_2(x)} (\log^*\varphi_{N,U_1,U_2})(g)e^{t\langle \xi,\log(g)\rangle} dx dg\right| \leq C^{N+1}(N+1)^N
\end{equation}
uniformly for $t>0$,  $\xi\in V_0 $ and for all $N\in \mathbb{N}$ (Note that the 
constant $C$ must not depend on $N$). Here $dg$ is a nonzero invariant density 
on $G$ and $dx$ is a nonzero invariant density on $X$; note that we must choose 
$U_2\subset \mathfrak{g}$ sufficiently small for the logarithm function to be 
well defined. As $\varphi_{N,U_1,U_2}$ is 
compactly supported all the integrals are absolutely convergent and we can 
interchange the order of integration and get
\begin{equation}\label{eq:SS_int1}
 (\star)=t^N\left|\int\limits_X \int\limits_Gf_1(g^{-1} x) \overline{f_2(x)} 
 (\log^*\varphi_{N,U_1,U_2})(g)e^{t\langle \xi,\log(g)\rangle} dg dx\right|
\end{equation}
The idea is now to prove the bound (\ref{eq:SS_bound}) by integrating the 
$G$-integral by parts with respect to a right invariant vector field $Y_x$ on $G$ 
that depends continuously on the point $x\in X$.

We will now construct these vector fields. First note that 
\[
\overline{\op{Ad}^*(G)\cdot i(\mathfrak{g}/\mathfrak{h})^*} = \overline{\bigcup_{x\in X}  
                                                \left(i\mathfrak{g}/\mathfrak{g}_x\right)^*}
\]
where $\mathfrak{g}_x$ denotes the Lie algebra of the stabilizer subgroup $G_x\subset G$
of the point $x\in X$. We fix an arbitrary, not necessarily 
$\op{Ad}(G)$ invariant scalar product on $\mathfrak{g}^*$. By this scalar product we can identify $\mathfrak{g}\cong\mathfrak{g}^*$ 
and obtain
\[
 (\mathfrak{g}/\mathfrak{g}_x)^* \cong \mathfrak{g}_x^\perp
\]
where $\mathfrak{g}_x^\perp$ denotes the orthogonal complement with respect to the chosen scalar product.

For $\eta_0\notin \overline{\bigcup\limits_{x\in X}  i\mathfrak{g}_x^\perp}$, we define a
continuous family of normalized elements in the Lie algebra $\mathfrak{g}$ parametrized by
$x\in X$,
\[
 Y_x := \frac{\textup{pr}_{\mathfrak{g}_x}(-i \eta_0)}{|\textup{pr}_{\mathfrak{g}_x}(\eta_0)|}\in \mathfrak{g}.
\]
Here $\textup{pr}_{\mathfrak{g}_x}$ denotes the orthogonal projection on the subspace $\mathfrak{g}_x$.
Note that $Y_x$ is well defined and continuous in $x\in X$ as 
$\eta_0\notin \overline{\bigcup\limits_{x\in X} i\mathfrak{g}_x^\perp}$ implies that
$|\textup{pr}_{\mathfrak{g}_x}(\eta_0)|$ is bounded away from zero. 

If we consider $Y_x$ as a right invariant vector field on $G$  we can study its
action on the smooth function $\exp(t\langle \xi,\log(\cdot)\rangle)$ when restricted to $\tilde{U}_2:=\exp(U_2)\subset G$
\[
 Y_x e^{t\langle \xi,\log(g)\rangle} = \frac{d}{ds}\Big|_{s=0} e^{t\langle \xi,\log(e^{sY_x} g)\rangle} =t \mu(\xi,Y_x,g) e^{t\langle \xi,\log(g)\rangle}
\]
where 
\begin{equation}\label{eq:mu}
  \mu(\xi,Y,g):=\frac{d}{ds}\Big|_{s=0} \langle \xi,\log(e^{sY} g)\rangle. 
\end{equation}
In the above expression, when $X\in \mathfrak{g}$, we are writing $e^X:=\exp(X)$ 
for the image of $X$ in $G$ under the exponential map.
We can thus insert the operator $(t^{-1}\mu(\xi,Y_x,g)^{-1}Y_x)^N$ in (\ref{eq:SS_int1})
in front of $e^{t\langle \xi,\log(g)\rangle}$ and integrate the $G$ integral 
by parts:
\begin{eqnarray} \label{eq:partial_integral}
(\star)&= t^N\left|\int\limits_X \int\limits_G f_1(g^{-1} x) \overline{f_2(x)} 
(\log^*\varphi_{N,U_1,U_2})(g) \left(t^{-1}\mu(\xi,Y_x,g)^{-1}Y_x\right)^Ne^{t\langle \xi,\log(g)\rangle} dg dx\right|
\nonumber \\
&\ \ =\left|\int\limits_X  \int\limits_G f_1(g^{-1} x) \overline{f_2(x)}
\left((Y_x \mu(\xi,Y_x,g)^{-1})^N (\log^*\varphi_{N,U_1,U_2})(g)\right) e^{t\langle\xi,\log(g)\rangle} dg dx\right|
\end{eqnarray}
where we used the fact that $Y_x\in\mathfrak{g}_x$ and thus $Y_xf_1(g^ {-1}x) =0$. 
Utilizing (\ref{eq:varphi_def}) as well as the fact that $\mu(\xi,Y_x,g)$ is an 
analytic function, we obtain bounds
\[
|(Y_x \mu(\xi,Y_x,g)^{-1})^N(\log^*\varphi_{N,U_1,U_2})(g)| \leq C^{N+1}(N+1)^N |\mu(\xi,Y_x,g)|^{-2N}
\]
for some constant $C$ independent of $g\in \exp(U_2)$, $\xi\in V_0$ and $Y_x$. 
It thus remains to consider the term $\mu(\xi,Y_x,g)^{-2N}$. Note that from the 
definition of $\mu$ and $Y_x$ we have
\[
  |\mu(\eta_0,Y_x,e)| = |\op{pr}_{\mathfrak g_x}(\eta_0)|
\]
and as remarked above this quantity is bounded away from zero for $x\in X$ so we 
have a positive constant
\[
 c:=\inf\limits_{x\in X} |\op{pr}_{\mathfrak g_x}(\eta_0)|>0
\]
As $\mu: i\mathfrak g^*\times \mathfrak g \times G \to \mathbb C$ is a smooth function 
and as $Y_x$ only takes values in the compact unit sphere in $\mathfrak g$ we 
can choose sufficiently small neighborhoods $V_0 \subset i\mathfrak g^*$ of $\eta_0$
and $\tilde{U_2}\subset G$ of $e$ such that 
\[
 |\mu(\xi,Y_x,g)|\geq \frac{c}{2}
\]
uniformly in $\xi\in V_0 , g\in \tilde{U}_2$ and $x\in X$. 

Then we obtain for all $N\in\mathbb N$
\begin{eqnarray*}
(\star) &\leq& \left(\frac{2}{c}\right)^{2N} C^{N+1}(N+1)^N\int_X \int_{\exp(U_2)} |f_1(g^{-1} x)f_2(x)|dg dx\\
&\leq& \left(\frac{2}{c}\right)^{2N} C^{N+1}(N+1)^N \|f_1\|_{L^2(X)}\|f_2\|_{L^2(X)} \op{vol}(\exp(U_2))\\
&\leq& C_1^{N+1}(N+1)^N
\end{eqnarray*}
uniformly in $t>0$ and $\xi\in V_0 $. We have thus established the bound 
(\ref{eq:SS_bound}) and proven Theorem~\ref{thm:L^2 space}.

Note that the crucial point in this proof was first the interchanging of the
order of integration on $X$ and $G$ and second the partial integration
on $G$ performed for each point $x\in X$ into a direction of the stabilizer subalgebra $\mathfrak g_x$ 
of the point $x$. Only because this $x$-dependent choice of our differential operator
were we able to obtain in (\ref{eq:partial_integral}) that the factor $f_1(g^{-1} x) f_2(x)$
is differentiable (even constant) into this direction (Note that in other directions
not in $\mathfrak g_x$ this would not have been the case because $f_1$ is not smooth
but only in $L^2(X)$). 

Let us now discuss what difficulties arise if we do not induce from the
trivial representation but instead start from a finite dimensional unitary 
$H$-representation $(\tau,V)$. We thus have to consider $L^2$-sections in the 
Hermitian vector bundle $\mathcal V=G\times_{\tau} V$. Let 
$\langle\bullet,\bullet\rangle_{\mathcal V_x}$ denote the scalar product in the fibre 
$\mathcal V_x$. Then, in (\ref{eq:partial_integral}), we then have to
derive $\langle f_1(g^{-1} x), f_2(x)\rangle_{\mathcal V_x}$ into a direction
$Y_x\in \mathfrak g_x$. As $\tau$ is finite dimensional and thus smooth this is 
still possible, but the derivatives do not vanish anymore. Instead there appear
powers of the operator $d\tau_x(Y_x)$ where $\tau_x$ is the unitary 
representation of the stabilizer subgroup $G_x$ on the fibre $\mathcal V_x$. While
for any $x\in X$ the linear operators $d\tau_x(Y_x)$ are bounded, it is in general
false, that this bound is uniform in $x\in X$. However the dense semisimple
condition introduced in Section~\ref{sec:CU_dense_semisimple} will allow us to
obtain these uniform bounds and prove the upper bounds on the wavefront sets. 
Note that the question of uniform bounds of $d\tau_x(Y_x)$ is not only a technical 
problem of our proof-strategy but the examples in Section~\ref{sec:counterexamples}
show that these non-uniform bounds may indeed lead to larger wavefront sets. 

A similar problem occurs if $X=G/H$ admits no $G$ invariant measure anymore and
if one has to tensor the Hermition vector bundle $\mathcal V$ with the half density
bundle $\mathcal D^{1/2}$. In this case partial differentiation does not only 
create  derivatives of $\tau_x$ but there occur also 
derivatives of the one dimensional $G_x$-representations $\sigma_x$ on the 
fibres $\mathcal D^{1/2}_x$ of the half density bundle. However, we will see that
the dense semisimple condition is also suited to find uniform bounds for these 
terms, so the additional difficulties coming from the density bundles is of the 
same class as the one coming from inducing from finite dimensional representations
$\tau$. 

The reason why the partial integration approach would work for finite dimensional 
$(\tau,V)$ comes from the fact that it has a trivial, empty wavefront set, 
i.e. all matrix coefficients are smooth. If we, however, try to apply this approach to general 
infinite dimensional representations $(\tau,V)$ we face the problem that 
$\langle f_1(g^{-1} x), f_2(x)\rangle_{\mathcal V_x}$ might not be differentiable
anymore in the $\mathfrak g_x$-direction. The partial integration approach breaks down.
We however still can use the principal idea of the partial integration approach, 
which is to interchange the $X$ and $G$-integral and use the $x$-dependent 
splitting $\mathfrak g=\mathfrak g_x\oplus\mathfrak g_x^\perp$. This way we can 
reduce our problem to studying the oscillating integrals belonging to the matrix coefficients of 
$(\tau_x,\mathcal V_x)$ (see the calculation at the beginning of Section~\ref{sec:uniformity}). 
Of course the same uniform bound problems as in the finite dimensional case also occurs
in a more general formulation involving oscillating integrals. In Section~\ref{sec:uniformity} 
we prove the upper bounds on the wavefront set (resp. singular spectrum) under some precise 
uniformity condition on these oscillating integrals which we term \emph{wave front condition U} 
(resp. \emph{singular spectrum condition U}). The advantage of this approach is that for compact $X$ 
and arbitrary unitary $\tau$ these uniformity conditions can be verified without
too much additional effort; we carry this out in Section \ref{sec:CU_compact}. 
In addition, the proof of condition U for finite dimensional 
$\tau$ in the dense semisimple case is easier than\ a possible direct proof of 
the upper bound on the wavefront set by partial integration. Finally we are 
convinced that these uniformity conditions will also be useful in the future for 
proving upper bounds of the wavefront set for cases not covered in this article,
for example for noncompact $X$ and special infinite dimensional 
representations $\tau$. 

\section{A Uniformity Condition and the Wavefront Set}
\label{sec:uniformity}

Let $G$ be a Lie group, let $H\subset G$ be a closed subgroup, and let $(\tau,V)$ be a unitary 
representation of $H$ on a possibly infinite dimensional Hilbert space $V$. As in the introduction,
we form the unitary representation $\op{Ind}_H^G\tau$ of $G$. In Theorem 1.1 
of \cite{HHO16}, it is shown that
$$\op{WF}(\op{Ind}_H^G\tau)\supset \op{Ind}_H^G\op{WF}(\tau);\ \op{SS}(\op{Ind}_H^G(\tau))\supset \op{Ind}_H^G\op{SS}(\tau).$$
In this section, we formulate a wavefront condition U on $G$, $H$, and $\tau$, and we show that 
\[
\op{WF}(\op{Ind}_H^G\tau)\subset \op{Ind}_H^G\op{WF}(\tau)
\]
when the wavefront condition U holds. In addition, we formulate an analogous 
singular spectrum condition U on $G$, $H$, and $\tau$, and we show that 
\[
\op{SS}(\op{Ind}_H^G\tau)\subset \op{Ind}_H^G\op{SS}(\tau) 
\]
when the singular spectrum condition U holds. In order to formulate the wave 
front condition U and the singular spectrum condition U, we require additional 
notation. 

Let $X=G/H$ be the corresponding homogeneous space on which $G$ acts transitively from the left. 
If $x\in X$, we denote by $G_x\subset G$ the stabilizer 
subgroup of the point $x$ in $G$. Obviously we have $G_{eH}=H$ and for all $x$,
$G_x$ is conjugate to $H$. Let $\mathcal{V}=G\times_H V$ denote the $G$ equivariant bundle on $X=G/H$ 
associated to $(\tau,V)$, and let $\mathcal{D}^{1/2}\rightarrow X$ denote the 
bundle of complex half densities on $X$ (See Appendix \ref{app:dense}).
The group $G$ acts in the standard way from the left on $\mathcal V\otimes \mathcal D^ {1/2}$
by $g[x,v\otimes z]=[gx,v\otimes z]$. This left action leads to a unitary representation 
on $L^2(X,\mathcal{V}\otimes \mathcal{D}^{1/2})$ which we denote by 
$\operatorname{Ind}_H^G\tau$. 

Let $\mathcal{V}_x$ (resp. $\mathcal{D}^{1/2}_x$) denote the fiber of 
$\mathcal{V}$ (resp. $\mathcal{D}^{1/2}$) over $x\in X$. Then the left $G$ action 
induces a representation of $G_x$ on $\mathcal{V}_x$ (resp. $\mathcal{D}_x^{1/2}$)
which we denote by $\tau_x$ (resp. $\sigma_x$).
Note that when $x=eH$, $\tau_{eH}$ coincides with $\tau$. We analogously denote 
$\sigma_{eH}$ by $\sigma$. 
Observe that we have the formula
\begin{equation}\label{eq:sigma_eH}
\sigma_{eH}(h) = |{\det}_{T_{eH}X}(dh_{|eH})|^ {-1/2} 
\end{equation}
where $\det_{T_{eH}X}(dh_{|eH})$ is the determinant of the differential 
\[
dh_{|eH}:T_{eH}X\to T_{eH}X
\]
(see (\ref{eq:sigma_x_explicit})).

Let $\mathfrak{g, h}$, and $\mathfrak{g}_x$ denote the Lie algebras of $G,H$, 
and $G_x$ respectively. Note that on the Lie algebra level there is a canonical 
embedding $\mathfrak g_x\hookrightarrow\mathfrak g$ as a subalgebra and on 
the dual side there is a canonical projection $q_x:i\mathfrak{g}^*\rightarrow i\mathfrak{g}_x^*$ 
which is defined by the restriction of a form in $i\mathfrak g^*$ to elements in 
$\mathfrak g_x\subset \mathfrak g$. Throughout the article, we 
fix an arbitrary scalar product on $\mathfrak g$ which allows us
to identify Lie algebras with their adjoints and defines a unique Lebesgue measure on 
$\mathfrak g$ and all of its linear subspaces. 

For $x=eH$ we will drop the subscript in $q_x$ and write
$q:=q_{eH}:i\mathfrak{g}^* \rightarrow i\mathfrak{h}^*$. 
If $S\subset i\mathfrak{h}^*$ is a subset, recall from the introduction
the notation
\[
\operatorname{Ind}_H^G S=\overline{\operatorname{Ad}^*(G)\cdot q^{-1}(S)}\subset i\mathfrak g^*.
\]

In the sequel we will be particularly interested in $\operatorname{Ind}_H^G \operatorname{WF}(\tau)$.
We first want to relate this quantity to the representations $\tau_x$ on the 
fibers over $x$. For this purpose, we recall how the representations $\tau_x$ and their wavefront 
sets are related. Let $x\in X$ and fix $g_x$ such that $x=g_x H$. Then conjugation by
$g_x$ defines an isomorphism
\[
 \operatorname{C}_{g_x} : H \to G_x,~h\mapsto g_x h g_x^{-1}.
\]
Additionally, the left action by $g_x$ defines a Hilbert space isomorphism 
\[
 g_x: V = \mathcal{V}_{eH} \to \mathcal V_x
\]
and we obtain for $h\in H$
\begin{equation}\label{eq:tau_x}
 \tau(h) = g_x^{-1} \tau_x(\operatorname{C}_{g_x} h) g_x.
\end{equation}

Thus, $\tau$ and $\tau_x\circ\operatorname{C}_{g_x}$ are equivalent representations.
In the same way we obtain
\begin{equation}\label{eq:sigma_x}
 \sigma(h) = g_x^{-1} \sigma_x(\operatorname{C}_{g_x} h) g_x.
\end{equation}

For the wavefront sets one consequently obtains from \cite[Theorem 8.2.4]{Hor83}
\begin{equation}\label{eq:WF_tau_x}
\operatorname{WF}(\tau) = \operatorname{WF}(\tau_x \circ \operatorname{C}_{g_x}) = (\operatorname{Ad}(g_x))^*(\operatorname{WF}(\tau_x)).
\end{equation}
Here $\operatorname{Ad}(g_x): \mathfrak h\to\mathfrak g_x$ appears as the differential 
of $C_{g_x}$ in the identity element $e\in H$ and by pullback it defines canonically
an isomorphism $(\operatorname{Ad}(g_x))^*:i\mathfrak g_x^* \to i\mathfrak h^*$. 
The analogous statement for the singular spectrum
\begin{equation}\label{eq:SS_tau_x}
\operatorname{SS}(\tau) = \operatorname{SS}(\tau_x \circ \operatorname{C}_{g_x}) = (\operatorname{Ad}(g_x))^*(\operatorname{SS}(\tau_x))
\end{equation}
can be obtained from \cite[Theorem 8.5.1]{Hor83} (H\"ormander uses the term 
\emph{analytic wave front set} instead of singular spectrum and writes $\op{WF}_{\op{A}}$ 
instead of $\op{SS}$ in his book).

Pullback and coadjoint action are always compatible with the natural projections
as recorded in the following Lemma.

\begin{lemma}
If $\operatorname{Ad}^*(g_x) : i\mathfrak g^*\to i\mathfrak g^*$ is the 
coadjoint representation and $(\operatorname{Ad}(g_x))^*:i\mathfrak g_x^*\to i\mathfrak h^*$
the pullback map, then
\[
 q\circ \operatorname{Ad}^*(g_x^{-1})=(\operatorname{Ad}(g_x))^*\circ q_x.
\] 
\end{lemma}
\begin{proof}
 If $\xi\in i\mathfrak g^*$ and $H\in \mathfrak h$, then we calculate
 \begin{align*}
  [q\circ \operatorname{Ad}^*(g_x^{-1}) (\xi)](H) &= [\operatorname{Ad}^*(g_x^{-1}) (\xi)](H)\\
  &=\xi(\operatorname{Ad}(g_x) H) \\
  &= [q_x(\xi)](\operatorname{Ad}(g_x)H)\\
  &=[(\operatorname{Ad}(g_x))^*\circ q_x](H).
 \end{align*}
\end{proof}
From this Lemma, (\ref{eq:WF_tau_x}), and (\ref{eq:SS_tau_x}) we conclude
\[
 q_x^{-1}(\operatorname{WF}(\tau_x)) = \operatorname{Ad}^*(g_x)q^{-1}(\operatorname{WF}(\tau)).
\]
and 
\[
 q_x^{-1}(\operatorname{SS}(\tau_x)) = \operatorname{Ad}^*(g_x)q^{-1}(\operatorname{SS}(\tau)).
\]
Finally we can express
\[
\mathcal W:=\operatorname{Ind}_H^G\operatorname{WF}(\tau) =\overline{\bigcup\limits_{g\in G}\operatorname{Ad}^*(g) q^{-1}(\operatorname{WF}(\tau))} = \overline {\bigcup\limits_{x\in X}q_x^ {-1}(\operatorname{WF}(\tau_x))}
\]
and
\[
\mathcal S:=\operatorname{Ind}_H^G\operatorname{SS}(\tau) =\overline{\bigcup\limits_{g\in G}\operatorname{Ad}^*(g) q^{-1}(\operatorname{SS}(\tau))} = \overline {\bigcup\limits_{x\in X}q_x^ {-1}(\operatorname{SS}(\tau_x))}.
\]

As a consequence for any $\eta \notin \mathcal W$ and for any $x\in X$ we have 
$q_x(\eta)\notin \operatorname{WF}(\tau_x)$
and by \cite[Theorem 1.4]{Ho81} there is a neighborhood 
$V_x\subset i\mathfrak{g}_x$ of $q_x(\eta)$ and a function $\varphi_x\in C^\infty_c(\mathfrak g_x)$
supported in a neighborhood of $0\in \mathfrak g_x$ such that for any $N$ there 
is $C_{N,x}>0$ such that
\begin{equation}
 \label{eq:nonuniform_decay}
 \left|\int\limits_{\mathfrak g_x} \langle \tau_x(e^ Y) v_1,v_2\rangle_{\mathcal{V}_x} \varphi_x(Y) e^{\langle t\xi,Y\rangle}dY\right| \leq C_{N,x} \|v_1\|_{\mathcal V_x}\|v_2\|_{\mathcal V_x} t^{-N}
\end{equation}
for $t>0,~v_1,v_2\in \mathcal V_x$. Here $dY$ is the 
Lebesgue measure on $\mathfrak g_x$ which is fixed by the choice of the metric
on $\mathfrak g$. The estimate is uniform in $\xi\in V_x$; however, we have 
a-priori no information about the uniformity of these estimates in $x$. 

As the induced representations are modeled on $\mathcal V\otimes \mathcal D^ {1/2}$,
we will not only have to consider the representations $\tau_x$
but $\tau_x\otimes \sigma_x$ on $\mathcal V_x\otimes \mathcal D^{1/2}_x$
and will be led to the study of
\[
 \left|\int\limits_{\mathfrak g_x} \langle \tau_x(e^ Y) v_1,v_2\rangle_{\mathcal{V}_x} 
 \left((\sigma_x(e^Y) z_1) \otimes \overline{z_2}\right) \varphi_x(Y)
 e^{\langle t\xi,Y\rangle} dY\right|_{\mathcal D_x}
\]
which now takes values in the fiber $(\mathcal D^{1}_{\geq 0})_x$. From (\ref{eq:sigma_eH})
and (\ref{eq:sigma_x}) we conclude that for every $x$, the factor 
$\left((\sigma_x(e^Y) z_1) \otimes \overline{z_2}\right)$ is a smooth function 
on $Y$ with values in $\mathcal D^1_x$ 
and consequently one immediately gets fast decay analogous to (\ref{eq:nonuniform_decay})
with an $x$ dependent constant $C_{x,N}$. However the presence of the additional 
factor might be an additional source of $x$ dependence of these constants. 

The following condition imposes a uniformity
in $x\in X$ on the decay of these Fourier transforms of the matrix coefficients
and we will see in Sections~\ref{sec:CU_compact} and \ref{sec:CU_dense_semisimple}
that this condition holds in many important examples.
\begin{definition}[Wavefront Condition U]\label{def:wf_cond_U}
 With the notation from above we say that \emph{wavefront condition U} is satisfied if
 for all $\eta\notin \op{Ind}_H^G\op{WF}(\tau)$ there is a neighborhood 
 $\Omega \subset i\mathfrak{g}^*\setminus \op{Ind}_H^G\op{WF}(\tau)$ of $\eta$ and a neighborhood
 $U_{\textup{CU,WF}} \subset \mathfrak g$ around $0 \in \mathfrak g$ such that for
 all functions  $\varphi\in C^{\infty}_c(U_{\textup{CU,WF}})$ and for all 
 $N\geq 0$ there is a constant $C_{N,\varphi}$ 
 such that
 \begin{equation}
 \label{eq:wf_uniform_decay}
 \begin{split}
 \left|\int\limits_{\mathfrak g_x} \langle \tau_x(e^Y) v_1,v_2\rangle_{\mathcal{V}_x}\cdot
 \left((\sigma_x(e^Y) z_1) \otimes \overline{z_2}\right) \cdot\varphi(Y) e^{\langle t\xi,Y\rangle}
 dY\right|_{\mathcal D^1_x} \\ 
 \leq C_{N,\varphi} (\|v_1\|_{\mathcal V_x} |z_1|_{\mathcal D^{1/2}_x})  \otimes(\|v_2\|_{\mathcal V_x}
 |z_2|_{\mathcal D^{1/2}_x} )t^{-N}
 \end{split}
\end{equation}
for $t>0$ and uniformly for all $\xi\in\Omega$, $x\in X$, $v_1,v_2\in \mathcal{V}_x$, 
and $z_1,z_2\in \mathcal{D}_x^{1/2}$. Additionally we 
demand that there is a constant $C$ such that for any $x\in X$, any 
$Y\in \mathfrak g_x\cap U_{\textup{CU,WF}}$ and any $z\in \mathcal D_x^{1/2}$, we 
have
\begin{equation}
 \label{eq:wf_uniform_sigma_bound}
 |\sigma_x(e^Y)z|_{\mathcal D_x^{1/2}}\leq C|z|_{\mathcal D_x^{1/2}}.
\end{equation}
\end{definition}

\begin{remark}
 Note that (\ref{eq:wf_uniform_decay}) is formulated as an inequality in the 
 fiber $(\mathcal D^1_{\geq_0})_x$. Using a local point of reference 
 $0\neq z\in(\mathcal D^1_{\geq 0})_x$
 these equations can simply be reduced to an inequality of non-negative real numbers.
 The advantage of the formulation (\ref{eq:wf_uniform_decay}) is that the left
 side as well as the right side can be considered as sections in $\mathcal D_{\geq 0}^1$
 and we can consider (\ref{eq:wf_uniform_decay}) as an inequality of sections which 
 will allow us to bound the integrals over $X$ using (\ref{eq:integral_inequ_pos}).
\end{remark}

We require an analogous condition for the singular spectrum. For this definition, 
we need additional notation. Choose a basis $\{X_1,\ldots,X_n\}$ of $\mathfrak{g}$, 
and for every multi-index $\alpha=(\alpha_1,\ldots,\alpha_n)\in \mathbb{N}^n$, define the 
differential operator 
\[
D^{\alpha}=\partial_{X_1}^{\alpha_1}\partial_{X_2}^{\alpha_2}\cdots \partial_{X_n}^{\alpha_n} 
\]
on $\mathfrak{g}$. If $0\in U_1\subset U_2\subset \mathfrak{g}$ are precompact, 
open sets with $\overline{U_1}\subset U_2$, then (see pages 25-26, 282 of \cite{Hor83}), 
we may find a sequence of functions $\{\varphi_{N,U_1,U_2}\}$ indexed by $N\in \mathbb{N}$ 
satisfying the following properties:
\begin{enumerate}
\item $\varphi_{N,U_1,U_2}\in C_c^{\infty}(U_2)$  
\item $\varphi_{N,U_1,U_2}(x)=1\ \text{if}\ x\in U_1$
\item There exist constants $C_{\alpha}>0$ for every multi-index $\alpha=(\alpha_1,\ldots,\alpha_n)$ 
such that
\[
|D^{\alpha+\beta}\varphi_{N,U_1,U_2}(x)|\leq C_{\alpha}^{|\beta|+1} (N+1)^{|\beta|} 
\]
for every $x\in U_1$ and every multi-index $\beta=(\beta_1,\ldots,\beta_n)$. 
Here $|\beta|=\beta_1+\cdots+\beta_n$.
\end{enumerate}

From now on, we fix a basis $\{X_1,\ldots,X_n\}$ of $\mathfrak{g}$ and for every 
pair of precompact, open sets $0\in U_1\subset U_2\subset \mathfrak{g}$ with 
$\overline{U_1}\subset U_2$, we fix a sequence of functions $\{\varphi_{N,U_1,U_2}\}$ 
satisfying the above properties.

\begin{definition}[Singular Spectrum Condition U]\label{def:ss_cond_U}
 With the notation from above we say that \emph{singular spectrum condition U} 
 is satisfied if
 for all $\eta\notin \op{Ind}_H^G\op{SS}(\tau)$ there is a neighborhood 
 $\Omega \subset i\mathfrak{g}^*\setminus \op{Ind}_H^G\op{SS}(\tau)$ of $\eta$ and a neighborhood
 $U_{\textup{CU,SS}}\subset \mathfrak g$ around $0 \in \mathfrak g$ such that for
 every pair of neighborhoods $0\in U_1\Subset U_2\subset \mathfrak{g}$ with $U_2\subset U_{\textup{CU,SS}}$ and all 
$N\in \mathbb{N}$, there is a constant $C_{U_1,U_2}$ such that
 \begin{equation}
 \label{eq:ss_uniform_decay}
 \begin{split}
 \left|\int\limits_{\mathfrak g_x} \langle \tau_x(e^Y) v_1,v_2\rangle_{\mathcal{V}_x}\cdot
 \left((\sigma_x(e^Y) z_1) \otimes \overline{z_2}\right) \cdot\varphi_{N,U_1,U_2}(Y) e^{\langle t\xi,Y\rangle}
 dY\right|_{\mathcal D^1_x} \\ 
 \leq C_{U_1,U_2}^{N+1}(N+1)^N (\|v_1\|_{\mathcal V_x} |z_1|_{\mathcal D^{1/2}_x})  \otimes(\|v_2\|_{\mathcal V_x}
 |z_2|_{\mathcal D^{1/2}_x} )t^{-N}
 \end{split}
\end{equation}
for $t>0$ and uniformly for all $\xi\in \Omega$, $x\in X$, $v_1,v_2\in \mathcal{V}_x$, and $z_1,z_2\in \mathcal{D}_x^{1/2}$. Additionally we 
demand that there is a constant $C$ such that for any $x\in X$, any 
$Y\in \mathfrak g_x\cap U_{\textup{CU,SS}}$ and any $z\in \mathcal D_x^{1/2}$, we 
have
\begin{equation}
 \label{eq:ss_uniform_sigma_bound}
 |\sigma_x(e^Y)z|_{\mathcal D_x^{1/2}}\leq C|z|_{\mathcal D_x^{1/2}}.
\end{equation}
\end{definition}

We then can show
\begin{theorem}\label{thm:WF_condU}
  Let $G$ be a Lie group, let $H\subset G$ be a closed subgroup, and let $\tau$ be a unitary
  $H$ representation. If wavefront condition U is satisfied, then we have
\[
\operatorname{WF}(\operatorname{Ind}_H^G\tau)\subset \operatorname{Ind}_H^G\operatorname{WF}(\tau).
\]
\end{theorem}

In addition, we can show

\begin{theorem}\label{thm:SS_condU}
  Let $G$ be a Lie group, let $H\subset G$ be a closed subgroup, and let $\tau$ be a unitary
  $H$ representation. If singular spectrum condition U is satisfied, then we have
\[
\operatorname{SS}(\operatorname{Ind}_H^G\tau)\subset \operatorname{Ind}_H^G\operatorname{SS}(\tau).
\]
\end{theorem}

We will spend the remainder of this section proving Theorem \ref{thm:WF_condU} 
and Theorem \ref{thm:SS_condU}.

Let $U_{\mathrm{inj}}\subset \mathfrak{g}$ be an open neigbourhood of zero such that the 
exponential map, 
\[
\exp\colon U_{\mathrm{inj}}\subset \mathfrak{g}\rightarrow \exp(U_{\mathrm{inj}})\subset G, 
\]
is a diffeomorphism and denote its inverse map by 
\[
\log:\exp(U_{\mathrm{inj}})\subset G\to U_{\mathrm{inj}}\subset 
\mathfrak{g}.
\]

For $f_1,f_2\in L^2(X,\mathcal{V}\otimes\mathcal{D}^{1/2})$ and $\varphi\in C^ {\infty}_c(G)$
with $\operatorname{supp} \varphi \subset \exp(U_{\mathrm{inj}})$, 
let us introduce the integral
\[
 I(f_1,f_2,\varphi,t\eta)=\int_G \left(\int_X \langle l_g f_1(g^{-1}\cdot x),f_2(x) \rangle_{\mathcal V_x\otimes\mathcal D^{1/2}_x}\right) \varphi(g)e^{\langle t\eta,\log(g)\rangle} dg .
\]
Here $dg$ is a non-zero, left invariant Haar measure. Note that 
\[
\langle l_g f_1(g^{-1}\cdot x),f_2(x) \rangle_{\mathcal V_x\otimes\mathcal{D}_x^{1/2}}
\]
takes values in $\mathcal D^1$ and can thus be integrated over $X$.

If $\eta_0 \notin \mathcal W =\operatorname{Ind}_H^G \operatorname{WF}(\tau)$, 
then in order to prove Theorem \ref{thm:WF_condU} we have to show 
that $\eta_0 \notin \operatorname{WF}(\operatorname{Ind}_H^G\tau)$. 
Thus we have to show that for the chosen $\eta_0$, there is a cutoff function 
$\varphi\in C^\infty_c(G)$ with $\varphi(e)\neq 0$ and a neighborhood 
$V_0$ of $\eta_0$ such that the integral has fast decay in $t$ uniformly in 
$\eta\in V_0$:
\begin{equation}\label{eq:wf_fast_decay}
 \left|I(f_1,f_2,\varphi,t\eta)\right|\leq C_{N,f_1,f_2,\varphi}|t|^{-N}. 
\end{equation}

On the other hand, in order to prove Theorem \ref{thm:SS_condU}, we have to show
that for a chosen $\eta_0\notin \mathcal S$ there is a neighborhood 
$V_0$ of $\eta_0$ and two precompact, open neighborhoods 
$0\in U_1\Subset U_2\subset U_{\mathrm{inj}}\subset \mathfrak{g}$ 
such that 
\begin{equation}\label{eq:ss_fast_decay}
 \left|I(f_1,f_2,\log^*\varphi_{N,U_1,U_2},t\eta)\right|\leq C_{f_1,f_2,\varphi}^{N+1}(N+1)^N|t|^{-N} 
\end{equation}
for every natural number $N\in \mathbb{N}$. 

Recall that we fixed  an inner product on $\mathfrak g$ and can
identify $\mathfrak g\cong i\mathfrak g^*$ by this inner product. 
With $\Omega$ from Definition~\ref{def:wf_cond_U} (or Definition~\ref{def:ss_cond_U}) 
there is $\varepsilon>0$ such that $B_{2\varepsilon}(\eta_0)$, the ball of 
radius $2\varepsilon$ around $\eta_0$, is contained in $\Omega$. We set 
\[
V_0:=B_{\varepsilon}(\eta_0).
\]

In order to formulate the precise conditions on the uniform cutoff function 
$\varphi$ which we will need in the proof of Theorem \ref{thm:WF_condU} 
and the conditions on the sets $0\in U_1\subset U_2\subset \mathfrak{g}$ 
which we will need in the proof of Theorem \ref{thm:SS_condU}, let us first 
formulate the following lemma. This lemma is a uniform non-stationary 
phase approximation and we will see 
below that together with wave front condition U and singular spectrum condition U,
it will be one of the main ingredients in 
order to obtain the desired estimates 
(\ref{eq:wf_fast_decay}, \ref{eq:ss_fast_decay}).

In what follows, whenever $L$ is a Lie group with Lie algebra $\mathfrak{l}$ 
and $\exp\colon \mathfrak{l}\rightarrow L$ is the exponential map, we will write 
$e^X=\exp(X)$ for $X\in \mathfrak{l}$.

\begin{lemma}\label{lem:unif_nonstat}
 Let $\varepsilon>0$. There is a neighborhood $U_{\textup{nsp}} \subset \mathfrak g$ of 
 $0\in \mathfrak g$ such that for any $\tilde\varphi \in C^\infty_c(U_{\textup{nsp}})$
 and for any $N$, there is a constant $C_{N,\tilde \varphi}$ such that
 \begin{equation}
  \label{eq:wf_unif_nonstat}
  \left|\int\limits_{\mathfrak g_x} e^{t(\langle\xi,Y\rangle-\langle\eta, \log(e^Ye^Z)\rangle)}\tilde\varphi(Y)dY\right|\leq C_{N,\tilde \varphi} \langle\xi\rangle^{-N}|t|^{-N}
 \end{equation}
uniformly for $x\in X$, for all 
$\xi\in i\mathfrak g_x^*\setminus B_{2\varepsilon}(q_x(\eta_0))$, $Z\in U_{\textup{nsp}}\cap\mathfrak g_x^\perp$, and $\eta\in V_0$. In addition, if $0\in \widetilde{U_1}\Subset \widetilde{U_2}\subset U_{\textup{nsp}}\subset \mathfrak{g}$, are open neighborhoods of zero, then there exists a constant
$C_{\widetilde{U_1},\widetilde{U_2}}$ such that for every $N\in \mathbb{N}$,
 \begin{equation}
  \label{eq:ss_unif_nonstat}
  \left|\int\limits_{\mathfrak g_x} e^{t(\langle\xi,Y\rangle-\langle\eta, \log(e^Ye^Z)\rangle)}\varphi_{N,\widetilde{U_1},\widetilde{U_2}}(Y)dY\right|\leq C_{\widetilde{U_1},\widetilde{U_2}}^{N+1}(N+1)^N \langle\xi\rangle^{-N}|t|^{-N}
 \end{equation}
uniformly for $x\in X$, for all $\xi\in i\mathfrak g_x^*\setminus B_{2\varepsilon}(q_x(\eta_0))$, $Z\in U_{\textup{nsp}}\cap\mathfrak g_x^\perp$, and $\eta\in V_0=B_\varepsilon(\eta_0)$.
\end{lemma}
\begin{proof}
For any ($\dim \mathfrak h$)-dimensinal linear subspace $V\subset \mathfrak g$
we define
\[
 \psi^V_\eta: \mathfrak g=V\oplus V^\perp \to i\mathbb R,~Y+Z\mapsto \langle\eta,\log(e^Ye^Z)\rangle.
\]
If 
\[
D_Y: C^\infty(V,i\mathbb R)\to C^{\infty}(V,iV^*)
\]
is the differential, then we clearly have
\[
 D_Y \psi^V_\eta(Y,0) = q_V(\eta)
\]
where $q_V:i\mathfrak g^*\to iV^*$ is the restriction map. As 
$ D_Y \psi^V_\eta(Y,Z)$ depends analytically on $Z$, for each $V$ there is an open 
neighborhood of zero $U_V\subset \mathfrak g$ such that 
\begin{equation}\label{eq:DPsi_estimate}
 \left\|D_Y \psi^V_\eta(Y,Z) - q_V(\eta)\right\|_{iV^*}<\frac{\varepsilon}{2}
\end{equation}
for all $Z\in U_V\cap V^\perp$. As the Grassmanian of all $\dim\mathfrak h$-dimensional 
subspaces $V$ is compact we can set
\[
 U_{\textup{nsp}}:=\bigcap\limits_{V\subset \mathfrak g} U_V
\]
which is an open neighborhood of zero. 

We now perform a partial integration via the differential operator 
\[
 L^x = t^{-1}\frac{\langle\xi -D_Y\psi^{\mathfrak g_x}_\eta(Y,Z),D_Y\rangle_{i\mathfrak g_x^*}}{\langle\xi -D_Y\psi^{\mathfrak g_x}_\eta(Y,Z),\xi -D_Y\psi^{\mathfrak g_x}_\eta(Y,Z) \rangle_{i\mathfrak g_x^*}}
\]
which has the property
\[
 L^x e^{t(\langle\xi,Y\rangle-\psi^{\mathfrak g_x}_\eta(Y,Z)\rangle)}= e^{t(\langle\xi,Y\rangle-\psi^{\mathfrak g_x}_\eta(Y,Z)\rangle)}.
\]
Using (\ref{eq:DPsi_estimate}) and the fact that 
$|\xi-q_x(\eta)|_{i\mathfrak g_x^*}>\varepsilon$ for all $\eta\in V_0$ and 
$\xi\in i\mathfrak g_x^*\setminus B_{2\varepsilon}(q_x(\eta_0))$, we calculate
\begin{eqnarray}
 \left\|\xi -D_Y\psi^{\mathfrak g_x}_\eta(Y,Z)\right\|_{i\mathfrak g_x^*}&\geq&
 \left\|\xi -q_x(\eta)\right\|_{i\mathfrak g_x^*}-\left\|q_x(\eta)-D_Y\psi^{\mathfrak g_x}_\eta(Y,Z)\right\|_{i\mathfrak g_x^*}\nonumber\\
  &\geq& \varepsilon - \frac{\varepsilon}{2} = \frac{\varepsilon}{2}.\label{eq:xi_DPsi_estimate}
\end{eqnarray}
So $L^x$ is well defined in the desired $\xi$ and $\eta$ range. Performing $N$-times
partial integration via $L^x$ on the integral on the left hand side of (\ref{eq:wf_unif_nonstat})
yields
 \begin{equation}\label{eq:wf_partint}
  \left|\int\limits_{\mathfrak g_x} e^{t(\langle\xi,Y\rangle-\langle\eta, \log(e^Ye^Z)\rangle)}\tilde\varphi(Y)dY\right|\leq C_{N,\tilde\varphi} \left\|\xi -D_Y\psi^{\mathfrak g_x}_\eta(Y,Z)\right\|_{i\mathfrak g_x^*}^{-N}|t|^{-N}.
 \end{equation}
Finally using the observation that there is a constant $c>0$ such that 
\[
 \left|\xi -q_x(\eta)\right|_{i\mathfrak g_x^*} \geq c\langle \xi\rangle
\]
uniform in $x\in X, \xi\in\mathfrak g_x^*\setminus B_{2\varepsilon}(q_x(\eta_0)),\eta\in V_0$ as well as 
(\ref{eq:xi_DPsi_estimate}) we obtain (\ref{eq:wf_unif_nonstat}). 
In order to obtain the stronger statement
(\ref{eq:ss_unif_nonstat}), we must keep track of the constants in the 
partial integration more carefully. Let us denote by $(L^x)^*$ the adjoint 
of the differential operator on $L^2(\mathfrak g_x)$. Utilizing statement (3) from the 
remarks preceding Definition \ref{def:ss_cond_U} the fact that $(L^x)^*$ 
is a first order differentiable operator as well as the fact, that $\psi_\eta^{\mathfrak g_x}$ 
is analytic, we obtain bounds
$$|((L^x)^*)^N\varphi_{N,\widetilde{U_1},\widetilde{U_2}}(y)|\leq C_{\widetilde{U_1},\widetilde{U_2}}^{N+1}(N+1)^N\left\|\xi -D_Y\psi^{\mathfrak g_x}_\eta(Y,Z)\right\|_{i\mathfrak g_x^*}^{-N}|t|^{-N}$$
for some constant $C_{\widetilde{U_1},\widetilde{U_2}}$ and all $y\in U_2$. Plugging this into (\ref{eq:wf_partint}) after replacing $\tilde\varphi$ with $\varphi_{N,\widetilde{U_1},\widetilde{U_2}}$ yields (\ref{eq:ss_unif_nonstat}).
\end{proof}

Now consider for any $(\dim \mathfrak h)$-dimensional subvectorspace $V\subset \mathfrak g$
the map
\[
 \kappa_V:\mathfrak g = V\oplus V^\perp \to G,~Y+Z\mapsto e^Ye^Z.
\]
Note that this map depends continuously on $V$ as a point in the Grassmanian and
as this Grassmanian is compact, there is a neighborhood of zero 
$U'_{\textup{inj}} \subset\mathfrak g$ such that $(\kappa_V)_{|U'_{\mathrm{inj}}}$
is a diffeomorphism for
all $V$. Now, in the proof of Theorem~\ref{thm:WF_condU} let $\tilde \varepsilon>0$ be such that 
\begin{equation}\label{eq:wf_tilde_eps}
B_{\tilde \varepsilon}(0) \Subset U_{\textup{nsp}}\cap U'_{\textup{inj}}\cap U_{\textup{CU,WF}}. 
\end{equation}
And in the proof of Theorem~\ref{thm:SS_condU} let $\tilde \varepsilon>0$ be such that 
\begin{equation}\label{eq:ss_tilde_eps}
B_{\tilde \varepsilon}(0) \Subset U_{\textup{nsp}}\cap U'_{\textup{inj}}\cap U_{\textup{CU,SS}}. 
\end{equation}
Again from the compactness of the Grassmanian we conclude that 
$\bigcap\limits_{V} \kappa_V(B_{\tilde \varepsilon}(0))\subset G$ is a nonempty 
open neighborhood of $e\in G$. With this observation we can find a cutoff function
$\varphi \in C_c^\infty(G)$ such that $\varphi(e)\neq0$ and
\begin{equation}\label{eq:supp_phi}
 \operatorname{supp} \varphi \subset \left(\bigcap\limits_{V} \kappa_V(B_{\tilde \varepsilon}(0))\right)\subset G
\end{equation}
to be used in the proof of Theorem~\ref{thm:WF_condU}. Analogously, we may choose
$0\in U_1\Subset U_2\subset \mathfrak{g}$ precompact, open subsets such that
\begin{equation}\label{eq:supp_U_2}
 \exp(U_2) \subset \left(\bigcap\limits_{V} \kappa_V(B_{\tilde \varepsilon}(0))\right)\subset G
\end{equation}
to be used in the proof of Theorem~\ref{thm:SS_condU}.

Let us now come back to the oscillatory integral $I(f_1,f_2,\varphi,t\eta)$: 
Interchanging the $X$ and $G$ integrals and replacing the integral over $G$ 
for each $x\in X$ by an integral over $\mathfrak g_x \oplus \mathfrak g_x^\perp$
via the diffeomorphism $\kappa_{\mathfrak g_x}$ we obtain 
\begin{align*}
&\left|I(f_1,f_2,\varphi,t\eta)\right|=\left|\int\limits_X \int\limits_G \langle l_g f_1(g^{-1}\cdot x),f_2(x)\rangle_{\mathcal V_x} \varphi(g) e^{\langle t\eta,\log(g)\rangle} dg\right|\\
&\leq\int\limits_X \left|\int\limits_G \langle l_g f_1(g^{-1}\cdot x),f_2(x)\rangle_{\mathcal V_x} \varphi(g) e^{\langle t\eta,\log(g)\rangle} dg\right|_{\mathcal D^{1}_x}\\
&\leq\int\limits_X \int\limits_{\mathfrak g_x^\perp}\Bigg|\int\limits_{\mathfrak g_x} 
\langle l_{e^Ye^Z}f_1((e^Ye^Z)^{-1}\cdot x),f_2(x) \rangle_{\mathcal V_x} \varphi(e^Ye^Z) 
e^{\langle t\eta,\log(e^Ye^Z)\rangle}\\
&\quad\quad\quad\quad\quad\quad\quad\quad\quad\quad\quad\quad\quad\quad\quad\quad\quad\quad j_{\mathfrak{g}_x,\mathfrak{g}_x^{\perp}}(Y+Z)dY\Bigg|_{\mathcal D^{1}_x}dZ\\
&= \int\limits_X \int\limits_{\mathfrak g_x^\perp}\Bigg|\int\limits_{\mathfrak g_x} 
\langle \tau_x\otimes\sigma_x (e^Y)l_{e^Z}f_1((e^Z)^{-1}\cdot x),f_2(x) \rangle_{\mathcal V_x} \varphi(e^Ye^Z)e^{\langle t\eta,\log(e^Ye^Z)\rangle} \\
&\quad\quad\quad\quad\quad\quad\quad\quad\quad\quad\quad\quad\quad\quad\quad\quad\quad\quad  j_{\mathfrak{g}_x,\mathfrak{g}_x^{\perp}}(Y+Z)dY\Bigg|_{\mathcal D^1_x}dZ
\end{align*}
where $j_{\mathfrak{g}_x,\mathfrak{g}_x^{\perp}}$ is the Jacobian of the diffeomorphism 
$\kappa_{\mathfrak g_x}$ and $dY$, $dZ$ are Lebesgue measures on $\mathfrak{g}_x$ 
and $\mathfrak{g}_x^\perp$ which are fixed by the choice of the scalar product on $\mathfrak g$. 

The analogous statement for the singular spectrum case is 
\begin{align*}
&\left|I(f_1,f_2,\log^*\varphi_{N,U_1,U_2},t\eta)\right|\\
&\leq \int\limits_X \int\limits_{\mathfrak g_x^\perp}\Bigg|\int\limits_{\mathfrak g_x} 
\langle \tau_x\otimes\sigma_x (e^Y)l_{e^Z}f_1((e^Z)^{-1}\cdot x),f_2(x) \rangle_{\mathcal V_x} \log^*\varphi_{N,U_1,U_2}(e^Ye^Z)e^{\langle t\eta,\log(e^Ye^Z)\rangle} \\
&\quad\quad\quad\quad\quad\quad\quad\quad\quad\quad\quad\quad\quad\quad\quad\quad\quad\quad  j_{\mathfrak{g}_x,\mathfrak{g}_x^{\perp}}(Y+Z)dY\Bigg|_{\mathcal D^1_x}dZ.
\end{align*}

Next for fixed $x\in X$ and $Z\in \mathfrak{g}_x^\perp$, we write 
$l_{e^Z}f_1(e^{-Z}\cdot x)=v_1\otimes z_1$ and $f_2(x)=v_2\otimes z_2$ with $v_{i}\in\mathcal V_x$
and $z_{i}\in\mathcal D^{1/2}_x$. We can then write the inner integral over $\mathfrak{g}_x$ as
\[
 \left\langle \langle \tau_x (e^Y)v_1,v_2 \rangle_{\mathcal V_x} \cdot (\sigma_x(e^Y)z_1)\otimes \overline z_2\cdot\varphi(e^Ye^Z)j_{\mathfrak{g}_x,\mathfrak{g}_x^{\perp}}(Y+Z),\overline{ e^{\langle t\eta,\log(e^Ye^Z)\rangle}}\right\rangle_{L^2(\mathfrak{g}_x)}
\]
which takes values in $\mathcal D^1_x$ (For the singular spectrum case, we simply 
replace $\varphi$ by $\log^*\varphi_{N,U_1,U_2}$).
For the proof of Theorem~\ref{thm:WF_condU}, we choose a second cutoff function 
$\tilde \varphi\in C^\infty_c(\mathfrak g)$ 
which fulfills
\begin{eqnarray*}
 \operatorname{supp}(\tilde\varphi)&\subset& U_{\textup{nsp}}\cap U_{\textup{CU,WF}}\cap U_{\textup{inj}}\\
 \tilde \varphi (Y) =1 &\textup{for}& Y \in B_{\tilde \varepsilon}(0)
\end{eqnarray*}
which is possible because of the choice of $\tilde \varepsilon$ according to 
(\ref{eq:wf_tilde_eps}). For the proof of Theorem~\ref{thm:SS_condU}, 
we choose two open sets
\begin{eqnarray*}
 \widetilde{U_2}:= & U_{\textup{nsp}}\cap U_{\textup{CU,SS}}\cap U_{\textup{inj}}\\
 \widetilde{U_1}:= & B_{\tilde \varepsilon}(0)
\end{eqnarray*}
which is possible because of the choice of $\tilde \varepsilon$ according to 
(\ref{eq:ss_tilde_eps}). In the proof of Theorem~\ref{thm:SS_condU}, 
we will utilize the second cutoff function 
$\varphi_{N,\widetilde{U_1},\widetilde{U_2}}$ instead of $\widetilde{\varphi}$.
From (\ref{eq:supp_phi}), we deduce that for all $x\in X$ and all 
$Y\notin B_{\tilde \varepsilon}(0) \cap \mathfrak g_x$ we have $\varphi(e^Ye^Z)=0$
for all $Z\in\mathfrak g_x^\perp$. Thus, we can insert this new, real valued, cutoff function 
$\tilde \varphi$ into our scalar product and calculate
\begin{align*}
 &\left|\left\langle \langle \tau_x (e^Y)v_1,v_2 \rangle_{\mathcal V_x} \cdot (\sigma_x(e^Y)z_1)\otimes \overline z_2\cdot\varphi(e^Ye^Z)j_{\mathfrak{g}_x,\mathfrak{g}_x^{\perp}}(Y+Z),\overline{e^{\langle t\eta,\log(e^Ye^Z)\rangle}\tilde\varphi(Y)}\right\rangle_{L^2(\mathfrak{g}_x)}\right|_{\mathcal D^1_x}\\
 &= \Bigg|\Bigg\langle \mathcal{F}\left[\langle \tau_x (e^\bullet)v_1,v_2 \rangle_{\mathcal V_x} \cdot (\sigma_x(e^\bullet)z_1)\otimes \overline z_2\cdot\varphi(e^\bullet e^Z)j_{\mathfrak{g}_x,\mathfrak{g}_x^{\perp}}(\bullet+Z)\right](\xi), \\
 &\quad\quad\quad\quad \mathcal{F}\left[\overline{e^{\langle t\eta,\log(e^\bullet e^Z)\rangle}\tilde\varphi(\bullet)}\right](\xi)\Bigg\rangle_{L^2(i\mathfrak{g}^*_x)}\Bigg|_{\mathcal D^1_x}\\
  &= |t|^{\dim \mathfrak h} \Bigg|\Bigg\langle \mathcal{F}\left[\langle \tau_x (e^\bullet)v_1,v_2 \rangle_{\mathcal V_x} \cdot (\sigma_x(e^\bullet)z_1)\otimes \overline z_2\cdot\varphi(e^\bullet e^Z)j_{\mathfrak{g}_x,\mathfrak{g}_x^{\perp}}(\bullet+Z)\right](t \xi), \\
  &\quad\quad\quad\quad \mathcal{F}\left[e^{-\langle t\eta,\log(e^\bullet e^Z)\rangle\tilde\varphi(\bullet)}\right](t \xi)\Bigg\rangle_{L^2(i\mathfrak{g}^*_x)}\Bigg|_{\mathcal D^1_x}\\
\end{align*}

For the singular spectrum case, the same expression holds with $\varphi$ replaced by 
$\log^*\varphi_{N,U_1,U_2}$ and $\widetilde{\varphi}$ replaced by 
$\varphi_{N,\widetilde{U_1},\widetilde{U_2}}$.
We will now split the integral over $i\mathfrak{g}_x^*$ into two parts, the first 
one over $B_{2\varepsilon}(q_x(\eta_0))\subset i\mathfrak g^*_x$ and the second one over 
$i\mathfrak g_x^*\setminus B_{2\varepsilon}(q_x(\eta_0))$. The idea is to get the 
fast decay in $t$ for the first integral by the wave front condition U 
(or, in the singular spectrum case, by singular spectrum condition U), and for 
the second part by the uniform nonstationary phase estimate (Lemma~\ref{lem:unif_nonstat}). 
Note that a similar strategy is used in the proof of \cite[Proposition 1.3.2]{Du73}

{\bf First part $\int_{B_{2\varepsilon}(q_x(\eta_0))}$:}\\
On the one side we have the trivial bound 
\begin{equation}
 \label{eq:wf_case1_trivial_bound}
 \left|\mathcal{F}\left[e^{-\langle t\eta,\log(e^\bullet e^Z)\rangle}\tilde\varphi(\bullet)\right](t \xi)\right|\leq \int\limits_{\mathfrak g_x} |\tilde\varphi (Y)|dY =:C_1.
\end{equation}

In the singular spectrum case, we utilize condition (3) satisfied by
$\varphi_{N,\widetilde{U_1},\widetilde{U_2}}$ (see the remarks preceding
Definition~\ref{def:ss_cond_U}) with $\alpha=\emptyset$ and $\beta=\emptyset$ 
to obtain the trivial bound
\begin{equation}
 \label{eq:ss_case1_trivial_bound}
\left|\mathcal{F}\left[e^{-\langle t\eta,\log(e^\bullet e^Z)\rangle}\varphi_{N,\widetilde{U_1},\widetilde{U_2}}(\bullet)\right](t \xi)\right|\leq \int\limits_{\mathfrak g_x} |\varphi_{N,\widetilde{U_1},\widetilde{U_2}} (Y)|dY \leq \op{vol}(\widetilde{U_2})C_{\emptyset}=:C_1.
\end{equation}

On the other hand let us consider
\begin{align*}
 & \left|\mathcal{F}\left[\langle \tau_x (e^\bullet)v_1,v_2 \rangle_{\mathcal V_x} \cdot (\sigma_x(e^\bullet)z_1)\otimes \overline z_2\cdot \varphi(e^\bullet e^Z)j_{\mathfrak{g}_x,\mathfrak{g}_x^{\perp}}(\bullet+Z)\right](t \xi)\right|_{\mathcal D^1_x}=\\
 &\left| \mathcal{F}\left[\langle \tau_x (e^\bullet)v_1,v_2 \rangle_{\mathcal V_x} \cdot (\sigma_x(e^\bullet)z_1)\otimes \overline z_2\cdot \tilde\varphi(\bullet)\varphi(e^\bullet e^Z)j_{\mathfrak{g}_x,\mathfrak{g}_x^{\perp}}(\bullet+Z)\right](t \xi)\right|_{\mathcal D^1_x}=\\
&\left| \left(\mathcal{F}\left[\langle \tau_x (e^\bullet)v_1,v_2 \rangle_{\mathcal V_x} \cdot (\sigma_x(e^\bullet)z_1)\otimes \overline z_2\cdot \tilde\varphi(\bullet)\right]*\mathcal F\left[\varphi(e^\bullet e^Z)j_{\mathfrak{g}_x,\mathfrak{g}_x^{\perp}}(\bullet+Z)\right]\right)(t \xi)\right|_{\mathcal D^1_x}
\end{align*}
Now wave front condition U implies 
\begin{equation}
\label{eq:wf_case1_unif_decay_a}
\begin{split}
 \left|\mathcal{F}\left[\langle \tau_x (e^\bullet)v_1,v_2 \rangle_{\mathcal V_x} \cdot (\sigma_x(e^\bullet)z_1)\otimes \overline z_2\cdot \tilde\varphi(\bullet)\right](t\xi) \right|_{\mathcal D^1_x}\leq\\
 C_{N,\tilde \varphi} (\|v_1\|_{\mathcal V_x} |z_1|_{\mathcal D^{1/2}_x})  \otimes(\|v_2\|_{\mathcal V_x}
 |z_2|_{\mathcal D^{1/2}_x} )t^{-N}
 \end{split}
\end{equation}
uniformly in $\xi\in B_{2\varepsilon}(q_x(\eta_0)) \subset q_x(\Omega)$, $x\in X$,
$v_1,v_2\in\mathcal V_x$, $z_1,z_2\in\mathcal D^{1/2}_x$. 
The analogous singular spectrum statement is
\begin{equation}
\label{eq:ss_case1_unif_decay_a}
\begin{split}
 \left|\mathcal{F}\left[\langle \tau_x (e^\bullet)v_1,v_2 \rangle_{\mathcal V_x} \cdot (\sigma_x(e^\bullet)z_1)\otimes \overline z_2\cdot \varphi_{N,\widetilde{U_1},\widetilde{U_2}}(\bullet)\right](t\xi) \right|_{\mathcal D^1_x}\leq\\
 C_{\widetilde{U_1},\widetilde{U_2}}^{N+1}(N+1)^N (\|v_1\|_{\mathcal V_x} |z_1|_{\mathcal D^{1/2}_x})  \otimes(\|v_2\|_{\mathcal V_x}
 |z_2|_{\mathcal D^{1/2}_x} )t^{-N}
 \end{split}
\end{equation}
uniformly in $\xi\in B_{2\varepsilon}(q_x(\eta_0)) \subset q_x(\Omega)$, $x\in X$,
$v_1,v_2\in\mathcal V_x$, $z_1,z_2\in\mathcal D^{1/2}_x$. Note that the value of $C_{N,\tilde\varphi}$
and $C_{\widetilde{U_1},\widetilde{U_2}}$ might change from line to line below, however the crucial 
uniformity properties will still be satisfied. 

Next consider the family of functions 
$\rho_{x,Z}(Y):= \varphi(e^Y e^Z)j_{\mathfrak{g}_x,\mathfrak{g}_x^{\perp}}(Y+Z)$ 
which can also be considered as a function of $Y\in \mathfrak g$. As the 
parameter $Z\in B_{\tilde\varepsilon}(0)$ 
can vary only in a bounded set due to the choice of $\varphi$ according to 
(\ref{eq:wf_tilde_eps}) and (\ref{eq:supp_phi}) and as the dependence of $\rho_{x,Z}$
on $x$ is only via the subvectorspace $\mathfrak g_x\subset\mathfrak g$ which can 
be considered as a point in the compact Grassmanian, $\rho_{x,Z}$ varies in a bounded 
set of $C^N(\mathfrak g)$ for all $N\geq 0$. Accordingly, partial integration 
yields uniform fast decay of its Fourier transform in all directions, i.e.
\begin{equation}
\label{eq:wf_case1_unif_decay_b}
 \mathcal F\left[\varphi(e^\bullet e^Z)j_{\mathfrak{g}_x,\mathfrak{g}_x^{\perp}}(\bullet+Z)\right](\xi) \leq C_{N,\varphi}\langle \xi\rangle^{-N}
\end{equation}
for all $\xi\in i\mathfrak g_x^*$ uniformly in $x\in X$ and 
$Z\in B_{\tilde\varepsilon}(0)\cap \mathfrak g_x^ \perp$. 
The analogous statement for the singular spectrum case can be deduced from Lemma~7.2
of \cite{HHO16}. We have
\begin{equation}
\label{eq:ss_case1_unif_decay_b}
 \mathcal F\left[\varphi_{N,U_1,U_2}(e^\bullet e^Z)j_{\mathfrak{g}_x,\mathfrak{g}_x^{\perp}}(\bullet+Z)\right](\xi) \leq C_{U_1,U_2}^{N+1}(N+1)^N\langle \xi\rangle^{-N}
\end{equation}
for all $\xi\in i\mathfrak g_x^*$ uniformly in $x\in X$ and 
$Z\in B_{\tilde\varepsilon}(0)\cap \mathfrak g_x^ \perp$.

Now a straightforward calculation shows that the uniform decay property
(\ref{eq:wf_case1_unif_decay_a}) is not destroyed by convolution with a function 
satisfying (\ref{eq:wf_case1_unif_decay_b}). So we obtain
\begin{equation}
 \label{eq:wf_case1_unif_decay_c}
 \begin{split}
 \left| \left(\mathcal{F}\left[\langle  \tau_x (e^\bullet)v_1,v_2 \rangle_{\mathcal V_x} 
 \cdot (\sigma_x(e^\bullet)z_1)\otimes \overline z_2\cdot\tilde\varphi(\bullet)\right]*\mathcal
 F\left[\rho_{x,Z}(\bullet)\right]\right)(t \xi)\right|_{\mathcal D^1_x} \\
 \leq C_{N,\tilde \varphi} (\|v_1\|_{\mathcal V_x} |z_1|_{\mathcal D^{1/2}_x})  \otimes(\|v_2\|_{\mathcal V_x}
 |z_2|_{\mathcal D^{1/2}_x} )t^{-N}
 \end{split}
 \end{equation}
uniformly in $\xi\in B_{2\varepsilon}(q_x(\eta_0)) \subset q_x(\Omega)$. 

For the singular spectrum case, we set 
$\rho_{x,Z}:=\varphi_{N,U_1,U_2}(e^Y e^Z)j_{\mathfrak{g}_x,\mathfrak{g}_x^{\perp}}(Y+Z)$. 
Then, the uniform decay property (\ref{eq:ss_case1_unif_decay_a}) 
is not destroyed by convolution with a function satisfying (\ref{eq:ss_case1_unif_decay_b}), 
and we obtain
 \begin{equation}
 \label{eq:ss_case1_unif_decay_c}
 \begin{split}
 \left| \left(\mathcal{F}\left[\langle  \tau_x (e^\bullet)v_1,v_2 \rangle_{\mathcal V_x} 
 \cdot (\sigma_x(e^\bullet)z_1)\otimes \overline z_2 \cdot\varphi_{N,\widetilde{U_1},\widetilde{U_2}}
 (\bullet)\right]*\mathcal F\left[\rho_{x,Z}(\bullet)\right]\right)(t \xi)\right|_{\mathcal D^1_x} \\
 \leq C_{U_1,U_2,\widetilde{U_1},\widetilde{U_2}}^{N+1}(N+1)^N (\|v_1\|_{\mathcal V_x} 
 |z_1|_{\mathcal D^{1/2}_x})  \otimes(\|v_2\|_{\mathcal V_x}
 |z_2|_{\mathcal D^{1/2}_x} )t^{-N}.
 \end{split}
 \end{equation}

 Now putting (\ref{eq:wf_case1_trivial_bound}) and (\ref{eq:wf_case1_unif_decay_c}) together we obtain
\begin{equation}\label{eq:wf_fast_decay_case1}
\begin{split}
  &\Bigg|\int\limits_{B_{2\varepsilon}(q_x(\eta_0))} \overline{\mathcal{F}\left[e^{-\langle t\eta,\log(e^\bullet e^Z)\rangle}\tilde\varphi(\bullet)\right](t \xi)}\\
  &\quad\quad\quad\quad\quad \mathcal{F}\left[\langle \tau_x (e^\bullet)v_1,v_2 \rangle_{\mathcal V_x} \cdot (\sigma_x(e^\bullet)z_1)\otimes \overline z_2\cdot\rho_{x,Z}(\bullet)\right](t \xi)d\xi \Bigg|_{\mathcal D^1_x} \\
  & \leq C_{N,\varphi} (\|v_1\|_{\mathcal V_x} |z_1|_{\mathcal D^{1/2}_x})  \otimes(\|v_2\|_{\mathcal V_x}
 |z_2|_{\mathcal D^{1/2}_x} )t^{-N}
\end{split}
\end{equation}
uniformly in $x\in X$ and $Z\in B_{\tilde \varepsilon}(0)$.

To get the singular spectrum analogue, we combine (\ref{eq:ss_case1_trivial_bound}) and (\ref{eq:ss_case1_unif_decay_c}) to obtain
\begin{equation}\label{eq:ss_fast_decay_case1}
\begin{split}
  &\Bigg|\int\limits_{B_{2\varepsilon}(q_x(\eta_0))} \overline{\mathcal{F}\left[e^{-\langle t\eta,\log(e^\bullet e^Z)\rangle}\varphi_{N,\widetilde{U_1},\widetilde{U_2}}(\bullet)\right](t \xi) }\\
  &\quad\quad\quad\quad\quad \mathcal{F}\left[\langle \tau_x (e^\bullet)v_1,v_2 \rangle_{\mathcal V_x} \cdot (\sigma_x(e^\bullet)z_1)\otimes \overline z_2\cdot\rho_{x,Z}(\bullet)\right](t \xi)d\xi \Bigg|_{\mathcal D^1_x} \\
  & \leq C_{U_1,U_2,\widetilde{U_1},\widetilde{U_2}}^{N+1}(N+1)^N (\|v_1\|_{\mathcal V_x} |z_1|_{\mathcal D^{1/2}_x})  \otimes(\|v_2\|_{\mathcal V_x}
 |z_2|_{\mathcal D^{1/2}_x} )t^{-N}
\end{split}
\end{equation}
uniformly in $x\in X$ and $Z\in B_{\tilde \varepsilon}(0)\cap\mathfrak{g}_x^\perp$.

{\bf Second part $\int_{i\mathfrak g_x^*\setminus B_{2\varepsilon}(q_x(\eta_0))}$:}\\
Both in the wavefront and singular spectrum cases, for 
$\xi \in i\mathfrak g_x^*\setminus B_{2\varepsilon}(q_x(0))$, Lemma~\ref{lem:unif_nonstat} 
provides us with fast decay in $t$ and $\xi$ of the second factor in the considered
scalar product. In order to obtain fast decay of the scalar product we thus need a general 
bound for the first factor
\[
 \mathcal{F}\left[\langle \tau_x (e^\bullet)v_1,v_2 \rangle_{\mathcal V_x} \cdot (\sigma_x(e^\bullet)z_1)\otimes \overline z_2\cdot \varphi(e^\bullet e^Z)j_{\mathfrak{g}_x,\mathfrak{g}_x^{\perp}}(\bullet+Z)\right](t \xi)
\]
in the wavefront case and
\[
 \mathcal{F}\left[\langle \tau_x (e^\bullet)v_1,v_2 \rangle_{\mathcal V_x} \cdot (\sigma_x(e^\bullet)z_1)\otimes \overline z_2\cdot \varphi_{N,U_1,U_2}(e^\bullet e^Z)j_{\mathfrak{g}_x,\mathfrak{g}_x^{\perp}}(\bullet+Z)\right](t \xi)
\]
in the singular spectrum case.
First note that (\ref{eq:wf_uniform_sigma_bound}) from wave front condition U 
and (\ref{eq:ss_uniform_sigma_bound}) from singular spectrum condition U together with 
the unitarity of $\tau_x$ imply
a general bound on the norm
\[
 \|\tau_x(e^Y)v_1\|_{\mathcal V_x} \cdot| \sigma_x(e^Y)z_1|_{\mathcal D^{1/2}_x } 
 \leq C_{\tau,\sigma} \|v_1\|_{\mathcal V_x}|z_1|_{\mathcal D^{1/2}_x}
\]
uniformly in $x\in X$ and $Y\in B_{\tilde\varepsilon}(0)\cap \mathfrak g_x$. 
Again arguing with the compactness of the Grassmanian, we obtain
\[
 \int \limits_{\mathfrak g_x} \left|\varphi(e^Ye^Z) 
 j_{\mathfrak{g}_x,\mathfrak{g}_x^{\perp}} (Y+Z)\right|dY \leq C_2
\]
uniformly in $x\in X$ and $Z\in B_{\tilde\varepsilon}(0)\cap \mathfrak g_x^ \perp$ in the wavefront case and
\[
 \int \limits_{\mathfrak g_x} \left|\varphi_{N,U_1,U_2}(e^Ye^Z) j_{\mathfrak{g}_x,\mathfrak{g}_x^{\perp}} (Y+Z)\right|dY \leq C_2
\]
uniformly in $x\in X$ and $Z\in B_{\tilde\varepsilon}(0)\cap \mathfrak g_x^ \perp$ in the singular spectrum case. We also 
utilized statement (3) of the remarks before Definition~\ref{def:ss_cond_U} 
in the singular spectrum case.
Consequently, in the wavefront case we obtain
\begin{align*}
& \left| \mathcal{F}\left[\langle \tau_x (e^\bullet)v_1,v_2 \rangle_{\mathcal V_x} \cdot (\sigma_x(e^\bullet)z_1)\otimes \overline z_2\cdot \varphi(e^\bullet e^Z)j_{\mathfrak{g}_x,\mathfrak{g}_x^{\perp}}(\bullet+Z)\right](t \xi) \right| \\
&\leq \int \limits_{\mathfrak g_x}  \|\tau_x(e^Y)v_1\|_{\mathcal V_x}\|v_2\|_{\mathcal V_x} \cdot| \sigma_x(e^Y)z_1|_{\mathcal D^{1/2}_x }\otimes|z_2|_{\mathcal D^{1/2}_x} \left|\varphi(e^Ye^Z) j_{\mathfrak{g}_x,\mathfrak{g}_x^{\perp}} (Y+Z)\right|dY\\
&\leq C_{\tau,\sigma}C_2 \|v_1\|_{\mathcal V_x} \|v_2\|_{\mathcal V_x}|z_1|_{\mathcal D^{1/2}_x}\otimes |z_2|_{\mathcal D^{1/2}_x}
\end{align*}
uniformly in $t>0, \xi \in i\mathfrak g_x^*, x\in X$, and $T\in B_{\tilde\varepsilon}(0)\cap \mathfrak g_x^ \perp$ 
and in the singular spectrum case, we obtain
\begin{align*}
& \left| \mathcal{F}\left[\langle \tau_x (e^\bullet)v_1,v_2 \rangle_{\mathcal V_x} \cdot (\sigma_x(e^\bullet)z_1)\otimes \overline z_2\cdot \varphi_{N,U_1,U_2}(e^\bullet e^Z)j_{\mathfrak{g}_x,\mathfrak{g}_x^{\perp}}(\bullet+Z)\right](t \xi) \right| \\
&\leq C_{\tau,\sigma}C_2 \|v_1\|_{\mathcal V_x} \|v_2\|_{\mathcal V_x}|z_1|_{\mathcal D^{1/2}_x}\otimes |z_2|_{\mathcal D^{1/2}_x}
\end{align*}
uniformly in $t>0, \xi \in i\mathfrak g_x^*, x\in X$, and $z\in B_{\tilde\varepsilon}(0)$.
Thus, in the wavefront case we get
\begin{equation} \label{eq:wf_fast_decay_case2}
\begin{split}
 &\Bigg|\int\limits_{i\mathfrak g_x^*\setminus B_{2\varepsilon}(q_x(\eta_0))} \overline{\mathcal{F}\left[e^{-\langle t\eta,\log(e^\bullet e^Z)\rangle}\tilde\varphi(\bullet)\right](t \xi)}\\
  &\quad\quad\quad\quad\quad \mathcal{F}\left[\langle \tau_x (e^\bullet)v_1,v_2 \rangle_{\mathcal V_x} \cdot (\sigma_x(e^\bullet)z_1)\otimes \overline z_2\cdot\rho_{x,Z}(Y)\right](t \xi)d\xi \Bigg|_{\mathcal D^1_x} \\
  & \leq C_{N,\varphi} (\|v_1\|_{\mathcal V_x} |z_1|_{\mathcal D^{1/2}_x})  \otimes(\|v_2\|_{\mathcal V_x}
 |z_2|_{\mathcal D^{1/2}_x} )t^{-N}
\end{split}
 \end{equation}
and in the singular spectrum case, we get
\begin{equation} \label{eq:ss_fast_decay_case2}
\begin{split}
 &\Bigg|\int\limits_{i\mathfrak g_x^*\setminus B_{2\varepsilon}(q_x(\eta_0))} \overline{\mathcal{F}\left[e^{-\langle t\eta,\log(e^\bullet e^Z)\rangle}\varphi_{N,\widetilde{U_1},\widetilde{U_2}}(\bullet)\right](t \xi)}\\
  &\quad\quad\quad\quad\quad \mathcal{F}\left[\langle \tau_x (e^\bullet)v_1,v_2 \rangle_{\mathcal V_x} \cdot (\sigma_x(e^\bullet)z_1)\otimes \overline z_2\cdot\rho_{x,Z}(Y)\right](t \xi)d\xi \Bigg|_{\mathcal D^1_x} \\
  & \leq C_{\widetilde{U_1},\widetilde{U_2}}^{N+1}(N+1)^N (\|v_1\|_{\mathcal V_x} |z_1|_{\mathcal D^{1/2}_x})  \otimes(\|v_2\|_{\mathcal V_x}
 |z_2|_{\mathcal D^{1/2}_x} )t^{-N}.
\end{split}
 \end{equation}
{\bf End of the proof of Theorem \ref{thm:WF_condU} and Theorem \ref{thm:SS_condU}}\\
To finish the proof of Theorem~\ref{thm:WF_condU}, we put (\ref{eq:wf_fast_decay_case1}) 
and (\ref{eq:wf_fast_decay_case2}) together to obtain 
\begin{eqnarray*}
 \left|I(f_1,f_2,\varphi,t\eta)\right|&\leq& \int\limits_{X} \int \limits_{g_x^\perp\cap B_{\tilde\varepsilon}(0)} C_{N,\varphi} (\|v_1\|_{\mathcal V_x} |z_1|_{\mathcal D^{1/2}_x})  \otimes(\|v_2\|_{\mathcal V_x}
 |z_2|_{\mathcal D^{1/2}_x} )t^{-N} dZ\\
 &=& \int\limits_{X} \int \limits_{\mathfrak g_x^\perp\cap B_{\tilde\varepsilon}(0)} C_{N,\varphi}  \|l_{e^Z}f_1(e^{-Z}x)\|_{\mathcal V_x\otimes D^{1/2}_x} \otimes \|f_2(x)\|_{\mathcal V_x\otimes\mathcal D^{1/2}_x} t^{-N}dZ
\end{eqnarray*}
uniformly in $\eta\in V_0$. Note that the compact support of the $Z$ integral 
follows from the choice of the cutoff 
function $\varphi$ according to (\ref{eq:supp_phi}). As the regular representation 
on $L^2(X,\mathcal V\otimes \mathcal D^{1/2})$ is 
unitary, $f_1\in L^2(X,\mathcal V\otimes \mathcal D^{1/2})$ implies that 
$\|l_{e^Z}f_1(e^{-Z}x)\|_{\mathcal V_x} \in L^2(X,\mathcal D^{1/2}_{\geq 0})$ and so we 
finally obtain
\[
 \left|I(f_1,f_2,\varphi,t\eta)\right|\leq C_{N,f_1,f_2,\varphi} |t|^{-N}
\]
uniformly in $\eta \in V_0$ which finishes the proof of Theorem~\ref{thm:WF_condU}.

To finish the proof of Theorem \ref{thm:SS_condU}, we put (\ref{eq:ss_fast_decay_case1}) and (\ref{eq:ss_fast_decay_case2}) together and make the analogous argument to obtain 
\[
 \left|I(f_1,f_2,\varphi_{N,U_1,U_2},t\eta)\right|\leq C_{U_1,U_2,f_1,f_2}^{N+1}(N+1)^N |t|^{-N}
\]
uniformly in $\eta \in V_0$.

\section{Proof of condition U for compact $X$}
\label{sec:CU_compact}

In this section, we verify the following Proposition.

\begin{proposition} \label{prop:compact} Suppose $G$ is a Lie group, $H\subset G$ 
is a closed subgroup, $(\tau,V)$ is a unitary representation of $H$, and $X=G/H$ 
is compact. Then wavefront condition U (Definition~\ref{def:wf_cond_U}) and singular 
spectrum condition U (Definition~\ref{def:ss_cond_U}) are both satisfied for the 
triple $(G,H,\tau)$.
\end{proposition}

When combined with Theorem~\ref{thm:WF_condU}, Theorem~\ref{thm:SS_condU}, and 
Theorem~1.1 of \cite{HHO16}, this will imply Theorem~\ref{thm:compact}.

First, we show (\ref{eq:wf_uniform_sigma_bound}) and the identical statement (\ref{eq:ss_uniform_sigma_bound}). To do this, we show that for any precompact, 
open subset $U\subset \mathfrak{g}$, there exists a 
constant $C$ such that for every $x\in X$, every 
$Y\in \mathfrak{g}_x\cap U$, and any $z\in \mathcal{D}_x^{1/2}$ 
\[
|\sigma_x(e^Y)z|_{\mathcal{D}_x^{1/2}}\leq C|z|_{\mathcal{D}_x^{1/2}}.
\]
For fixed $x\in X$ and $Y\in \mathfrak{g}_x$, the quotient
\[
\frac{|\sigma_x(e^Y)z|_{\mathcal{D}_x^{1/2}}}{|z|_{\mathcal{D}_x^{1/2}}} 
\]
is independent of $z\in \mathcal{D}_x^{1/2}\setminus \{0\}$. For each $x\in X$, 
this quotient defines a continuous function on $\mathfrak{g}_x$; in particular, 
it must attain a maximum value $C_x$ on $\overline{\mathfrak{g}_x\cap U}$. 
Moreover, since $x\mapsto \mathfrak{g}_x$ is a continuous function of $X$ into 
the Grassmannian and $U$ is an open set, we deduce that $x\mapsto C_x$ is a 
continuous function on the compact space $X$ and therefore attains a maximum 
value $C$. We deduce
\[
|\sigma_x(e^Y)z|_{\mathcal{D}_x^{1/2}}\leq C|z|_{\mathcal{D}_x^{1/2}} 
\]
for every $x\in X$, $Y\in \mathfrak{g}_x\cap U$, and $z\in \mathcal{D}_x^{1/2}$.
This is (\ref{eq:wf_uniform_sigma_bound}) and (\ref{eq:ss_uniform_sigma_bound}).

Next, we must verify (\ref{eq:wf_uniform_decay}) for the wavefront case and 
(\ref{eq:ss_uniform_decay}) for the singular spectrum case. For the wavefront 
case, we must show that for a fixed $\eta_0 \notin \op{Ind}_H^G\op{WF}\tau$, there exists 
a neighborhood $\Omega\subset i\mathfrak{g}^*\setminus \op{Ind}_H^G\op{WF}(\tau)$ 
of $\eta_0$ and an open subset $0\in U_{\textup{CU,WF}}\subset \mathfrak{g}$ such that 
for all $\varphi\in C_c^{\infty}(U_{\textup{CU,WF}})$ and all $N\in \mathbb{N}$, there 
exists a constant $C_{N,\varphi}$ such that
\begin{equation}\label{eq:wfcondUcomp}
\begin{split}
\left|\int\limits_{\mathfrak g_x} \langle \tau_x(e^Y) v_1,v_2\rangle_{\mathcal{V}_x}\cdot
 \left((\sigma_x(e^Y) z_1) \otimes \overline{z_2}\right) \cdot\varphi(Y) e^{\langle t\xi,Y\rangle}
 dY\right|_{\mathcal D^1_x}\\
\leq C_{N,\varphi}(\|v_1\|_{\mathcal{V}_x}|z_1|_{\mathcal{D}_x^{1/2}})\otimes (\|v_2\|_{\mathcal{V}_x}|z_2|_{\mathcal{D}_x^{1/2}})t^{-N}
\end{split}
\end{equation}
for $t>0$ and uniformly for all $\xi\in \Omega$, $x\in X$, $v_1,v_2\in \mathcal{V}_x$, 
and $z_1,z_2\in \mathcal{D}_x^{1/2}$. 

For the singular spectrum case, we must show that for a fixed 
$\eta_0 \notin \op{Ind}_H^G\op{SS}(\tau)$, there is a neighborhood 
$\Omega \subset i\mathfrak{g}^*\setminus \op{Ind}_H^G\op{SS}(\tau)$ of $\eta_0$ 
and a neighborhood $0\in U_{\textup{CU,SS}}\subset \mathfrak g$ such that for
every pair of neighborhoods $0\in U_1\Subset U_2\subset \mathfrak{g}$ with 
$U_2\subset U_{\textup{CU,SS}}$ and all $N\in \mathbb{N}$, there is a constant 
$C_{U_1,U_2}$ such that
 \begin{equation}\label{eq:sscondUcomp}
 \begin{split}
 \left|\int\limits_{\mathfrak g_x} \langle \tau_x(e^Y) v_1,v_2\rangle_{\mathcal{V}_x}\cdot
 \left((\sigma_x(e^Y) z_1) \otimes \overline{z_2}\right) \cdot\varphi_{N,U_1,U_2}(Y) e^{\langle t\xi,Y\rangle}
 dY\right|_{\mathcal D^1_x} \\ 
 \leq C_{U_1,U_2}^{N+1}(N+1)^N (\|v_1\|_{\mathcal V_x} |z_1|_{\mathcal D^{1/2}_x})  \otimes(\|v_2\|_{\mathcal V_x}
 |z_2|_{\mathcal D^{1/2}_x} )t^{-N}
 \end{split}
\end{equation}
for $t>0$ and uniformly for all $\xi\in \Omega$, $x\in X$, $v_1,v_2\in \mathcal{V}_x$, 
and $z_1,z_2\in \mathcal{D}_x^{1/2}$.
 
In order to study the integrals (\ref{eq:wfcondUcomp}) and (\ref{eq:sscondUcomp}), 
we utilize a calculation in Appendix~\ref{app:dense}. By (\ref{eq:densityaction}), 
if $x\in X$, then the group $G_x$ acts on the one dimensional complex vector space
$\mathcal{D}_x^{1/2}$ by the scalar
\begin{equation}
\sigma_{x}(h)z=|\op{det}_{T_{x}X}({dh}_{|x})|^{-1/2}z
\end{equation}
for all $z\in \mathcal{D}_x^{1/2}, h\in G_x$ where $\op{det}_{T_xX}(dh|_x)$ denotes the 
determinant of the differential map on the fibre $T_xX$
\[
{dh}_{|x}\colon T_xX\rightarrow T_xX.
\]
Injecting this expression into the left hand side of (\ref{eq:wfcondUcomp}) or 
(\ref{eq:sscondUcomp}) yields
\begin{align*}
 &\left|\int\limits_{\mathfrak g_x} \langle \tau_x(e^Y) v_1,v_2\rangle_{\mathcal{V}_x}\cdot
 \left((\sigma_x(e^Y) z_1) \otimes \overline{z_2}\right) \cdot\varphi(Y) e^{\langle t\xi,Y\rangle}
 dY\right|_{\mathcal D^1_x}\\ 
 =&\left|\int\limits_{\mathfrak g_x} \langle \tau_x(e^Y) v_1,v_2\rangle_{\mathcal{V}_x}\cdot
 |{\det}_{T_xX}({de^Y}_{|x})|^{-1/2}\cdot z_1 \otimes \overline{z_2} \cdot\varphi(Y) e^{\langle t\xi,Y\rangle}
 dY\right|_{\mathcal D^1_x}\\ 
 =&\left|\int\limits_{\mathfrak g_x} \langle \tau_x(e^Y) v_1,v_2\rangle_{\mathcal{V}_x}\cdot
 |{\det}_{T_xX}({de^Y}_{|x})|^{-1/2}\cdot \varphi(Y) e^{\langle t\xi,Y\rangle}
 dY\right|\cdot |z_1|_{\mathcal D^ {1/2}_x} \otimes |z_2|_{\mathcal D^{1/2}_x}. 
\end{align*}

Next, for each $x\in X$, we choose $g_x\in G$ such that $g_xHg_x^{-1}=G_x$. 
Then we conjugate our integral from $\mathfrak{g}_x$, the Lie algebra of $G_x$, 
to $\mathfrak{h}$, the Lie algebra of $H$.

\begin{align*}
&\left|\int\limits_{\mathfrak g_x} \langle \tau_x(e^Y) v_1,v_2\rangle_{\mathcal{V}_x}\cdot
 |{\det}_{T_xX}({de^Y}_{|x})|^{-1/2}\cdot \varphi(Y) e^{\langle t\xi,Y\rangle}
 dY\right|\cdot |z_1|_{\mathcal D^ {1/2}_x} \otimes |z_2|_{\mathcal D^{1/2}_x}\\ 
 =&\Bigg|\int\limits_{\mathfrak h} \langle \tau_x(e^{\operatorname{Ad}(g_x)\tilde Y}) v_1,v_2\rangle_{\mathcal{V}_x}\cdot
 |{\det}_{T_xX}({de^{\operatorname{Ad}(g_x)\tilde Y}}_{|x})|^{-1/2}\cdot \\
 &\varphi(\operatorname{Ad}(g_x)\tilde Y) e^{\langle t\xi,\operatorname{Ad}(g_x)\tilde Y\rangle}
 j_{\operatorname{Ad}(g_x):\mathfrak h\to\mathfrak g_x}(\tilde Y)d\tilde Y\Bigg|\cdot |z_1|_{\mathcal D^ {1/2}_x} \otimes |z_2|_{\mathcal D^{1/2}_x}\\ 
 =&\Bigg|\int\limits_{\mathfrak h} \langle g_x^ {-1}\tau_x(C_{g_x}e^{\tilde Y}) g_xg_x^ {-1}v_1,g_x^ {-1}v_2\rangle_{V}\cdot
 |{\det}_{T_{eH}X}({de^{\tilde Y}}_{|eH})|^{-1/2}\cdot \\
 &\varphi(\operatorname{Ad}(g_x)\tilde Y) e^{\langle t\xi,\operatorname{Ad}(g_x)\tilde Y\rangle}
 j_{\operatorname{Ad}(g_x):\mathfrak h\to\mathfrak g_x}(\tilde Y)d\tilde Y\Bigg|\cdot |z_1|_{\mathcal D^ {1/2}_x} \otimes |z_2|_{\mathcal D^{1/2}_x}\\ 
 =&\Bigg|\int\limits_{\mathfrak h} \langle\tau(e^{\tilde Y})g_x^ {-1}v_1,g_x^ {-1}v_2\rangle_{V}\cdot
 |{\det}_{T_{eH}X}({de^{\tilde Y}}_{|eH})|^{-1/2}\cdot \\
 &\varphi(\operatorname{Ad}(g_x)\tilde Y) e^{\langle t(\operatorname{Ad}(g_x))^*\xi,\tilde Y\rangle}
 j_{\operatorname{Ad}(g_x):\mathfrak h\to\mathfrak g_x}(\tilde Y)d\tilde Y\Bigg|\cdot |z_1|_{\mathcal D^ {1/2}_x} \otimes |z_2|_{\mathcal D^{1/2}_x}\\ 
 =&(\star)
\end{align*}
In reading the above calculations, recall that we have a fixed inner product on 
$\mathfrak{g}$ which restricts to $\mathfrak{g}_x$ and $\mathfrak{h}$, determining 
the translation invariant densities $dY$ and $d\tilde{Y}$ on these vector spaces. 
In the above calculation, we have used the Jacobian of the conjugation map 
$\op{Ad}(g_x)$, which is defined by the expression
\[
\op{Ad}(g_x)^*dY=j_{\op{Ad}(g_x)\colon \mathfrak{h}\rightarrow \mathfrak{g}_x}(\tilde{Y})d\tilde{Y}.
\]
In addition, as before, $C_{g_x}$ denotes conjugation by $g_x$.

The following Lemma will be useful.
\begin{lemma}\label{lem:precomp_gx}
 If $X=G/H$ is a compact homogenous space, then for any $x\in X$ we can choose
 a $g_x$ such that $x=g_xH$ and that 
$$\{g_x, x\in X\}\subset G$$
 is a precompact set.  
\end{lemma}
\begin{proof}
 Define the $g_x$ via a finite number of continuous sections of the principle 
 $H$-fiber bundle $G\to X$.
\end{proof}
Let us assume from now on that the $g_x$ are fixed according to Lemma~\ref{lem:precomp_gx}.

Now we are ready to prove wavefront condition U and singular spectrum condition U. 
For the wavefront case, let $\eta_0\notin \mathcal W=\op{Ind}_H^G\op{WF}(\tau)$. 
As $\mathcal W$ is closed we can fix 
$\Omega = B_{\varepsilon}(\eta_0)\in i\mathfrak g^*\setminus \mathcal W$.
Next let us introduce
\[
 V_{0,\op{WF}}:= \bigcup\limits_{x\in X} (\operatorname{Ad}(g_x))^*q_x(\Omega)=\bigcup\limits_{x\in X} q(\operatorname{Ad}^*(g_x^{-1})\Omega) \subset i\mathfrak h^*.
\]
The importance of the set $V_{0,\op{WF}}$ arises from the observation 
that for $\xi\in q_x(\Omega)$ we have $(\operatorname{Ad}(g_x))^*\xi \in V_{0,\op{WF}}$.
From the choice of $g_x$ according to Lemma~\ref{lem:precomp_gx} we conclude that 
$V_{0,\op{WF}}$ is a precompact set. Additionally we get from the $\operatorname{Ad}^*(G)$ 
invariance of $\mathcal W$ that
\[
 \overline{\bigcup_{x\in X} \operatorname{Ad}^*(g_x^{-1})\Omega} \cap \mathcal W = \emptyset
\]
and consequently
\[
 V_{0,\op{WF}}\subset  i\mathfrak{h}^* \setminus \operatorname{WF}(\tau)
\]
is a precompact subset. Using the fact that $V_{0,\op{WF}}$ is disjoint with 
$\op{WF}(\tau)$ and that it is precompact, we deduce that there is 
$\tilde \varepsilon_{\op{WF}} >0$ such that
for all smooth functions $\tilde \varphi$ with 
$\operatorname{supp}(\tilde\varphi)\subset B_{\tilde\varepsilon_{\op{WF}}}(0)\subset \mathfrak h$
and all $N>0$ there is a constant $C_{N,\tilde \varphi}$ uniformly in $\tilde \xi\in V_{0,\op{WF}}$ 
and $v_1,v_2\in V$ such that \cite[Theorem 1.4 (v)]{Ho81}
\begin{equation}\label{eq:ccu_uniform_decay_wf}
 \left| \int_{\mathfrak h}\langle\tau(e^{\tilde Y})v_1,v_2\rangle_V\tilde\varphi(\tilde Y)e^ {\langle t\tilde \xi,\tilde Y\rangle} d\tilde Y \right| 
 \leq C_{N,\tilde\varphi}\|v_1\|_V\cdot\|v_2\|_Vt^{-N}
\end{equation}

For the singular spectrum analogue, we define $\mathcal{S}:=\op{Ind}_H^G\op{SS}(\tau)$, 
and we define $V_{0,\op{SS}}$ with $\op{WF}$ replaced by $\op{SS}$ everywhere. 
We analogously obtain that there is $\tilde{\varepsilon}_{\op{SS}}>0$ such that 
whenever $0\in \widetilde{U_1}\Subset \widetilde{U_2}\subset B_{\tilde{\varepsilon}_{\op{SS}}}(0)\subset \mathfrak{h}$ 
are open sets, there exists a constant $C_{\widetilde{U_1},\widetilde{U_2}}$ for which
\begin{equation}\label{eq:ccu_uniform_decay_ss}
 \left| \int_{\mathfrak h}\langle\tau(e^{\tilde Y})v_1,v_2\rangle_V\varphi_{N,\widetilde{U_1},
 \widetilde{U_2}}(\tilde Y)e^ {\langle t\tilde \xi,\tilde Y\rangle} d\tilde Y \right| 
 \leq C_{\widetilde{U_1},\widetilde{U_2}}^{N+1}(N+1)^N\|v_1\|_V\cdot\|v_2\|_V t^{-N}
\end{equation}
uniformly in $\tilde \xi\in V_{0,\op{SS}}$.

Now we will fix the open neighborhood of $0\in U_{CU,\op{WF}}\subset\mathfrak g$ 
according to 
\[
U_{CU,\op{WF}}\subset \bigcap_{x\in X} \operatorname{Ad}(g_x) 
B_{\tilde \varepsilon_{\op{WF}}/2}(0) \subset \mathfrak g
\]
Once more this is possible due to Lemma~\ref{lem:precomp_gx}. Finally we choose 
a particular $\tilde\varphi \in C^\infty_0(B_{\tilde\varepsilon_{\op{WF}}}(0))$ 
such that $\tilde \varphi(Y)=1$ for 
$Y\in B_{\tilde \varepsilon_{\op{WF}}/2}(0)\subset \mathfrak h$. Note that the cutoff function
$\varphi$ from above satisfies $\operatorname{supp} \varphi \subset U_{CU,\op{WF}}$ and 
by the choice of $U_{CU,\op{WF}}$ we conclude 
$\operatorname{supp}(\varphi\circ\operatorname{Ad}(g_x))\subset B_{\tilde\varepsilon_{\op{WF}}/2}(0)$.
Thus we can insert $\tilde\varphi$ in $(\star)$ and obtain
\begin{align*}
 (\star)&=\Bigg| \mathcal F_{\mathfrak{h}}\bigg[\langle\tau(e^{\bullet})g_x^ {-1}v_1,g_x^ {-1}v_2\rangle_{V}\cdot
 \tilde\varphi(\bullet) \cdot |{\det}_{T_{eH} X}({de^{\bullet}}_{|x})|^{-1/2}\\
 &\cdot \varphi\circ\operatorname{Ad}(g_x)(\bullet)\cdot
 j_{\operatorname{Ad}g_x:\mathfrak h\to\mathfrak g_x}(\bullet)\bigg](t\cdot q(\operatorname{Ad}(g_x))^*\xi)\Bigg|
 \cdot |z_1|_{\mathcal D^ {1/2}_x} \otimes |z_2|_{\mathcal D^{1/2}_x}\\
 &=\Bigg| \bigg(\mathcal F_{\mathfrak{h}}\bigg[\langle\tau(e^{\bullet})g_x^ {-1}v_1,g_x^ {-1}v_2\rangle_{V}\cdot
 \tilde\varphi(\bullet)\bigg]*\mathcal F_{\mathfrak h}\bigg[\rho_{x,\varphi}(\bullet)\bigg]\Bigg)(t\cdot q(\operatorname{Ad}(g_x))^*\xi)\Bigg|\\
 &\cdot |z_1|_{\mathcal D^ {1/2}_x} \otimes |z_2|_{\mathcal D^{1/2}_x}\\
\end{align*}
where $\mathcal F_{\mathfrak{h}}:\mathcal E(\mathfrak h)\to C^\infty(i\mathfrak h^*)$
is the Euclidean Fourier transform on $\mathfrak h$ and
\[
 \rho_{x,\varphi}(\tilde Y)=|{\det}_{T_{eH}X}({de^{\tilde Y}}_{|x})|^{-1/2} 
 \cdot \varphi(\operatorname{Ad}(g_x)\tilde Y) \cdot j_{g_x:\mathfrak h\to\mathfrak g_x}(\tilde Y).
\]

For the singular spectrum case, we analogously define $U_{CU,\op{SS}}$ utilizing $\tilde\varepsilon_{\op{SS}}$ and we obtain the above statement with $\varphi$ replaced by $\varphi_{N,U_1,U_2}$ where $0\in U_1\Subset U_2\subset U_{CU,\op{SS}}$ and with $\tilde\varphi$ replaced by $\varphi_{N,\widetilde{U_1},\widetilde{U_2}}$ where $0\in \widetilde{U_1}\Subset \widetilde{U_2}\subset B_{\tilde\varepsilon_{\op{SS}}}(0)$.

Back in the wavefront case, we observe that from the compactness of $X$ we obtain that 
\[
|\mathcal F_{\mathfrak h}[\rho_{x,\varphi}](\xi)|\leq C_{N,\varphi} \langle \xi\rangle^{-N}
\]
with $C_{N,\varphi}$ is independent of $x\in X$ and $\xi\in \mathfrak h^*$. Using this observation together with 
(\ref{eq:ccu_uniform_decay_wf}) we finally obtain
\begin{align*}
 (\star)&\leq C_{N,\varphi} \|g_x^{-1}v_1\|_V\|g_x^{-1}v_2\|_V\cdot |z_1|\otimes|z_2| |t|^{-N}\\
        &=C_{N,\varphi} \|v_1\|_{\mathcal V_x}\|v_2\|_{\mathcal V_x}\cdot |z_1|\otimes|z_2| |t|^{-N}.\\
\end{align*}
 
We have thus shown (\ref{eq:wf_uniform_decay}) and verified wavefront condition U.
For the singular spectrum case, we replace $\varphi$ by $\varphi_{N,U_1,U_2}$ in the 
definition of $\rho_{x,\varphi}$ (which now depends on $N$) and we utilize the stronger bounds
\[
|\mathcal F_{\mathfrak h}[\rho_{x,\varphi}](\xi)|\leq C_{U_1,U_2}^{N+1}(N+1)^N \langle \xi\rangle^{-N}
\]
where $C_{U_1,U_2}$ is a constant independent of $N$. These stronger bounds can 
be obtained from a boundary
values of holomorphic functions argument which is similar to (though not identical to) 
the proof of Lemma 7.2 of \cite{HHO16}.
Finally the bound (\ref {eq:ss_uniform_sigma_bound}) follows directly with the compactness 
of $X$. This finishes the proof of singular spectrum condition U.

\section{Proof of condition U in the dense semisimple case}
\label{sec:CU_dense_semisimple}
\subsection{The dense semisimple condition}

Suppose $G$ is a real, linear algebraic group, and let $H\subset G$ be a closed 
subgroup. Let $\mathfrak{g}$ (resp. $\mathfrak{h}$) denote the Lie algebra of 
$G$ (resp. $H$). Assume that there exists a real, linear algebraic group 
$H_1\subset G$ with Lie algebra $\mathfrak{h}$ (Note that we do not assume $H$ is algebraic). 

\begin{definition} We say $X\in \mathfrak{h}$ is semisimple if and only if for 
every homomorphism of algebraic groups 
$$\rho\colon H_1\rightarrow \operatorname{GL}(N,\mathbb{R}),$$ 
$d\rho(X)\in \mathfrak{gl}(N,\mathbb{R})$ is diagonalizable over the complex numbers.
\end{definition}

By the general theory of linear, algebraic groups, to check whether $X$ is 
semisimple, it is enough to find one \emph{injective} map 
$$\rho\colon H_1 \rightarrow \operatorname{GL}(N,\mathbb{R})$$ and check 
whether $d\rho(X)$ is semisimple (see for instance 4.4 Theorem on pages 83 and 
84 of \cite{Bo}). Moreover, we note that $X\in \mathfrak{h}\subset \mathfrak{g}$ 
is semisimple with respect to $H_1$ iff it is semisimple with respect to $G$.

\begin{definition} Let $\mathfrak{h}_s\subset \mathfrak{h}$ denote the cone of 
elements of $\mathfrak{h}$ that are semisimple with respect to $H_1$.
\end{definition}

We wish to study the case where $\mathfrak{h}_s\subset \mathfrak{h}$ is dense and 
prove the following theorem.

\begin{theorem} \label{thm:cond_U_dense_semisimple} 
Suppose $G$ is a real, linear algebraic group, 
suppose $H\subset G$ is a closed subgroup, and suppose $\tau$ is a finite dimensional, 
unitary representation of $H$. Assume that there exists a real, linear algebraic 
group $H_1\subset G$ such that $\mathfrak{h}$, the Lie algebra of $H$, is also the 
Lie algebra of $\mathfrak{h}_1$. In addition, assume $\mathfrak{h}_s\subset \mathfrak{h}$ 
is dense. Then wavefront condition U and singular spectrum condition U are both satisfied.
\end{theorem}

Note that Theorem~\ref{thm:cond_U_dense_semisimple} together with 
Theorem~\ref{thm:WF_condU} (resp. Thm~\ref{thm:SS_condU} in the singular 
spectrum case) and \cite[Theorem 1.1]{HHO16} imply Theorem~\ref{thm:dense semisimple}.

If $H$ is a real, reductive algebraic group, then $\mathfrak{h}$ 
contains an open, dense subset of semisimple elements. If $H=P$ is a parabolic subgroup
of a real, reductive algebraic group, then $\mathfrak{h}=\mathfrak{p}$ has a dense subset
of semisimple elements. In particular, this holds when $H=B$ and 
\[
B=\left\{\left(\begin{matrix} a & x\\ 0 & a^{-1}\end{matrix}\right)\Big|\ a\in \mathbb{R}^{\times},\ x\in \mathbb{R}\right\}
\]
denotes the motion group of the real line. If $H=N$ is an infinite unipotent group, for instance
\[
N=\left\{\left(\begin{matrix} 1 & x\\ 0 & 1\end{matrix}\right)\Big|\ x\in \mathbb{R}\right\},
\]
then $\mathfrak{h}_s=\mathfrak{n}_s=\{0\}$ and $\mathfrak{h}=\mathfrak{n}$ does not have a dense subset
of semisimple elements.

\subsection{On the conjugacy of maximal toral subalgebras in real, linear algebraic groups}
In order to prove Theorem~\ref{thm:cond_U_dense_semisimple}, we first 
need a technical fact from the structure theory of real, linear algebraic groups.

Let $H_{\mathbb{C}}$ be a connected, complex linear algebraic group defined over $\mathbb{R}$, and let $H=H_{\mathbb{C}}(\mathbb{R})$ 
be the real points of $H_{\mathbb{C}}$. 

\begin{definition}
We say that a complex linear algebraic group $T_{\mathbb{C}}$ is \emph{diagonalizable} 
if it is isomorphic to a closed subgroup of a product of copies of $\mathbb{C}^{\times}$.
A \emph{torus} $T_{\mathbb{C}}\subset H_{\mathbb{C}}$ in $H_{\mathbb{C}}$ is a connected, 
diagonalizable subgroup of $H_{\mathbb{C}}$. A \emph{maximal torus} in $H_{\mathbb{C}}$ 
is a torus in $H_{\mathbb{C}}$ that is not properly contained in another torus in 
$H_{\mathbb{C}}$.
\end{definition}

\begin{proposition} \label{lem:finitetori} There are finitely many $H$ conjugacy 
classes of maximal tori $T_{\mathbb{C}}\subset H_{\mathbb{C}}$ which are defined 
over $\mathbb{R}$. 
\end{proposition}
\begin{proof}
Let $(H_{\mathbb{C}})_u$ denote the unipotent radical of $H_{\mathbb{C}}$. It is 
the unique maximal closed, connected, normal, unipotent subgroup of $H_{\mathbb{C}}$ 
(see for instance 11.21 on page 157 of \cite{Bo}). Moreover, $(H_{\mathbb{C}})/(H_{\mathbb{C}})_u$ 
is a complex, reductive algebraic group. Now, we assumed that $H_{\mathbb{C}}$ 
was defined over $\mathbb{R}$ with real form $H$. It follows that $(H_{\mathbb{C}})_u$ 
is defined over $\mathbb{R}$ (see 14.4.5 Proposition on page 250 of \cite{Sp}); 
we will denote the corresponding set of real points by $H_u$. By 
Corollary 12.2.2 on page 212 of \cite{Sp}, we see that $H/H_u$ is a real form 
of $H_{\mathbb{C}}/(H_{\mathbb{C}})_u$. In particular, $H/H_u$ is a real, reductive 
algebraic group. 

Let
$$\rho_{\mathbb{C}}\colon H_{\mathbb{C}}\rightarrow H_{\mathbb{C}}/(H_{\mathbb{C}})_u$$
denote the natural surjective homomorphism. By 11.14 Proposition (1) on page 152 
of \cite{Bo}, $\rho_{\mathbb{C}}$ induces a map from maximal tori in $H_{\mathbb{C}}$ 
to maximal tori in $H_{\mathbb{C}}/(H_{\mathbb{C}})_u$. Note that this map is 
surjective. Indeed, suppose $B_{\mathbb{C}}\subset H_{\mathbb{C}}/(H_{\mathbb{C}})_u$ 
is a maximal torus in $H_{\mathbb{C}}/(H_{\mathbb{C}})_u$. Choose a maximal 
torus $T_{\mathbb{C}}\subset H_{\mathbb{C}}$, and note that $\rho_{\mathbb{C}}(T_{\mathbb{C}})$ 
is conjugate to $B_{\mathbb{C}}$ since all maximal tori in 
$H_{\mathbb{C}}/(H_{\mathbb{C}})_u$ are conjugate (see 11.3 Corollary (1) on 
page 148 of \cite{Bo}). In particular, there exists 
$\rho_{\mathbb{C}}(g)\in H_{\mathbb{C}}/(H_{\mathbb{C}})_u$ such that
$$\rho_{\mathbb{C}}(g)\rho_{\mathbb{C}}(T_{\mathbb{C}})\rho_{\mathbb{C}}(g)^{-1}=B_{\mathbb{C}}.$$
One sees
$$\rho_{\mathbb{C}}(gT_{\mathbb{C}}g^{-1})=B_{\mathbb{C}}$$
and the induced map on maximal tori is surjective.

Now, we know that there are finitely many $H$ conjugacy classes of maximal tori 
$B_{\mathbb{C}}\subset H_{\mathbb{C}}/(H_{\mathbb{C}})_u$ that are defined over 
$\mathbb{R}$ (this may be deduced from the definition of Cartan subalgebra on 
page 254 of \cite{Kn05}, the remarks on pages 457 and 458 of \cite{Kn05} including 
Proposition 7.35, and the easy fact that a maximal torus is determined by its 
Lie algebra). Let $\{B_{\mathbb{C}}^{\alpha}\}_{\alpha\in \mathcal{A}}$ be a set
of representatives for these conjugacy classes. For each $\alpha$, choose a 
maximal torus $T_{\mathbb{C}}^{\alpha}\subset H_{\mathbb{C}}$ such that 
\[
\rho_{\mathbb{C}}(T_{\mathbb{C}}^{\alpha})=B_{\mathbb{C}}^{\alpha}.
\]

We claim that every maximal torus $T_{\mathbb{C}}\subset H_{\mathbb{C}}$ that is 
defined over $\mathbb{R}$ is $H$ conjugate to a maximal torus of the form 
$T_{\mathbb{C}}^{\alpha}$. Indeed, $\rho_{\mathbb{C}}(T_{\mathbb{C}})$ is defined 
over $\mathbb{R}$; hence, there exists $g\in H$ such that 
\[
\rho_{\mathbb{C}}(g)\rho_{\mathbb{C}}(T_{\mathbb{C}})\rho_{\mathbb{C}}(g)^{-1}=B_{\mathbb{C}}^{\alpha}
\]
for some $\alpha$. In particular, 
$gT_{\mathbb{C}}g^{-1}\in \rho_{\mathbb{C}}^{-1}(B_{\mathbb{C}}^{\alpha})$. 

Next, we observe
\[
\rho_{\mathbb{C}}^{-1}(B_{\mathbb{C}}^{\alpha})=T_{\mathbb{C}}^{\alpha}\cdot (H_{\mathbb{C}})_u
\]
since $(H_{\mathbb{C}})_u\subset H_{\mathbb{C}}$ is a normal subgroup. One deduces 
that $\rho_{\mathbb{C}}^{-1}(B_{\mathbb{C}}^{\alpha})$ is a solvable subgroup of 
$H_{\mathbb{C}}$ that is defined over $\mathbb{R}$. By 19.2 Theorem on page 223 
of \cite{Bo}, all maximal tori in $\rho_{\mathbb{C}}^{-1}(B_{\mathbb{C}}^{\alpha})$ 
that are defined over $\mathbb{R}$ are conjugate via the set of real points of 
$\rho_{\mathbb{C}}^{-1}(B_{\mathbb{C}}^{\alpha})$ (in particular, they are conjugate 
via $H$). Since $T_{\mathbb{C}}^{\alpha}$ is a maximal torus in 
$\rho_{\mathbb{C}}^{-1}(B_{\mathbb{C}}^{\alpha})$ that is defined over $\mathbb{R}$, 
we deduce that $gT_{\mathbb{C}}g^{-1}$ is conjugate to $T_{\mathbb{C}}^{\alpha}$. 
Thus, $T_{\mathbb{C}}$ is conjugate to $T_{\mathbb{C}}^{\alpha}$. The Proposition 
follows.
\end{proof}
Next, we move to the Lie algebra. Let $\mathfrak{h}$ denote the Lie algebra of 
$H$, and let $\mathfrak{h}_{\mathbb{C}}$ denote the Lie algebra of $H_{\mathbb{C}}$

\begin{definition} A \emph{toral subalgebra} $\mathfrak{t}_{\mathbb{C}}$ in 
$\mathfrak{h}_{\mathbb{C}}$ is an abelian subalgebra consisting of semisimple 
elements. A \emph{maximal toral subalgebra} of $\mathfrak{h}_{\mathbb{C}}$ is a 
toral subalgebra of $\mathfrak{h}_{\mathbb{C}}$ that is not properly contained 
in another toral subalgebra of $\mathfrak{h}_{\mathbb{C}}$. 
\end{definition}

\begin{lemma} \label{lem:toritoral} Suppose $\mathfrak{t}_{\mathbb{C}}\subset \mathfrak{h}_{\mathbb{C}}$ 
is a maximal toral subalgebra. Then there exists a maximal torus $T_{\mathbb{C}}\subset G_{\mathbb{C}}$
such that $\mathfrak{t}_{\mathbb{C}}$ is the Lie algebra of $T_{\mathbb{C}}$. 
In particular, all maximal toral subalgebras of $\mathfrak{h}_{\mathbb{C}}$ are 
conjugate. 

In addition, if $\mathfrak{t}_{\mathbb{C}}\subset \mathfrak{h}_{\mathbb{C}}$ is 
a maximal toral subalgebra defined over $\mathbb{R}$, then the corresponding 
maximal torus $T_{\mathbb{C}}\subset H_{\mathbb{C}}$ is defined over $\mathbb{R}$. 
In particular, there are finitely many $H$ conjugacy classes of maximal toral 
subalgebras of $\mathfrak{h}_{\mathbb{C}}$ that are defined over $\mathbb{R}$. 
\end{lemma}

\begin{proof} Let $B_{\mathbb{C}}=Z_{H_{\mathbb{C}}}(\mathfrak{t}_{\mathbb{C}})$ 
be the centralizer in $H_{\mathbb{C}}$ of $\mathfrak{t}_{\mathbb{C}}$, and note
$$\mathfrak{b}_{\mathbb{C}}:=\op{Lie}(B_{\mathbb{C}})=Z_{\mathfrak{h_{\mathbb{C}}}}(\mathfrak{t}_{\mathbb{C}}).$$
Let 
$\widetilde{T_{\mathbb{C}}}\subset B_{\mathbb{C}}$ be a maximal torus, and let 
$C_{\mathbb{C}}=Z_{B_{\mathbb{C}}}(\widetilde{T_{\mathbb{C}}})$ be the associated 
Cartan subgroup of $B_{\mathbb{C}}$. Then 
\[
\mathfrak{c}_{\mathbb{C}}:=\operatorname{Lie}(Z_{B_{\mathbb{C}}}(\widetilde{T_{\mathbb{C}}}))=Z_{\mathfrak{b}_{\mathbb{C}}}(\op{Lie}(\widetilde{T_{\mathbb{C}}}))
\]
\noindent Hence, 
$\mathfrak{t}_{\mathbb{C}}\subset \mathfrak{c}_{\mathbb{C}}$ since $\mathfrak{t}_{\mathbb{C}}$ is in the center of $\mathfrak{b}_{\mathbb{C}}$, and therefore it must centralize $\op{Lie}(\widetilde{T}_{\mathbb{C}})\subset \mathfrak{b}_{\mathbb{C}}$. But, $C_{\mathbb{C}}$ is a nilpotent Lie group; hence, by part (3) of 10.6 Theorem on page 138 of \cite{Bo}, 
we have 
\[
C_{\mathbb{C}}\simeq (C_{\mathbb{C}})_s\times (C_{\mathbb{C}})_u
\]
and 
\[
\mathfrak{c}_{\mathbb{C}}\simeq (\mathfrak{c}_{\mathbb{C}})_s\oplus (\mathfrak{c}_{\mathbb{C}})_n.
\]
In particular, we see that $\widetilde{T_{\mathbb{C}}}=(C_{\mathbb{C}})_s$ is 
the unique maximal toral subgroup of $C_{\mathbb{C}}$ and its Lie algebra 
$\widetilde{\mathfrak{t}_{\mathbb{C}}}=(\mathfrak{c}_{\mathbb{C}})_s$ is the 
unique maximal toral subalgebra of $\mathfrak{c}_{\mathbb{C}}$. Since 
$\mathfrak{t}_{\mathbb{C}}$ is a toral subalgebra in $\mathfrak{c}_{\mathbb{C}}$, 
it must be a subalgebra of $\widetilde{\mathfrak{t}_{\mathbb{C}}}$. But, 
$\mathfrak{t}_{\mathbb{C}}$ is maximal in $\mathfrak{g}_{\mathbb{C}}$; hence, it 
must be maximal in the smaller algebra $\mathfrak{c}_{\mathbb{C}}$. Therefore, 

$$\mathfrak{t}_{\mathbb{C}}=\widetilde{\mathfrak{t}_{\mathbb{C}}}$$
and $\mathfrak{t}_{\mathbb{C}}$ is the Lie algebra of the maximal torus 
$\widetilde{T_{\mathbb{C}}}\subset B_{\mathbb{C}}$. However, any torus containing 
$\widetilde{T_{\mathbb{C}}}$ in $H_{\mathbb{C}}$ would have to centralize 
$\widetilde{T_{\mathbb{C}}}$; hence, it would have to lie in $B_{\mathbb{C}}$. 
Therefore $\mathfrak{t}_{\mathbb{C}}$ is the Lie algebra of a maximal torus in 
$H_{\mathbb{C}}$. 

The second statement follows from the conjugacy of maximal tori and the fact 
that if two groups are conjugate, then their Lie algebras must also be conjugate.

For the third statement, we assume $\mathfrak{t}_{\mathbb{C}}$ is defined over 
$\mathbb{R}$. Then its centralizer $B_{\mathbb{C}}$ must be defined over $\mathbb{R}$. 
Now, a complex, linear algebraic group defined over $\mathbb{R}$ always has a 
maximal torus defined over $\mathbb{R}$ (see 18.2 Theorem (i) on page 182 of \cite{Bo}).
Then we may choose $\widetilde{T_{\mathbb{C}}}$ in the above argument to be 
defined over $\mathbb{R}$, and we have a maximal torus with Lie algebra 
$\mathfrak{t}_{\mathbb{C}}$ that is defined over $\mathbb{R}$.

The last statement follows since two maximal toral subalgebras are conjugate by 
an element of $H$ if they are the Lie algebras of maximal tori which are conjugate 
by an element of $H$. 
\end{proof}

\begin{definition} We say $\mathfrak{t}\subset \mathfrak{h}$ is a 
\emph{toral subalgebra} of $\mathfrak{h}$ if $\mathfrak{t}$ is an abelian subalgebra 
consisting of semisimple elements. A \emph{maximal toral subalgebra} of $\mathfrak{h}$ 
is a toral subalgebra of $\mathfrak{h}$ that is not properly contained in another 
toral subalgebra of $\mathfrak{h}$. 
\end{definition}

\begin{lemma} \label{lem:realtoral} If $\mathfrak{t}\subset \mathfrak{h}$ is a 
maximal toral subalgebra, then $\mathfrak{t}_{\mathbb{C}}\subset \mathfrak{h}_{\mathbb{C}}$ 
is a maximal toral subalgebra. 
\end{lemma} 

\begin{proof} Suppose $\mathfrak{t}\subset \mathfrak{h}$ is a maximal toral 
subalgebra with complexification $\mathfrak{t}_{\mathbb{C}}$. Note 
$Z_H(\mathfrak{t})\subset Z_{H_{\mathbb{C}}}(\mathfrak{t}_{\mathbb{C}})$ is a 
real form. Choose a maximal toral subalgebra $\widetilde{\mathfrak{t}_{\mathbb{C}}}\subset \mathfrak{h}_{\mathbb{C}}$ of $\mathfrak{h}_{\mathbb{C}}$ containing $\mathfrak{t}_{\mathbb{C}}$. Clearly $\widetilde{\mathfrak{t}_{\mathbb{C}}}$ is contained in $Z_{\mathfrak{h}_{\mathbb{C}}}(\mathfrak{t}_{\mathbb{C}})$, the Lie algebra of $Z_{H_{\mathbb{C}}}(\mathfrak{t}_{\mathbb{C}})$, since $\mathfrak{t}_{\mathbb{C}}$ is contained in $\widetilde{\mathfrak{t}_{\mathbb{C}}}$ and $\widetilde{\mathfrak{t}_{\mathbb{C}}}$ is abelian. By Lemma \ref{lem:toritoral}, there exists a maximal torus $\widetilde{T_{\mathbb{C}}}$ with Lie algebra $\widetilde{\mathfrak{t}_{\mathbb{C}}}$. This maximal torus of $H_{\mathbb{C}}$ is contained in $Z_{H_{\mathbb{C}}}(\mathfrak{t}_{\mathbb{C}})$ since it is abelian and $\mathfrak{t}_{\mathbb{C}}$ is contained in its Lie algebra. All maximal tori in a complex, linear algebraic group are conjugate; therefore, every maximal torus 
in $Z_{H_{\mathbb{C}}}(\mathfrak{t}_{\mathbb{C}})$ is a maximal torus in $H_{\mathbb{C}}$. 

Now, by 18.2 Theorem (i) on page 218 of \cite{Bo}, there exists a maximal torus 
$B_{\mathbb{C}}\subset Z_{H_{\mathbb{C}}}(\mathfrak{t}_{\mathbb{C}})$ that is 
defined over $\mathbb{R}$. Then its Lie algebra $\mathfrak{b}_{\mathbb{C}}$ is 
defined over $\mathbb{R}$ with real points $\mathfrak{b}$. Since $\mathfrak{t}_{\mathbb{C}}$ 
is in the center of $Z_{\mathfrak{h}_{\mathbb{C}}}(\mathfrak{t}_{\mathbb{C}})$, we must have 
$\mathfrak{t}_{\mathbb{C}}\subset \mathfrak{b}_{\mathbb{C}}$. In particular, 
$\mathfrak{t}\subset \mathfrak{b}$. By maximality, we must have $\mathfrak{t}=\mathfrak{b}$. 
Therefore, $\mathfrak{t}_{\mathbb{C}}=\mathfrak{b}_{\mathbb{C}}$ is a maximal 
toral subalgebra of $\mathfrak{h}_{\mathbb{C}}$ as desired.
\end{proof}  

\begin{proposition} \label{prop:finitetoral} Let $H$ be a Lie group with Lie algebra 
$\mathfrak{h}$, and suppose that there exists a real, linear algebraic group 
$H_1$ with Lie algebra $\mathfrak{h}$ as well. Assume $\mathfrak{h}_s\subset \mathfrak{h}$ 
is dense. Then there exists a finite number of maximal toral subalgebras 
$\mathfrak{t}_1,\ldots,\mathfrak{t}_r$ in $\mathfrak{h}$ such that 
\[
\overline{\bigcup_{j=1}^r\bigcup_{g\in H} \operatorname{Ad}(g)\cdot \mathfrak{t}_j}=\mathfrak{h}.
\]
\end{proposition}

\begin{proof} Note that $H_e$ and $(H_1)_e$ are both connected Lie groups with 
the same Lie algebra. In particular, we have
\[
\operatorname{Ad}(H_e)=\operatorname{Ad}((H_1)_e)\subset \operatorname{Aut}(\mathfrak{h}).
\]
Now, there are finitely many $((H_1)_{\mathbb{C}})_e(\mathbb{R})$ conjugacy classes of maximal toral subalgebras 
in $\mathfrak{h}$ by Lemma \ref{lem:toritoral} and Lemma \ref{lem:realtoral}. Hence, there are finitely
many $H_1\supset ((H_1)_{\mathbb{C}})_e(\mathbb{R})$ conjugacy classes of maximal toral subalgebras in $\mathfrak{h}$.
Also, note that $H_1$ has finitely many connected components since $H_1$ is algebraic. Hence, 
there are finitely many $(H_1)_e$ conjugacy classes of maximal tori in $\mathfrak{h}$. 
Since $\operatorname{Ad}(H_e)=\operatorname{Ad}((H_1)_e)$, there are also finitely 
many $H$ conjugacy classes of maximal tori in $\mathfrak{h}$. Let 
$\mathfrak{t}_1,\ldots,\mathfrak{t}_r$ be representatives of these conjugacy 
classes. Since every semisimple element in $\mathfrak{h}$ is contained in a 
maximal toral subalgebra, we have
$$\mathfrak{h}_s=\bigcup_{j=1}^r\bigcup_{g\in H} \operatorname{Ad}(g)\cdot \mathfrak{t}_j$$
for maximal toral subalgebras $\mathfrak{t}_1,\ldots,\mathfrak{t}_r$. Since we 
assumed $\mathfrak{h}_s\subset \mathfrak{h}$ to be dense, the Proposition follows.
\end{proof}

\subsection{Proof of Theorem~\ref{thm:cond_U_dense_semisimple}}

Recall that we fixed an inner product $\langle \cdot,\cdot\rangle$ on 
$\mathfrak{g}$, and let $|\cdot|$ be the associated norm. Using the inner product 
and division by $i$ we can identify $i\mathfrak{g}^*\simeq \mathfrak{g}$, and by 
abuse of notation also denote by $|\cdot|$ the corresponding norm on $i\mathfrak{g}^*$.

Recall that for $x\in X$, $G_x\subset G$ is the stabilizer subgroup
of $x$ and $\mathfrak{g}_x$ is the associated Lie algebra.
Let us first show that Proposition~\ref{prop:finitetoral} allows
us to conjugate any semisimple element $Y\in\mathfrak g_x$ to 
an elment in $\mathfrak h$ in a bounded way.

\begin{lemma}\label{lem:bounded_AdY_norm}
Let $G$ be a real, linear algebraic group, let $H\subset G$ 
be a closed subgroup, and suppose there exists a real algebraic
subgroup $H_1\subset G$ such that the Lie algebra of $H_1$, $\mathfrak{h}_1$, and 
the Lie algebra of $H$, $\mathfrak{h}$, are equal as subsets of the Lie algebra of $G$, $\mathfrak{g}$.
 Then there is a constant $C$ such that for any $x\in X=G/H$ and any
 semisimple element  $Y\in \mathfrak g_x$ there is a  $g_Y \in G$
 such that 
 $\operatorname{Ad}(g_Y^{-1})Y\in\mathfrak h$ and
 \[
  |\operatorname{Ad}(g_Y^{-1})Y|\leq C |Y|.
 \] 
\end{lemma}
\begin{proof} 
For any element $Y\in\mathfrak g_x$ there is an element $g_x$ 
such that $\operatorname{Ad}(g_x^{-1})Y\in\mathfrak h$. If $Y$
is semisimple then also $\operatorname{Ad}(g_x^{-1})Y\in
\mathfrak h$ is semisimple and as in the proof of Proposition~\ref{prop:finitetoral} we may conjugate 
$\operatorname{Ad}(g_x^{-1})Y$ by some $h\in H$ to an element 
in one of the finitely many toral subalgebras $\mathfrak t_1,\ldots,\mathfrak t_r$. In particular, putting $g_Y=g_x h$, we obtain
\[
 \operatorname{Ad}(g_Y^{-1})Y\in \mathfrak t_j
\]
for some $j=1,\ldots,r$.

Now lets choose an embedding 
\[
\rho:\ G_{\mathbb{C}}\rightarrow \operatorname{GL}(N,\mathbb{C}),
\]
and choose a maximal torus $T_{\mathbb{C}}\subset G_{\mathbb{C}}$. Since 
$\rho(T_{\mathbb{C}})$ is a group of commuting diagonalizable matrices, they are 
simultaneously diagonalizable. In particular, after conjugating $\rho$, we may 
assume that $\rho$ takes $T_{\mathbb{C}}$ into the set of diagonal matrices in 
$\operatorname{GL}(N,\mathbb{C})$. Now, we fix a new norm $|\cdot|_{\rho}$ on 
$\mathfrak{g}_{\mathbb C}$ by 
\[
|X|_{\rho}:=|d\rho(X)|_{\operatorname{op}}.
\]
That is, the norm of $X$ is the operator norm of the endomorphism $d\rho(X)$ of 
$\mathbb{C}^N$. Suppose $X\in \mathfrak{g}_{\mathbb C}$ and $v\in \mathbb{C}^N$ is an eigenvector 
for $d\rho(X)$ with eigenvalue $\lambda$. Then $d\rho(X)v=\lambda v$ and we deduce 
$|X|_{\rho}\geq |\lambda|$. More generally, we have
\[
|X|_{\rho}\geq \sup_{\lambda\in \operatorname{Spec}d\rho(X)} |\lambda|
\]
where $\operatorname{Spec}d\rho(X)$ denotes the set of eigenvalues of $d\rho(X)$. 
On the other hand, if $X\in \mathfrak{t}_{\mathbb C}=\operatorname{Lie}(T_{\mathbb C})$, then $d\rho(X)$ 
is a diagonal matrix, and the above inequality is in fact an equality. If $g\in G_{\mathbb C}$, 
since $d\rho(\operatorname{Ad}(g)X)=\rho(g)d\rho(X)\rho(g)^{-1}$ is conjugate to 
$d\rho(X)$, we deduce
\[
|X|_{\rho}=\sup_{\lambda\in \operatorname{Spec}d\rho(X)} |\lambda|=\sup_{\lambda\in \operatorname{Spec}d\rho(\operatorname{Ad}_gX)} |\lambda|\leq |\operatorname{Ad}(g)X|_{\rho}.
\]
In particular, for every semisimple $G_{\mathbb{C}}$ orbit 
$\mathcal{O}\subset \mathfrak{g}_{\mathbb{C}}$, the norm $|\cdot|_{\rho}$ takes 
its minimum value on $\mathfrak{t}_{\mathbb{C}}\cap \mathcal{O}$ (note that the 
latter set is not empty since all maximal tori are conjugate and every semisimple 
element belongs to some maximal torus). 
Recall that $\mathfrak t_j\subset \mathfrak h$ are maximal toral subalgebras so 
accroding to Lemma~\ref{lem:realtoral}, $\mathfrak t_{j,\mathbb C}\subset \mathfrak h_{\mathbb C}$
is a maximal toral subalgebra. Applying Lemma~\ref{lem:toritoral} to $G_{\mathbb C}$
and using $\mathfrak h_{\mathbb C} \subset\mathfrak g_{\mathbb C}$ we fix for each $j$ a 
maximal torus $T_{j,\mathbb{C}}\subset G_{\mathbb{C}}$ such that 
\[
\operatorname{Lie}(T_{j,\mathbb{C}})\supset\mathfrak{t}_{j,\mathbb{C}}
\]

For each $j$, we fix a homomorphism
\[
\rho_j\colon G_{\mathbb{C}}\rightarrow \operatorname{GL}(N_j,\mathbb{C})
\]
for which $\rho_j(T_{j,\mathbb{C}})$ is a collection of diagonal matrices. By this map
we obtain a finite number of norms $|~|_{\rho_j}$ on the Lie algebra $\mathfrak g_{\mathbb C}$
and thus also in the real subspace $\mathfrak g$.
Since all norms on a finite dimensional vectorspace are equivalent we obtain a 
constant $d>0$ such that for all $j=1,\ldots,r$
\begin {equation}\label{eq:norm_equiv}
 \frac{1}{d}|X|\leq|X|_{\rho_j}\leq d|X|.
\end{equation}
Now suppose that $j$ is such that $\operatorname{Ad}(g_Y^ {-1})Y\in \mathfrak t_j$.
As noted before, when restricted to the orbit $G_{\mathbb{C}}\cdot Y$, the norm 
$|\cdot|_{\rho_j}$ takes its minimum on $\mathfrak{t}_{j,\mathbb{C}}$. Therefore,
\[
|\operatorname{Ad}(g_Y^ {-1})Y|_{\rho_j}\leq |Y|_{\rho_j}
\]
and using (\ref{eq:norm_equiv}) we get
\[
 |\operatorname{Ad}(g_Y^ {-1})Y|\leq d^2|Y|.
\]
Note that the constant $d^2$ did only depend on the choices of the norms $\rho_j$
and is thus independent of $Y$ and $g_Y$ and we have thus proven Lemma~\ref{lem:bounded_AdY_norm}.
\end{proof}

As a corollary of this lemma we can now prove (\ref{eq:wf_uniform_sigma_bound}) and 
(\ref{eq:ss_uniform_sigma_bound}) for $U_{\textup{CU,WF}}=U_{\textup{CU,SS}}=U_{\textup{CU}}\subset \mathfrak g$ being any precompact 
open neighborhood of $0\in \mathfrak g$: 
Let $ g_x\in G$ be an arbitrary representative of $x= g_xH \in X$. 
After choosing a nonzero point of reference in $\mathcal D^{1/2}_x$ and using 
(\ref{eq:sigma_eH}) and (\ref{eq:sigma_x}) we can identify $\mathcal D^{1/2}_x\cong \mathbb C$
and
\[
 \sigma_x(e^Y) \cong |{\det}_{T_{eH}X}(d{e^{\operatorname{Ad}( g_x^{-1}) Y}}_{|eH})|^{-1/2}
\]
for any $Y\in \mathfrak g_x$. From continuity we conclude
\[
 |{\det}_{T_{eH}X}(d{e^{\operatorname{Ad}( g_x^{-1}) Y}}_{|eH})|^{-1/2} \leq C_1|\operatorname{Ad}( g_x^{-1}) Y|.
\]
Note that the left side is independent of the 
choice of the representative $ g_x$. Now suppose that $Y \in \mathfrak g_x$ is semisimple, then we can take $g_x=g_Y$ 
according to Lemma~\ref{lem:bounded_AdY_norm} and we obtain
\[
 |{\det}_{T_{eH}X}(d{e^{\operatorname{Ad}(g_Y^{-1}) Y}}_{|eH})|^{-1/2} \leq C_1C|Y|.
\]
By the precompactnes of $U_{\textup{CU}}$, we have thus established 
(\ref{eq:wf_uniform_sigma_bound}) and (\ref{eq:ss_uniform_sigma_bound}) on the dense subset of semisimple elements and finally 
by continuity (\ref{eq:wf_uniform_sigma_bound}) and (\ref{eq:ss_uniform_sigma_bound}) follow on the whole 
set $U_{\textup{CU}}$.

In order to prove the central estimates (\ref{eq:wf_uniform_decay}) and (\ref{eq:ss_uniform_decay}), we will perform a partial integration with respect to some vectors
$Y_x\in \mathfrak{g}_x$ for every $x\in X$ satisfying certain properties. We first show that we can choose them appropriately.

As we assume in Theorem~\ref{thm:cond_U_dense_semisimple} that the representation
$\tau$ is finite dimensional, we know that $\operatorname{WF}(\tau_x) = \op{SS}(\tau_x) = \{0\}$, so
\[
 \mathcal W = \mathcal S = \overline{\bigcup_{x\in X} q_x^{-1}(0)} = \overline{\bigcup_{x\in X} i(\mathfrak g/\mathfrak g_x)^*}
\]
where $(\mathfrak g/\mathfrak g_x)^*$ are those linear functionals on $\mathfrak g$ 
that vanish on $\mathfrak g_x$.
\begin{lemma} \label{lem:lowerbound} 
Suppose $\eta_0 \notin \mathcal W=\mathcal S$. Then there exists an open neighborhood 
$\eta_0 \in \Omega \subset i\mathfrak{g}^*\setminus \mathcal W$ such that the 
nonnegative constant
\[
C_{\Omega}:= \inf_{x\in X}  \sup_{\substack{Y_x\in \mathfrak{g}_x\\ |Y_x|=1}} 
\inf_{\xi\in \Omega} |\langle \xi,Y_x\rangle|
\]
is nonzero.
\end{lemma}

\begin{proof} We prove the Lemma by contradiction. Let $\{\Omega_m\}$ be a 
sequence of open sets in $i\mathfrak{g}^*$ such that $\Omega_m\supset \Omega_{m+1}$ 
for every $m$ and 
\[
\bigcap_m \Omega_m=\{\eta_0\}.
\]
Note that by the definition of $C_\Omega$, we have $C_{\omega_{m+1}}\geq C_{\Omega_m}$. Let us now
suppose that $C_{\Omega_m}=0$ for every $m$. Then, for every $m$, we must be able to 
find a sequence $\{x_n^m\}$ of elements of $X$ such that
\[
c_n^m=\sup_{\substack{Y_{n,m}\in \mathfrak{g}_{x_n^m}\\ |Y_{n,m}|=1}} \inf_{\xi\in \Omega_m}|\langle \xi,Y_{n,m}\rangle|
\]
converges to zero as $n\rightarrow \infty$. Choose a sequence of increasing 
natural numbers $\{n_m\}$ such that $c_{n_m}^m$ converges monotonically to zero 
as $m\rightarrow \infty$.

Since the Grassmannian of all vector spaces of dimension $\dim \mathfrak{h}$ in 
$\mathfrak{g}$ is a compact space, after passing to a subsequence, we may assume 
that $\{\mathfrak{g}_{x_{n_m}^m}\}\rightarrow V$ where $V\subset \mathfrak{g}$ is 
a subspace of $\mathfrak{g}$ of dimension $\dim \mathfrak{h}$.

Next, we will show $q_V(\eta_0)=0$. If $Y\in V$ and $|Y|=1$, then we may write 
$Y=\lim_{m\to\infty} Y_m$ with $Y_m\in \mathfrak{g}_{x_{n_m}^m}$ and $|Y_m|=1$. We note 
\[
\inf_{\xi\in \Omega_m}|\langle \xi,Y_m\rangle|\leq c_{n_m}^m.
\]
For every $\xi\in \Omega_m$ we can bound
\[
|\langle \eta_0,Y \rangle|\leq |\langle \eta_0,Y-Y_m\rangle|+|\langle \eta_0-\xi,Y_m\rangle|+|\langle \xi,Y_m\rangle|.
\]
By choosing $\xi$ such that the last term becomes as small as possible, we obtain
\[
|\langle \eta_0,Y\rangle|\leq |\langle \eta_0,Y-Y_m\rangle|+\sup_{\xi\in \Omega_m}|\langle \eta_0-\xi,Y_m\rangle|+\inf_{\xi\in \Omega_m}|\langle \xi,Y_m\rangle|.
\]
The right hand side converges to zero as $m\rightarrow \infty$ since 
$c_{n_m}^m\rightarrow 0$ and $\cap \Omega_m=\eta_0$. 
Therefore, $\eta_0$ vanishes on $V$.

Finally, we show that
\[
i(\mathfrak{g}/V)^*\subset \overline{\bigcup_{x\in X} i(\mathfrak{g}/\mathfrak{g}_x)^*}. 
\]

Indeed, it is not difficult to see that if $\{\mathfrak{g}_{x_n}\}\rightarrow V$ 
in the Grassmannian of dimension $\dim \mathfrak{h}$ subspaces of $\mathfrak{g}$, 
then $\{\mathfrak{g}_{x_n}^{\perp}\}\rightarrow V^{\perp}$ in the Grassmannian of 
dimension $\dim \mathfrak{g}-\dim \mathfrak{h}$ subspaces of $\mathfrak{g}$. 
Dividing by $i$ and utilizing our fixed inner product on $\mathfrak{g}$, we may 
identify imaginary valued linear functionals on $\mathfrak{g}$ that vanish on 
$V$ (resp. $\mathfrak{g}_{x_n}$) with $V^{\perp}$ (resp. $\mathfrak{g}_{x_n}^{\perp}$). 
The statement now follows. 

Now, $\eta_0\in i(\mathfrak{g}/V)^*$. Hence, 
$\eta_0\in \overline{\bigcup_{x\in X} i(\mathfrak{g}/\mathfrak{g}_x)^*}$. But, this 
contradicts our hypothesis. Hence, $C_{\Omega_m}\neq 0$ for some $m$, and the 
Lemma has been proven. 
\end{proof}
Now putting together Lemma~\ref{lem:bounded_AdY_norm} and Lemma~\ref{lem:lowerbound} we can specify the vectors $Y_x$.
\begin{corollary} \label{cor:vectorchoices} Let $G$ be a real,
linear algebraic group, let $H\subset G$ be a closed subgroup, 
and suppose there exists a real algebraic group $H_1\subset G$ 
such that the Lie algebra of $H_1$, $\mathfrak{h}_1$, and 
the Lie algebra of $H$, $\mathfrak{h}$, are equal as subsets of the Lie algebra 
of $G$, $\mathfrak{g}$. In addition, assume that the set of
semisimple elements of $\mathfrak{h}$, denoted $\mathfrak{h}_s$ 
is dense in $\mathfrak{h}$. Fix 
\[
\eta_0\notin \mathcal W = \mathcal S = \overline{\bigcup_{x\in X} i(\mathfrak{g}/\mathfrak{g}_x)^*}, 
\]
and let $\eta_0\in \Omega\subset i\mathfrak{g}^*$ be the open set 
from Lemma~\ref{lem:lowerbound}. For every $x\in X$, we 
can choose $Y_x$ and $g_x$ satisfying
\begin{enumerate}
\item $Y_x\in \mathfrak{g}_x\ \text{for\ all}\ x\in X$
\item $|Y_x|=1\ \text{for\ all}\ x\in X$
\item For all $x\in X$, we have the inequality 
\[
\inf_{\xi\in\Omega} |\langle \xi, Y_x\rangle|>\frac{C_{\Omega}}{2} 
\]
\item For each $x\in X$, there exists $g_x\in G$ such that
$\operatorname{Ad}(g_x^ {-1})Y_x\in \mathfrak h$ and
\[
|\operatorname{Ad}(g_x^ {-1})Y_x|\leq C
\]
uniformly in $x$. 
\end{enumerate}
\end{corollary}
\begin{proof}
From Lemma~\ref{lem:lowerbound} we conclude that the set of possible 
choices of $Y_x$ satisfying (1)-(3)
\[
\left\{Y\in \mathfrak{g}_x\Big|\ \inf_{\xi\in\Omega} |\langle \xi,Y\rangle|>\frac{C_{\Omega}}{2}\ \text{and}\ |Y|=1\right\}
\]
is an open subset of the unit sphere in $\mathfrak{g}_x$.

As $\mathfrak{h}_s\subset \mathfrak{h}$ is a dense cone we can choose 
$Y_x$ semisimple. We may then choose $g_x=g_Y$ according to 
Lemma~\ref{lem:bounded_AdY_norm} and obtain (4).
\end{proof}

For the remainder of the proof of Theorem~\ref{thm:cond_U_dense_semisimple}, 
we fix choices of $Y_x$ and $g_x$ satisfying the above four properties. 

Recall that, in order to prove the central estimate (\ref{eq:wf_uniform_decay})
in the wavefront case, we have to consider the asymptotics of the 
oscillating integrals
\[
\mathcal I_x(t):= \left|\int\limits_{\mathfrak g_x} \langle \tau_x(e^Y) v_1,v_2\rangle_{\mathcal{V}_x}\cdot
 \left((\sigma_x(e^Y) z_1) \otimes \overline{z_2}\right) \cdot\varphi(Y) e^{\langle t\xi,Y\rangle}
 dY\right|_{\mathcal D^1_x}.
\]
In the singular spectrum case, we have to consider the same expression with $\varphi$ 
replaced by $\varphi_{N,U_1,U_2}$. 
We will do this by partial integration, analogous to the case of homogeneous spaces
in Section~\ref{sec:scalar}. In order to get the uniform bounds in $x$, we will
first rewrite $\mathcal I_x(t)$ as an integral over $\mathfrak g$, respectively a neighborhood
of the identity in $G$. Let us first describe how we do this:

For any $(\dim \mathfrak h)$-dimensional subspace $V\subset \mathfrak g$ we consider 
the map 
\[
 \kappa_V:\mathfrak g=V\oplus V^\perp \to G,~Y+Z\mapsto e^Ye^Z,
\]
which is a local diffeomorphism around $0\in\mathfrak g$.
As the Grassmanian of all (dim$\mathfrak h$)-dimensional subspaces is 
compact, we can find an open neighborhood of $0\in\mathfrak g$
such that $\kappa_V$ is a diffeomorphism for any subspace $V\subset \mathfrak g$.
We choose $U_{\textup{CU}}$ to be contained in this subset.

Next, let $\rho\in C_c^\infty(U_{\textup{CU}})$, be a positive function such that
$\int_{V^\perp}\rho(Z)dZ=1$ for all subspaces $V$. This is possible 
by taking a function that is radial w.r.t.~the chosen scalar product on $\mathfrak g$.

We can write
\begin{align*}
 \mathcal I_x(t)&= \left|\int\limits_{\mathfrak g_x}\int\limits_{\mathfrak g_x^\perp} \langle \tau_x(e^Y) v_1,v_2\rangle_{\mathcal{V}_x}\cdot
 \left((\sigma_x(e^Y) z_1) \otimes \overline{z_2}\right) \cdot\varphi(Y) e^{\langle t\xi,Y\rangle}
 \rho(Z)dYdZ\right|_{\mathcal D^1_x}\\
 &=\left|\int\limits_{G} \langle \tau_x(e^{Y_{\mathfrak g_x}(g)}) v_1,v_2\rangle_{\mathcal{V}_x}\cdot
 \left((\sigma_x(e^{Y_{\mathfrak g_x}(g)}) z_1) \otimes \overline{z_2}\right) 
 \cdot\varphi(Y_{\mathfrak g_x}(g)) \rho(Z_{\mathfrak g_x}(g)) 
 j_{\mathfrak g_x}(g)e^{\langle t\xi,{Y_{\mathfrak g_x}(g)}\rangle} dg\right|_{\mathcal D^1_x}.
\end{align*}
Here dg is a left invariant Haar measure on $G$ and for any (dim$\mathfrak h$)-dimensional subspace 
$V\subset \mathfrak g$, $j_V(g)$ is the analytic Jacobian, such that $\kappa_V^*(j_V(g)dg)$ is
the Lebesgue measure $dYdZ$ on $\mathfrak g$. Furthermore $Y_V(g)\in V$ and $Z_V(g)\in V^\perp$
are the projections on $V$ and $V^\perp$ of $\kappa_V^{-1}(g)$. 

We can now perform the partial integration. Therefore consider the Lie algebra elements
$ Y_x\in \mathfrak g_x$ as right invariant differential operators on $G$ acting from left and
define
\[
 \mu(\xi, x,g):=Y_x \langle\xi ,Y_{\mathfrak g_x}(g)\rangle .
\]
Note that $\mu(\xi, x, e) = \langle \xi, Y_x\rangle$, thus by property (3) of 
Corollary~\ref{cor:vectorchoices}, $|\mu(\xi, x, e)|$ is bounded away from zero uniformly in $\xi\in \Omega$ and 
$x\in X$. We choose $U_{\textup{CU}}$ small enough, such that $|\mu(\xi, x, g)|>c$
is still bounded away from zero, if additionally $g$ is in the support of 
$\varphi(Y_{\mathfrak g_x}(g)) \rho(Z_{\mathfrak g_x}(g))$. 
Note that this is again possible due to the compactness of the Grassmanian 
of (dim$\mathfrak h$)-dimensional subspaces. Now we can insert the differential operator 
$(t^{-1}\mu(\xi, x,g)^{-1}Y_x)^N$ in front of 
$e^{\langle t\xi,{Y_{\mathfrak g_x}(g)}\rangle}$ and integrate by parts which yields
\begin{align*}
\mathcal I_x(t)=&\Bigg|\int\limits_{G} \langle \tau_x(e^{Y_{\mathfrak g_x}(g)}) v_1,v_2\rangle_{\mathcal{V}_x}\cdot
 \left((\sigma_x(e^{Y_{\mathfrak g_x}(g)}) z_1) \otimes \overline{z_2}\right) \\
 &\quad\quad\quad\quad\quad\quad \cdot\varphi(Y_{\mathfrak g_x}(g)) \rho(Z_{\mathfrak g_x}(g)) 
 j_{\mathfrak g_x}(g)\left[t^{-1}\mu(\xi, x,g)Y_x\right]^N 
 e^{\langle t\xi,{Y_{\mathfrak g_x}(g)}\rangle} dg\Bigg|_{\mathcal D^1_x}\\
 \leq&t^{-N}
  \sum_{q+r+s\leq N} C_{q,r,s} \Bigg|\int\limits_{G} \langle {Y_x}^q\tau_x(e^{Y_{\mathfrak g_x}(g)}) v_1,v_2\rangle_{\mathcal{V}_x}\cdot\left((Y_x^r\sigma_x(e^{Y_{\mathfrak g_x}(g)}) z_1) \otimes \overline{z_2}\right) \\
  &\quad\quad\quad\quad\quad\quad\quad\quad
  \cdot \Big(Y_x^s \varphi(Y_{\mathfrak g_x}(g)) \rho(Z_{\mathfrak g_x}(g)) 
 j_{\mathfrak g_x}(g)\Big) e^{\langle t\xi,{Y_{\mathfrak g_x}(g)}\rangle}
 dY\Bigg|_{\mathcal D^1_x}\\
 =&t^{-N}
  \sum_{q+r+s \leq N} C_{q,r,s} \Bigg|\int\limits_{G} \langle d\tau_x(Y_x)^q 
  \tau_x(e^{Y_{\mathfrak g_x}(g)})v_1,v_2\rangle_{\mathcal{V}_x}\cdot \left((d\sigma_x(Y_x)^r\sigma_x(e^{Y_{\mathfrak g_x}(g)}) z_1) 
  \otimes \overline{z_2}\right) \\
  &\quad\quad\quad\quad\quad\quad\quad\quad
  \cdot  \rho(Z_{\mathfrak g_x}(g))\cdot  \Big({Y_x}^s\varphi(Y_{\mathfrak g_x}(g))
 j_{\mathfrak g_x}(g) \Big) e^{\langle t\xi,{Y_{\mathfrak g_x}(g)}\rangle} dY\Bigg|_{\mathcal D^1_x}\\
\end{align*}
Note that in the first inequality we absorbed the functions $\mu(\xi,x,g)$ and its derivatives in the
constants $C_{q,r,s}$ which is possible as $|Y_x|=1$ (Corollary~\ref{cor:vectorchoices} (2))
and because the function $Y_V(g)$ depends continuously in the $C^\infty$-topology on on the subspace
$V\subset\mathfrak g$ considered as a point in the compact Grassmanian. For the second equality, 
we used, that 
\[
\tfrac{d}{ds}_{|s=0}Z_{\mathfrak{g}_x}(e^{sY_x}g)=0 \textup{ and }
\tfrac{d}{ds}_{|s=0}e^{Y_{\mathfrak{g}_x}(e^{sY_x}g)}=\tfrac{d}{ds}_{|s=0}e^{sY_x}e^{Y_{\mathfrak{g}_x}},
\]
which follows as $Y_X\in \mathfrak g_x$ (Corollary~\ref{cor:vectorchoices} (1)).

We obtain an analogous formula for the singular spectrum case if we insert this 
differential operator on the left hand side of (\ref{eq:ss_uniform_decay}). In the 
singular spectrum case formula, $\varphi$ is replaced by $\varphi_{N,U_1,U_2}$. 

Back in the wavefront case, for any  $\varphi\in C_c^\infty(U_{\textup{CU}})$ the 
precompactnes of $U_{\textup{CU}}$ and the compactness of the Grassmanian 
in which $\mathfrak g_x$ varies, assures a bound of the supremum norm
\[
 \|{Y_x}^s\varphi(Y_{\mathfrak g_x}(g)) j_{\mathfrak g_x}(g)) )\|_\infty \leq C_{s,\varphi}
\]
uniformly in $x\in X$. For the singular spectrum case, part (3) of the definition 
of the family of functions $\varphi_{N,U_1,U_2}$ given directly above
Definition \ref{def:ss_cond_U} together with the analyticity of 
$j_{\mathfrak g_x}(g)$ give the stronger bounds
\[
 \|{Y_x}^s\varphi_{N,U_1,U_2}(Y_{\mathfrak g_x}(g))j_{\mathfrak g_x}(g)\|_\infty \leq C_{U_1,U_2}^{s+1}(N+1)^s.
\]
Additionally we have shown above that
$|\sigma_x(e^Y)|_{\textup{op}} < C$ uniformly in $x\in X$, $Y\in U_{\textup{CU}}\cap\mathfrak g_x$ 
and a uniform bound for $|\tau_x(e^Y)|_{\textup{op}}$ follows trivially from the
unitarity of $\tau_x$. It thus remains to prove uniform bounds for 
$|d\tau_x(Y_x)|_{\operatorname{op}}$ and $|d\sigma_x(Y_x)|_{\operatorname{op}}$.
Using once more (\ref{eq:tau_x}) we obtain
\[
|d\tau_x(Y_x)|_{\operatorname{op}}=|d\tau(\operatorname{Ad}(g_x^{-1})Y_x)|_{\operatorname{op}}.
\]
Now the continuity of $d\tau:\mathfrak h\to\operatorname{End}(V)$ and 
$|~|_{\textup{op}}: \operatorname{End}(V) \to \mathbb R$ imply
\[
|d\tau_x(Y_x)|_{\operatorname{op}}\leq C |\operatorname{Ad}(g_x^{-1})Y_x| \leq \tilde C 
\]
where the last inequality is justified by Lemma~\ref{lem:bounded_AdY_norm}. The
same arguments apply to $|d\sigma_x(Y_x)|_{\operatorname{op}}$. Putting everything 
together, in the wavefront case, we obtain 
\begin{align*}
 \mathcal I_x(t)\leq  
 t^{-N} C_{N,\varphi}(\|v_1\|_{\mathcal V_x} |z_1|_{\mathcal D^{1/2}_x})  \otimes(\|v_2\|_{\mathcal V_x}
 |z_2|_{\mathcal D^{1/2}_x} ).
\end{align*}
Analogously, in the singular spectrum case, we obtain
\begin{align*}
 \mathcal I_x(t)\leq  
 t^{-N} C_{U_1,U_2}^{N+1}(N+1)^N(\|v_1\|_{\mathcal V_x} |z_1|_{\mathcal D^{1/2}_x})  \otimes(\|v_2\|_{\mathcal V_x}
 |z_2|_{\mathcal D^{1/2}_x} ).
\end{align*}
This finishes the proofs of (\ref{eq:wf_uniform_decay}) and (\ref{eq:ss_uniform_decay}) 
and thus of Theorem~\ref{thm:cond_U_dense_semisimple}.

\section{Applications and Examples}
\label{sec:examples}

In this section, we consider applications of Theorem~\ref{thm:L^2 space} and Theorem~\ref{thm:dense semisimple}, and we consider examples of those applications. We begin by expanding upon Theorem \ref{thm:main} and Theorem \ref{thm:tempered}. As in the introduction, let $G_x$ be the stabilizer in $G$ of $x\in X$, and let $\mathfrak{g}_x$ denote the Lie algebra of $G_x$. For each $x\in X$, we identify $iT_x^*X$ with $i(\mathfrak{g}/\mathfrak{g}_x)^*$ in the obvious way.  

\begin{corollary} \label{cor:main} Suppose $G$ is a real, reductive algebraic group, and suppose $X$ is a homogeneous space for $G$ with a non-zero invariant density. Then
\begin{equation}\label{eq:maincor1}
\operatorname{AC}\left(\bigcup_{\substack{ \sigma\in \operatorname{supp}L^2(X)\\ \sigma\in \widehat{G}_{\text{temp}}}}\mathcal{O}_{\sigma}\right)\subset \overline{\bigcup_{x\in X} iT_x^*X}=\overline{\bigcup_{x\in X} i(\mathfrak{g}/\mathfrak{g}_x)^*}.
\end{equation}
Intersecting with the set of regular semisimple elements, we obtain the equality
\begin{equation}\label{eq:maincor2}
\operatorname{AC}\left(\bigcup_{\substack{ \sigma\in \operatorname{supp}L^2(X)\\ \sigma\in \widehat{G}_{\text{temp}}^{\text{\ }\prime}}}\mathcal{O}_{\sigma}\right)\cap (i\mathfrak{g}^*)'=\overline{\bigcup_{x\in X} iT_x^*X}\cap (i\mathfrak{g}^*)'=\overline{\bigcup_{x\in X} i(\mathfrak{g}/\mathfrak{g}_x)^*}\cap (i\mathfrak{g}^*)'.
\end{equation}
If, in addition, $\operatorname{supp}L^2(X)\subset \widehat{G}_{\text{temp}}$, then we obtain equality without intersecting with the set of regular semisimple elements,
\begin{equation}\label{eq:maincor3}
\operatorname{AC}\left(\bigcup_{\substack{ \sigma\in \operatorname{supp}L^2(X)\\ \sigma\in \widehat{G}_{\text{temp}}}}\mathcal{O}_{\sigma}\right)=\overline{\bigcup_{x\in X} iT_x^*X}=\overline{\bigcup_{x\in X} i(\mathfrak{g}/\mathfrak{g}_x)^*}.
\end{equation}
\end{corollary}	

This Corollary is a consequence of Theorems~1.1 and 1.2 of \cite{HHO16}, Theorem~1.1 of \cite{Ha}, and Theorem~1.1 of this paper. The last statement is especially interesting to us because of the recent work of Benoist and Kobayashi \cite{BK15}, which gives a large class of homogeneous spaces $X$ for a real, reductive algebraic group $G$ for which $\operatorname{supp}L^2(X)\subset \widehat{G}_{\text{temp}}$.
Let us write down a brief proof of Corollary \ref{cor:main}.

\begin{proof} Consider the representation $L^2(X)$ of $G$. We may decompose this 
representation into irreducibles as 
$$L^2(X)\simeq \int^{\oplus}_{\sigma\in \widehat{G}}\sigma^{\oplus m(L^2(X),\sigma)} d\mu$$
where $\mu$ is a nonnegative measure on the unitary dual $\widehat{G}$. Recall that $\widehat{G}_{\text{temp}}\subset \widehat{G}$ is a closed subset and $\widehat{G}_{\text{temp}}'\subset \widehat{G}$ is an open subset. This can be deduced from Corollary 2 on page 391 of \cite{Fel60} together with standard facts about tempered characters (see for instance \cite{Kn86}). Therefore, we may decompose
$$\widehat{G}=\widehat{G}_{\text{temp}}\bigcup \left(\widehat{G}\setminus\widehat{G}_{\text{temp}}\right)$$
as the union of a closed set and an open set and we may analogously decompose
$$\mu=\mu|_{\widehat{G}_{\text{temp}}}+\mu|_{\widehat{G}\setminus\widehat{G}_{\text{temp}}}.$$
We then obtain an analogous direct sum decomposition of the representation $L^2(X)$,
$$L^2(X)\cong \int^{\oplus}_{\sigma\in \widehat{G}_{\text{temp}}} \sigma^{\oplus m(L^2(X),\sigma)} d\mu|_{\widehat{G}_{\text{temp}}}\bigoplus \int^{\oplus}_{\sigma\in \widehat{G}\setminus\widehat{G}_{\text{temp}}} \sigma^{\oplus m(L^2(X),\sigma)} d\mu|_{\widehat{G}\setminus\widehat{G}_{\text{temp}}}.$$
Since every matrix coefficient of the first summand is also a matrix coefficient of $L^2(X)$, we obtain the inclusion
$$\op{WF}(L^2(X))\supset\op{WF}\left(\int^{\oplus}_{\sigma\in \widehat{G}_{\text{temp}}} \sigma^{\oplus m(L^2(X),\sigma)} d\mu|_{\widehat{G}_{\text{temp}}}\right).$$
Now, by Theorem 1.2 of \cite{HHO16}, we obtain
$$\op{WF}\left(\int^{\oplus}_{\sigma\in \widehat{G}_{\text{temp}}} \sigma^{\oplus m(L^2(X),\sigma)} d\mu|_{\widehat{G}_{\text{temp}}}\right)=\op{AC}\left(\bigcup_{\substack{\sigma\in \op{supp}L^2(X)\\ \sigma\in \widehat{G}_{\text{temp}}}}\mathcal{O}_{\sigma}\right).$$
On the other hand, one obtains from Theorem 1.1 of this paper
$$\op{WF}(L^2(X))=\overline{\bigcup_{x\in X} iT_x^*X}=\overline{\bigcup_{x\in X} i(\mathfrak{g}/\mathfrak{g}_x)^*}.$$
Now, statement (\ref{eq:maincor1}) follows. If, in addition, $\op{supp}L^2(X)\subset \widehat{G}_{\text{temp}}$, then 
$$\mu|_{\widehat{G}\setminus\widehat{G}_{\text{temp}}}=0$$
and we obtain 
$$\op{WF}(L^2(X))=\op{WF}\left(\int^{\oplus}_{\sigma\in \widehat{G}_{\text{temp}}} \sigma^{\oplus m(L^2(X),\sigma)} d\mu|_{\widehat{G}_{\text{temp}}}\right).$$
Now, statement (\ref{eq:maincor3}) follows from Theorem 1.2 of \cite{HHO16} together with Theorem 1.1 of this paper (which utilizes Theorem 1.1 of \cite{HHO16} in its proof).

To show statement (\ref{eq:maincor2}), we require a different decomposition of measures. We instead break up
$$\widehat{G}=\widehat{G}_{\text{temp}}^{\ \prime}\bigcup \left(\widehat{G}\setminus\widehat{G}_{\text{temp}}^{\ \prime}\right)$$
into the union of an open set and a closed set and we have the corresponding decomposition of measures
$$\mu=\mu|_{\widehat{G}_{\text{temp}}^{\ \prime}}+\mu|_{\widehat{G}\setminus\widehat{G}_{\text{temp}}^{\ \prime}}.$$
This yields a decomposition of representations
$$L^2(X)\cong \int^{\oplus}_{\sigma\in \widehat{G}_{\text{temp}}^{\ \prime}} \sigma^{\oplus m(L^2(X),\sigma)} d\mu|_{\widehat{G}_{\text{temp}}^{\ \prime}}\bigoplus \int^{\oplus}_{\sigma\in \widehat{G}\setminus\widehat{G}_{\text{temp}}^{\ \prime}} \sigma^{\oplus m(L^2(X),\sigma)} d\mu|_{\widehat{G}\setminus\widehat{G}_{\text{temp}}^{\ \prime}}.$$
Now, as shown in Proposition 1.3 of \cite{Ho81}, every matrix coefficient of $L^2(X)$ decomposes into the sum of a matrix coefficient of the first representation plus a matrix coefficient of the second representation which means that we can write $\op{WF}(L^2(X))$ as a union of 
$$\op{WF}\left(\int^{\oplus}_{\sigma\in \widehat{G}_{\text{temp}}^{\ \prime}} \sigma^{\oplus m(L^2(X),\sigma)} d\mu|_{\widehat{G}_{\text{temp}}^{\ \prime}}\right)$$
and
$$\op{WF}\left(\int^{\oplus}_{\sigma\in \widehat{G}\setminus\widehat{G}_{\text{temp}}^{\ \prime}} \sigma^{\oplus m(L^2(X),\sigma)} d\mu|_{\widehat{G}\setminus\widehat{G}_{\text{temp}}^{\ \prime}}\right).$$
Now, Theorem 1.1 of \cite{Ha} says that the second set is contained in the singular set, $i\mathfrak{g}^*\setminus(i\mathfrak{g}^*)'$, since it is a direct integral of singular representations. Therefore, we obtain
$$\op{WF}(L^2(X))\cap (i\mathfrak{g}^*)'=\op{WF}\left(\int^{\oplus}_{\sigma\in \widehat{G}_{\text{temp}}^{\ \prime}} \sigma^{\oplus m(L^2(X),\sigma)} d\mu|_{\widehat{G}_{\text{temp}}^{\ \prime}}\right)\cap (i\mathfrak{g}^*)^{\prime}$$
By Theorem 1.1 of this paper (which utilizes Theorem 1.1 of \cite{HHO16} in its proof)
$$\op{WF}(L^2(X))=\overline{\bigcup_{x\in X} iT_x^*X}=\overline{\bigcup_{x\in X} i(\mathfrak{g}/\mathfrak{g}_x)^*}.$$
And by Theorem 1.2 of \cite{HHO16}, we obtain
$$\op{WF}\left(\int^{\oplus}_{\sigma\in \widehat{G}_{\text{temp}}^{\ \prime}} \sigma^{\oplus m(L^2(X),\sigma)} d\mu|_{\widehat{G}_{\text{temp}}^{\ \prime}}\right)=\op{AC}\left(\bigcup_{\substack{\sigma\in \op{supp}L^2(X)\\ \sigma\in \widehat{G}_{\text{temp}}^{\ \prime}}}\mathcal{O}_{\sigma}\right).$$
Statement (\ref{eq:maincor2}) now follows.
\end{proof}

Utilizing Theorem \ref{thm:dense semisimple}, we give a variant of Corollary \ref{cor:main} for finite rank vector bundles.

\begin{corollary} \label{cor:dense semisimple cor} Suppose $G$ is a real, reductive algebraic group, suppose $X$ is a homogeneous space for $G$, let $\mathcal{D}^{1/2}\rightarrow X$ denote the bundle of complex half densities on $X$ (see Appendix A), and suppose $\mathcal{V}\rightarrow X$ is a finite rank, $G$ equivariant, Hermitian vector bundle on $X$. Assume that for some (equivalently any) $x\in X$, there is a closed, real, linear algebraic subgroup $\widetilde{G_x}\subset G$ whose Lie algebra is $\mathfrak{g}_x$ (of course, this is satisfied if $G_x$ is itself an algebraic subgroup of $G$) and $\mathfrak{g}_x$ contains a dense subset of semisimple elements.
Then
\begin{equation}\label{eq:dense_semisimple_cor_part1}
\operatorname{AC}\left(\bigcup_{\substack{ \sigma\in \operatorname{supp}L^2(X,\mathcal{D}^{1/2}\otimes \mathcal{V})\\ \sigma\in \widehat{G}_{\text{temp}}}}\mathcal{O}_{\sigma}\right)\subset \overline{\bigcup_{x\in X} iT_x^*X}=\overline{\bigcup_{x\in X} i(\mathfrak{g}/\mathfrak{g}_x)^*}.
\end{equation}
Intersecting with the set of regular semisimple elements, we obtain the equality
\begin{equation}\label{eq:dense_semisimple_cor_part2}
\operatorname{AC}\left(\bigcup_{\substack{ \sigma\in \operatorname{supp}L^2(X,\mathcal{D}^{1/2}\otimes \mathcal{V})\\ \sigma\in \widehat{G}_{\text{temp}}^{\text{\ }\prime}}}\mathcal{O}_{\sigma}\right)\cap (i\mathfrak{g}^*)'=\overline{\bigcup_{x\in X} i(\mathfrak{g}/\mathfrak{g}_x)^*}\cap (i\mathfrak{g}^*)'.
\end{equation}
If, in addition, $\operatorname{supp}L^2(X,\mathcal{D}^{1/2}\otimes \mathcal{V})\subset \widehat{G}_{\text{temp}}$, then we obtain equality without intersecting with the set of regular, semisimple elements,
\begin{equation}\label{eq:dense_semisimple_cor_part3}
\operatorname{AC}\left(\bigcup_{\substack{ \sigma\in \operatorname{supp}L^2(X,\mathcal{D}^{1/2}\otimes \mathcal{V})\\ \sigma\in \widehat{G}_{\text{temp}}}}\mathcal{O}_{\sigma}\right)=\overline{\bigcup_{x\in X} iT_x^*X}=\overline{\bigcup_{x\in X} i(\mathfrak{g}/\mathfrak{g}_x)^*}.
\end{equation}
\end{corollary} 

This Corollary is a consequence of Theorems 1.1 and 1.2 of \cite{HHO16}, Theorem 1.1 of \cite{Ha}, and Theorem 1.2 of this paper. The verification of Corollary \ref{cor:dense semisimple cor} is nearly identical to the verification of Corollary \ref{cor:main} with Theorem 1.1 of this paper replaced by Theorem 1.2 of this paper in all arguments. Utilizing results of Matumoto \cite{Ma92}, we will point out in Section \ref{sec:counterexamples} that all three of the statements of Corollary \ref{cor:dense semisimple cor} may fail if one does not assume that $\mathfrak{g}_x$ contains a dense subset of semisimple elements.

In order to utilize Corollary \ref{cor:main} and Corollary \ref{cor:dense semisimple cor}, it is useful to write these statements in terms of purely imaginary valued linear functionals on the Cartan subalgebras of $\mathfrak{g}$, the Lie algebra of $G$. Suppose $\mathfrak{b}\subset \mathfrak{g}$ is a Cartan subalgebra of $\mathfrak{g}$, and identify $i\mathfrak{b}^*\subset i\mathfrak{g}^*$ utilizing the decomposition
$$\mathfrak{g}=\mathfrak{b}\oplus [\mathfrak{g},\mathfrak{b}].$$
Given $\sigma$, an irreducible, tempered representation of $G$ with regular infinitesimal character, define
$$\lambda_{\sigma,\mathfrak{b}}:=\mathcal{O}_{\sigma}\cap i\mathfrak{b}^*.$$
When $\mathfrak{b}$ is understood, we will often just write $\lambda_{\sigma}$. If $\sigma$ is an irreducible, tempered representation with regular infinitesimal character, then $\lambda_{\sigma}$ is a single $W$ orbit for the real Weyl group $W=W(G,B)=N_G(B)/B$. Here $B:=N_G(\mathfrak{b})\subset G$ is the Cartan subgroup of $G$ with Lie algebra $\mathfrak{b}$.
 
\begin{corollary} \label{cor:Cartan L^2 cor} Let $G$ be a real, reductive algebraic group, and let $X$ be a homogeneous space for $G$ with a nonzero invariant density. Suppose $\mathfrak{b}\subset \mathfrak{g}$ is a Cartan subalgebra of the Lie algebra $\mathfrak{g}$ of $G$. As before, $G_x$ is the stabilizer in $G$ of $x\in X$ and $\mathfrak{g}_x$ is the Lie algebra of $G_x$. In addition, $(i\mathfrak{b}^*)'=i\mathfrak{b}^*\cap (i\mathfrak{g}^*)'$. Then 
\begin{equation}\label{eq:Cartan_L^2_cor}
\op{AC}\left(\bigcup_{\substack{\sigma\in \op{supp}L^2(X)\\ \sigma\in \widehat{G}_{\text{temp}}^{\ \prime}}}\lambda_{\sigma,\mathfrak{b}}\right)\cap (i\mathfrak{b}^*)'=\overline{\left\{\xi\in i\mathfrak{b}^*|\exists\ x\in X\ \text{s.t. }\ \xi|_{\mathfrak{g}_x}=0\ \right\}}\cap (i\mathfrak{b}^*)'.
\end{equation}
The asymptotic cone is taken inside the vector space $i\mathfrak{b}^*$.
\end{corollary}
\begin{proof} We show how to deduce Corollary \ref{cor:Cartan L^2 cor} from Corollary \ref{cor:main}. To show that the right hand side is contained in the left hand side, take $\xi\in (i\mathfrak{b}^*)'$ with $\xi|_{\mathfrak{g}_x}=0$ for some $x\in X$, and choose an open cone $\xi\in \mathcal{C}_1\subset (i\mathfrak{b})^*$ for which $\mathcal{C}_1\subset (i\mathfrak{b}^*)'$. Choose a precompact, open subset $e\in K\subset G$, and consider the open cone $\xi\in \mathcal{C}:=K\cdot \mathcal{C}_1\subset i\mathfrak{g}^*$. Then by Corollary \ref{cor:main}, we deduce 
$$\left(\bigcup_{\substack{\sigma\in \op{supp}L^2(X)\\ \sigma\in \widehat{G}_{\text{temp}}^{\ \prime}}} \mathcal{O}_{\sigma}\right)\cap \mathcal{C}$$
is unbounded. Now, we have a proper, continuous map $K\times \mathcal{C}_1\rightarrow i\mathfrak{g}^*$, and we know that the image of the set 
$$K\times \left(\bigcup_{\substack{\sigma\in \op{supp}L^2(X)\\ \sigma\in \widehat{G}_{\text{temp}}^{\ \prime}}}\lambda_{\sigma,\mathfrak{b}}\right)\cap \mathcal{C}_1$$
is unbounded in $i\mathfrak{g}^*$. We deduce that 
$$\left(\bigcup_{\substack{\sigma\in \op{supp}L^2(X)\\ \sigma\in \widehat{G}_{\text{temp}}^{\ \prime}}}\lambda_{\sigma,\mathfrak{b}}\right)\cap \mathcal{C}_1$$ is unbounded in $i\mathfrak{b}^*$. Therefore, 
$$\xi\in  \op{AC}\left(\bigcup_{\substack{\sigma\in \op{supp}L^2(X)\\ \sigma\in \widehat{G}_{\text{temp}}^{\ \prime}}}\lambda_{\sigma,\mathfrak{b}}\right).$$
If $\xi\in (i\mathfrak{b}^*)'$ is a limit of $\xi_n\in (i\mathfrak{b}^*)'$ with $\xi_n|_{\mathfrak{g}_{x_n}}=0$ each $n$, then every open cone containing $\xi$ must also contain some $\xi_n$. Hence, the required set must intersect this cone in an unbounded set. We have shown that the right hand side of (\ref{eq:Cartan_L^2_cor}) is contained in the left hand side of (\ref{eq:Cartan_L^2_cor}).

Next, we show that the left hand side of (\ref{eq:Cartan_L^2_cor}) is contained in the right hand side of (\ref{eq:Cartan_L^2_cor}). Suppose 
$$\xi\in \op{AC}\left(\bigcup_{\substack{\sigma\in \op{supp}L^2(X)\\ \sigma\in \widehat{G}_{\text{temp}}^{\ \prime}}}\lambda_{\sigma,\mathfrak{b}}\right)\cap (i\mathfrak{b}^*)',$$
and let $\xi\in \mathcal{C}\subset i\mathfrak{g}^*$ be an open cone in $i\mathfrak{g}^*$ containing $\xi$. Without loss of generality, we may assume $\mathcal{C}\subset (i\mathfrak{g}^*)'$. Then $\mathcal{C}_1:=\mathcal{C}\cap i\mathfrak{b}^*$ is an open cone in $i\mathfrak{b}^*$ containing $\xi$. By our assumption on $\xi$, we conclude that 
$$\mathcal{C}_1\cap \op{AC}\left(\bigcup_{\substack{\sigma\in \op{supp}L^2(X)\\ \sigma\in \widehat{G}_{\text{temp}}^{\ \prime}}}\lambda_{\sigma,\mathfrak{b}}\right)$$
is unbounded. Then we deduce that
$$\mathcal{C}\cap \left(\bigcup_{\substack{\sigma\in \op{supp}L^2(X)\\ \sigma\in \widehat{G}_{\text{temp}}^{\ \prime}}}\mathcal{O}_{\sigma}\right)$$
is unbounded. Therefore, 
$$\xi\in \op{AC}\left(\bigcup_{\substack{\sigma\in \op{supp}L^2(X)\\ \sigma\in \widehat{G}_{\text{temp}}^{\ \prime}}}\mathcal{O}_{\sigma}\right).$$
And by Corollary \ref{cor:main}, we deduce that 
$$\xi\in \overline{\bigcup_{x\in X} i(\mathfrak{g}/\mathfrak{g}_x)^*}.$$
Therefore, we can write $\xi=\lim_{n\rightarrow \infty} \xi_n'$ with $\xi_n'|_{\mathfrak{g}_{x_n'}}=0$ for some $x_n'\in X$. Using the finite to one, local diffeomorphism onto its image $G/B\times (i\mathfrak{b}^*)'\rightarrow (i\mathfrak{g}^*)'$, we deduce that we may find $\xi_n\in (i\mathfrak{b}^*)'$ such that $\xi_n'=\operatorname{Ad}_{g_n}^*\xi_n$ for some $g_n\in G$ and $\{\xi_n\}\rightarrow \xi$. If $\op{Ad}_{g_n}x_n'=x_n$, then we note that 
$$\xi_n|_{\mathfrak{g}_{x_n}}=0$$
for all $n$, and we conclude 
$$\xi\in\overline{\left\{\eta\in i\mathfrak{b}^*|\ \eta|_{\mathfrak{g}_x}=0\ \text{some}\ x\in X\right\}}\cap (i\mathfrak{b}^*)'.$$
The Corollary has been verified.
\end{proof}

Next, we give a version of Corollary \ref{cor:Cartan L^2 cor} for vector bundles.

\begin{corollary} \label{cor:Cartan dense semisimple cor} Suppose $G$ is a real, reductive algebraic group, suppose $X$ is a homogeneous space for $G$, let $\mathcal{D}^{1/2}\rightarrow X$ denote the bundle of complex half densities on $X$ (see Appendix A), and suppose $\mathcal{V}\rightarrow X$ is a finite rank, $G$-equivariant, Hermitian vector bundle on $X$. As before, $G_x$ is the stabilizer in $G$ of $x\in X$ and $\mathfrak{g}_x$ is the Lie algebra of $G_x$. Assume that for some (equivalently any) $x\in X$, there is a closed, real, linear algebraic subgroup $\widetilde{G_x}\subset G$ whose Lie algebra is $\mathfrak{g}_x$ (of course, this is satisfied if $G_x$ is itself an algebraic subgroup of $G$) and $\mathfrak{g}_x$ contains a dense subset of semisimple elements. Suppose $\mathfrak{b}\subset \mathfrak{g}$ is a Cartan subalgebra of the Lie algebra $\mathfrak{g}$ of $G$, and let $(i\mathfrak{b}^*)':=i\mathfrak{b}^*\cap (i\mathfrak{g}^*)'$ be the set of regular elements in $i\mathfrak{b}^*$. Then 
\begin{equation}\label{eq:Cartan_dense_semisimple_cor}
\op{AC}\left(\bigcup_{\substack{\sigma\in \op{supp}L^2(X,\mathcal{D}^{1/2}\otimes \mathcal{V})\\ \sigma\in \widehat{G}_{\text{temp}}^{\ \prime}}}\lambda_{\sigma,\mathfrak{b}}\right)\cap (i\mathfrak{b}^*)'=\overline{\left\{\xi\in i\mathfrak{b}^*|\exists\ x\in X\text{ s.t. } \xi|_{\mathfrak{g}_x}=0\right\}}\cap (i\mathfrak{b}^*)'.
\end{equation}
The asymptotic cone is taken inside the vector space $i\mathfrak{b}^*$.
\end{corollary}

The proof of Corollary \ref{cor:Cartan dense semisimple cor} is identical to the proof of Corollary \ref{cor:Cartan L^2 cor}. Of course, one utilizes Corollary \ref{cor:dense semisimple cor} instead of Corollary \ref{cor:main} in the proof.
\bigskip

Before doing several examples, we briefly recall some key facts about Harish-Chandra discrete series representations. If $G$ is a real, reductive algebraic group, then $G$ has at most one (up to $G$ conjugacy) compact Cartan subgroup $T\subset G$. A \emph{Harish-Chandra discrete series} representation of $G$ is an irreducible, unitary representation $\sigma$ with 
$$\op{Hom}_G(\sigma,L^2(G))\neq \{0\}.$$
The group $G$ has discrete series representations if, and only if $G$ has a compact Cartan subgroup $T\subset G$ \cite{HC65b}, \cite{HC66}, \cite{HC70}, \cite{HC76}. Since Harish-Chandra discrete series occur discretely in the unitary dual $\widehat{G}$, whenever $\pi$ is a unitary representation of $G$ and $\sigma\in \op{supp}\pi$, we must have
$$\op{Hom}_G(\sigma,\pi)\neq \{0\}.$$
If $G$ has a compact Cartan subgroup $T\subset G$, then the Lie algebras of compact Cartan subgroups $\mathfrak{t}\subset \mathfrak{g}$ are called \emph{fundamental Cartan subaglebras}. The \emph{regular elliptic} elements in $\mathfrak{g}$ (resp. $i\mathfrak{g}^*$) are precisely the regular, semisimple elements of $\mathfrak{g}$ (resp. $i\mathfrak{g}^*$) meeting $\mathfrak{t}$ (resp. $i\mathfrak{t}^*$). Further, if $\sigma$ is an irreducible, tempered representation of $G$, then $\mathcal{O}_{\sigma}\cap (i\mathfrak{t}^*)'\neq \{0\}$ if and only if $\sigma$ is a Harish-Chandra discrete series representation of $G$.


\begin{example} Let us consider the class of examples $G=\op{Sp}(2n,\mathbb{R})$ and $H=\op{Sp}(2l,\mathbb{Z})\times \op{Sp}(2m,\mathbb{R})$ that is mentioned in the introduction. Let $\mathfrak{g}=\mathfrak{sp}(2n,\mathbb{R})$ and $\mathfrak{h}=\op{sp}(2m,\mathbb{R})$ denote the Lie algebras of $G$ and $H$. Viewing $\mathfrak{g}$ as a set of $2n$ by $2n$ matrices and $\mathfrak{h}$ as a subset of $\mathfrak{g}$ of matrices with at least $4n^2-4m^2$ zeroes, one can define the complementary subspace $\mathfrak{q}$ consisting of matrices having zeroes in precisely the entries where elements of $\mathfrak{h}$ can have nonzero entries. One notes that $\mathfrak{q}$ is orthogonal to $\mathfrak{h}$ under the Killing form, $B$, of $\mathfrak{g}$. In particular, under the isomorphism $i\mathfrak{g}^*\simeq \mathfrak{g}$, which involves dividing by $i$ and using the Killing form, one obtains
$$i(\mathfrak{g}/\mathfrak{h})^*\simeq \mathfrak{q}.$$
In particular, utilizing Corollary \ref{cor:Cartan L^2 cor}, we see that if there exists $\xi\in \mathfrak{q}$ that is a regular elliptic element of $\mathfrak{g}$, then there must exist infinitely many Harish-Chandra discrete series representations $\sigma$ of $G=\op{Sp}(2n,\mathbb{R})$ such that
$$\op{Hom}_{\op{Sp}(2n,\mathbb{R})}(\sigma,L^2(X_{l,m,n}))\neq \{0\}$$
where
\[X_{l,m,n}:=\operatorname{Sp}(2n,\mathbb{R})/[\operatorname{Sp}(2l,\mathbb{Z})\times \operatorname{Sp}(2m,\mathbb{R})].\]
Utilizing a bit of linear algebra, one readily checks that this is the case precisely when $2m\leq n$. In the special case $m=0$, much stronger results are already known (see Proposition 10.5 on pages 117-118 of \cite{KK16}). When $2m>n$, we note that there are no regular elliptic elements in $\mathfrak{q}$ (in fact, $\mathfrak{q}$ has no regular elements at all in this case). In this case, one might like to deduce that $L^2(X_{l,m,n})$ has no Harish-Chandra discrete series of $\op{Sp}(2n,\mathbb{R})$ occurring in its Plancherel formula. Unfortunately, our results are not that powerful. Instead, let $\mathfrak{t}\subset \op{sp}(2n,\mathbb{R})$ be a compact Cartan subalgebra of $\mathfrak{g}$. Applying Corollary \ref{cor:Cartan L^2 cor}, we obtain
\[
\op{AC}\left(\bigcup_{\substack{\sigma\in \op{supp}L^2(X_{l,m,n})\\ \sigma\in \widehat{\op{Sp}(2n,\mathbb{R})}_{\text{temp}}^{\ \prime}}}\lambda_{\sigma,\mathfrak{t}}\right)\cap (i\mathfrak{t}^*)'=\{0\} 
\]
whenever $2m>n$. Again, this does not show that $L^2(X_{l,m,n})$ has no Harish-Chandra discrete series representations when $2m>n$. The above statement would still be true if there were finitely many Harish-Chandra discrete series in the above Plancherel formula or even if there were infinitely many Harish-Chandra discrete series in the Plancherel formula but their parameters were ``bunched up near the singular set'' in a certain way.
\end{example}

\begin{example}
Next, we give an example of Corollary \ref{cor:dense semisimple cor}. Suppose $G$ is a real, reductive algebraic group, $H\subset G$ is a real, reductive algebraic subgroup, and $P\subset H$ is a parabolic subgroup of $H$. Let $P=MAN$ be a Langlands decomposition of $P$, let $(\chi,\mathbb{C}_{\chi})$ be a unitary character of $MA$ extended trivially on $N$ to a character of $P$, and let
$$\mathcal{L}_{\chi}:=G\times_P \mathbb{C}_{\chi}$$
be the corresponding $G$ equivariant, Hermitian vector bundle on $G/P$. Then by Corollary \ref{cor:dense semisimple cor}, the set
\[
\op{AC}\left(\bigcup_{\substack{\sigma\in L^2(G/P,\mathcal{D}^{1/2}\otimes \mathcal{L}_{\chi})\\ \sigma \in \widehat{G}_{\text{temp}}^{\ \prime}}}\mathcal{O}_{\sigma}\right)\cap (i\mathfrak{g}^*)'
\]
is independent of the unitary character $\chi$ of $P$. Consider the case where $G_1=\op{SL}(2,\mathbb{R})$, $G=\op{SL}(2,\mathbb{R})\times \op{SL}(2,\mathbb{R})$ and $B\subset G_1=\op{SL}(2,\mathbb{R})$ is a Borel subgroup embedded diagonally in $G$.  Let us write $\sigma_{n}^+$ (resp. $\sigma_{n}^-$) for the holomorphic (resp. antiholomorphic) discrete series with parameter $n=1,2,\ldots$. Let us write $\sigma_{\nu,+}$ (resp. $\sigma_{\nu,-}$) for the spherical (resp. non-spherical) unitary principal series with parameter $\nu\in \mathbb{R}_{\geq 0}$. Observe 
\[
\overline{\op{Ad}^*(G)\cdot i(\mathfrak{g}/\Delta(\mathfrak{b}))^*}=\overline{\{(\xi_1,\xi_2)\in \mathfrak{g}_1^{\oplus 2}|\ \op{Ad}^*(g_1)\xi_1+\op{Ad}^*(g_2)\xi_2\in \mathcal{N}\ \text{some}\ g_1,g_2\in G_1\}}
\]
where $\mathcal{N}\subset \op{sl}(2,\mathbb{R})$ denotes the nilpotent cone. Analyzing this set together with Corollary \ref{cor:dense semisimple cor} and Corollary \ref{cor:Cartan dense semisimple cor}, one arrives at several conclusions regarding 
\[
L^2(G/\Delta(B),\mathcal{D}^{1/2}\otimes \mathcal{L}_{\chi}),
\]
all of which are independent of the character $\chi$ of $B$:

\begin{itemize}
\item There exists a natural number $m$ such that whenever $(m_1,m_2)\in \mathbb{N}\times \mathbb{N}$ and
$$\op{Hom}(\sigma_{m_1}^+\otimes \sigma_{m_2}^+,L^2(G/\Delta(B),\mathcal{D}^{1/2}\otimes \mathcal{L}_{\chi}))\neq \{0\}$$
we must have $m_j\leq m$ for some $j$.
\noindent (The analogous statement holds for $\sigma_{m_1}^-\otimes \sigma_{m_2}^-$).
\item There are infinitely many pairs $(m_1,m_2)\in \mathbb{N}\times \mathbb{N}$ for which
$$\op{Hom}(\sigma_{m_1}^+\otimes \sigma_{m_2}^-,L^2(G/\Delta(B),\mathcal{D}^{1/2}\otimes \mathcal{L}_{\chi})))\neq \{0\}.$$
In fact, for any open cone $\Gamma\subset \overline{\Gamma}\subset \mathbb{R}_{>0}^2$, there
are infinitely many such pairs $(m_1,m_2)\in \Gamma$.
\noindent (The analogous statement holds for $\sigma_{m_1}^-\otimes \sigma_{m_2}^+$).
\item The collection of $(m,\nu)\in \mathbb{N}\times \mathbb{R}_{\geq 0}\subset \mathbb{R}^2$ such that
$$\sigma_m^{\pm}\otimes \sigma_{\nu,\pm}\in \op{supp}L^2(G/\Delta(B),\mathcal{D}^{1/2}\otimes \mathcal{L}_{\chi})$$
is unbounded in every direction in $\mathbb{R}_{\geq 0}^2$.
\noindent (The analogous statement holds for $\sigma_{\nu,\pm}\otimes \sigma_m^{\pm}$).
\item For every $\epsilon>0$, the collection of $(\nu_1,\nu_2)\in \mathbb{R}_{\geq 0}^2\subset \mathbb{R}^2$
for which 
$$\sigma_{\nu_1,\pm}\otimes \sigma_{\nu_2,\pm}\in \op{supp}L^2(G/\Delta(B),\mathcal{D}^{1/2}\otimes \mathcal{L}_{\chi})$$
is unbounded in every direction in $\mathbb{R}^2_{\geq 0}$.
\end{itemize}

In particular, we see that 
$$\op{supp}L^2([\op{SL}(2,\mathbb{R})\times \op{SL}(2,\mathbb{R})]/\Delta(B),\mathcal{D}^{1/2}\otimes \mathcal{L}_{\chi})$$
is much larger than 
$$\op{supp}L^2([\op{SL}(2,\mathbb{R})\times \op{SL}(2,\mathbb{R})]/\Delta(\op{SL}(2,\mathbb{R})))$$
for every unitary character $\chi$ of $B$.
\end{example}

\begin{example}
We end the section with a family of examples related to an interesting example of Kobayashi. In Theorem 6.2 of \cite{Ko98a}, Kobayashi considers the group $G=\op{O}(p,q)$ and the subgroup $H=U(r,s)$ with $2r=p$ and $2s\leq q$. He assumes that $p$ is positive and divisible by four. Under these assumptions, Kobayashi shows that there exist infinitely many distinct discrete series $\sigma$ of $G=\op{O}(p,q)$ for which 
$$\op{Hom}_G(\sigma,L^2(G/H))\neq \{0\}.$$
We note that the techniques of Kobayashi in \cite{Ko98a} are very different from our own techniques. In particular, he utilizes his work on the theory of discretely decomposable restrictions (\cite{Ko94}, \cite{Ko98b}, \cite{Ko98c}) together with his work on the decay of functions on certain types of homogeneous spaces \cite{Ko97} and the classification of the discrete spectrum of reductive symmetric spaces (\cite{FJ80}, \cite{MO84}; see also the exposition \cite{Vo88}). 

Since our work utilizes very different ideas, it is worth considering what we can show. Let $G=\op{SO}(p,q)$ and $H=\op{U}(r,s)$. Assume 
\[2r+1\leq p\ \text{or}\ 2r=p\ \text{and}\ 4|p\]
and
\[2s+1\leq q\ \text{or}\ 2s=q\ \text{and}\ 4|q.\]
In addition, assume that at least one of $p$ and $q$ is even (This condition is necessary so that Harish-Chandra discrete series of $G=\op{SO}(p,q)$ exist). Let $\mathcal{L}\rightarrow G/H$ be any (possibly trivial) finite rank, $G$-equivariant, Hermitian line bundle on $G/H$. It follows from Corollary \ref{cor:dense semisimple cor} that there exist infinitely many distinct Harish-Chandra discrete series $\sigma$ of $G=\op{SO}(p,q)$ for which 
\begin{equation}\label{eq:Kobayashi_example}
\op{Hom}_G(\sigma,L^2(G/H,\mathcal{L}))\neq \{0\}.
\end{equation}

In some ways, our results are more general. We do not assume $2r=p$, we do not assume at least one of $p$ or $q$ is divisible by four, and we consider bundle-valued harmonic analysis. However, Kobayashi's construction of Harish-Chandra discrete series is more explicit, and there are a few cases where his methods show existence of Harish-Chandra discrete series and our methods do not. The simplest example is $G=\op{SO}(4,2)$, $H=U(2,1)$. In this case, if $X=G/H$, then
\begin{equation}\label{eq:singular_momentum}
\mu(iT^*X)\subset i\mathfrak{g}^*\setminus (i\mathfrak{g}^*)'.
\end{equation}
That is the image of the momentum map lies in the singular set. Corollary \ref{cor:Cartan dense semisimple cor} then implies that the asymptotic cone of the Harish-Chandra parameters of the Harish-Chandra discrete series of $G=\op{SO}(4,2)$ occurring in $L^2(X)=L^2(\op{SO}(4,2)/U(2,1))$ lies in the singular set. In particular, the discrete spectrum of $L^2(\op{SO}(4,2)/U(2,1))$ is less robust than the discrete spectrum of $L^2(\op{SO}(4,2)/U(1,1))$ or $L^2(\op{SO}(4,2)/U(2))$. This example shows how it would be useful to strengthen Theorem \ref{thm:main} so that one intersects both sides with a larger set than $(i\mathfrak{g}^*)'$ in order to compute singular asymptotics. This would require generalizing the main results of \cite{HHO16} and \cite{Ha}.

In order to check that (\ref{eq:Kobayashi_example}) follows from Corollary \ref{cor:dense semisimple cor}, one first checks that
\[i(\mathfrak{o}(4)/\mathfrak{u}(2))^*\subset i\mathfrak{o}(4)^*\ \text{and}\ i(\mathfrak{o}(3)/\mathfrak{u}(1))^*\subset i\mathfrak{o}(3)^*\]
both contain regular, elliptic elements. Then one embeds $r/2$ copies of 
$$i(\mathfrak{o}(4)/\mathfrak{u}(2))^*\subset i\mathfrak{o}(4)^*$$
into $i(\mathfrak{o}(2r)/\mathfrak{u}(r))^*$ if $r$ even to deduce that $i(\mathfrak{o}(2r)/\mathfrak{u}(r))^*$ contains regular, elliptic elements. One then extends these regular, elliptic elements to regular, elliptic elements of $i(\mathfrak{o}(p)/\mathfrak{u}(r))^*$ using that $2r\leq p$. If $r$ odd, then one embeds $(r-1)/2$ copies of $i(\mathfrak{o}(4)/\mathfrak{u}(2))^*\subset i\mathfrak{o}(4)^*$ and one copy of $i(\mathfrak{o}(3)/\mathfrak{u}(1))^*\subset i\mathfrak{o}(3)^*$ into 
\[i(\mathfrak{o}(2r+1)/\mathfrak{u}(r))^*\]
 to show that $i(\mathfrak{o}(2r+1)/\mathfrak{u}(r))^*$ contains regular, elliptic elements. And one embeds $(r+1)/2$ copies of $i(\mathfrak{o}(4)/\mathfrak{u}(2))^*\subset i\mathfrak{o}(4)^*$ into $i(\mathfrak{o}(2r+2)/\mathfrak{u}(r))^*$ to show that 
\[i(\mathfrak{o}(2r+2)/\mathfrak{u}(r))^*\]
contains regular elliptic elements. Using one of these two statements and extending regular, elliptic elements to regular, elliptic elements, one deduces that 
\[i(\mathfrak{o}(p)/\mathfrak{u}(r))^*\] 
contains regular, elliptic elements if $r$ odd and $2r+1\leq p$. Identical statements hold when $p$ is replaced by $q$ and $r$ is replaced by $s$. Putting these together, we embed
\[i(\mathfrak{o}(p)/\mathfrak{u}(r))^*\times i(\mathfrak{o}(q)/\mathfrak{u}(s))^*\hookrightarrow i(\mathfrak{o}(p,q)/\mathfrak{u}(r,s))^*.\]
If at least one of $p$ and $q$ is even, then regular, elliptic elements map to regular, elliptic elements. The claim (\ref{eq:Kobayashi_example}) follows.

\end{example}

\section{Counterexamples and Whittaker Functionals}
\label{sec:counterexamples}

Let $G$ be a real, reductive algebraic group, let $N\subset G$ be a unipotent subgroup, and let $(\chi,\mathbb{C}_{\chi})$ denote a unitary character of $N$. If $(\sigma,W_{\sigma})$ is a unitary representation of $G$ with smooth vectors $W_{\sigma}^{\infty}$, then a \emph{distribution Whittaker functional} on $(\sigma,W_{\sigma})$ with respect to $(N,\chi)$ is a continuous, $N$-equivariant homomorphism
$$\psi\colon W_{\sigma}^{\infty}\longrightarrow \mathbb{C}_{\chi}.$$
We denote the vector space of such homomorphisms by $\op{Wh}_{N,\chi}(\sigma)$. Distribution Whittaker functionals have primarily been studied in the special case where $N$ is the nilradical of a parabolic subgroup $P\subset G$. In the case where $(\sigma,W_{\sigma})$ is an irreducible, tempered representation of $G$, the study of distribution Whittaker functionals was related to harmonic analysis by Harish-Chandra (unpublished) and Wallach \cite{Wa92}. Let us write down this relationship.

We may associate a Hermitian line bundle 
$$\mathcal{L}_{\chi}=G\times_N \mathbb{C}_{\chi}\longrightarrow G/N$$
 to the unitary character $\chi$ of $N$, and we may consider the unitary representation $L^2(G/N,\mathcal{L}_{\chi})$ of $G$. First, by the Lemma on page 365 of \cite{Wa92}, we have
$$\op{supp}L^2(G/N,\mathcal{L}_{\chi})\subset \widehat{G}_{\text{temp}}.$$
That is, the decomposition of $L^2(G/N,\mathcal{L}_{\chi})$ into irreducibles consists entirely of irreducible, tempered representations. Next, the multiplicity of an irreducible, tempered representation $\sigma$ in $L^2(G/N,\mathcal{L}_{\chi})$ is equal to the dimension of the space of distribution Whittaker functionals for $\sigma$ with respect to the pair $(N,\chi)$ (see the Theorem on page 425 of \cite{Wa92}).

Now, the space $\op{Wh}_{N,\chi}(\sigma)$ is not completely understood in general. However, some partial results exist in special cases (\cite{Ko78}, \cite{Ma92}, \cite{GS15}). For instance, consider the case where $N$ is the nilradical of a minimal parabolic subgroup $P=MAN$.
Note that $d\chi$ is trivial on the commutator algebra $[\mathfrak{n},\mathfrak{n}]$, and therefore it descends to a linear functional
$$d\chi\in i(\mathfrak{n}/[\mathfrak{n},\mathfrak{n}])^*.$$
We say $\chi$ is nondegenerate if $d\chi$ is contained in an open $MA$ orbit in $i(\mathfrak{n}/[\mathfrak{n},\mathfrak{n}])^*$. Matumoto proved a nice result under these conditions \cite{Ma92}.

\begin{theorem}[Matumoto]
Let $G$ be a real, reductive algebraic group, let $P$ be a minimal parabolic subgroup with nilradical $N$ and Langlands decomposition $P=MAN$, and let $\chi$ be a nondegerate character of $N$. Then there exists a distribution Whittaker functional for an irreducible, unitary representation $\sigma$ with respect to $(N,\chi)$ if and only if
$$d\chi\in \operatorname{WF}(\sigma).$$
Here we identify $d\chi \in i\mathfrak{n}^*\subset i\mathfrak{g}^*$ in the usual way.
\end{theorem}

Combining Matumoto's Theorem with the above result of Harish-Chandra (and independently Wallach), one can determine $\op{supp}L^2(G/N,\mathcal{L}_{\chi})$ whenever $G$ is a real, reductive algebraic group, $N\subset G$ is the nilradical of a minimal parabolic subgroup, and $\chi$ is a nondegenerate character of $N$. Combining this information with Theorem 1.2 of \cite{HHO16} on wave front sets of arbitrary direct integrals of tempered representations, one checks that in many cases
$$\op{WF}(\op{Ind}_N^G\chi)\supsetneq \op{Ind}_N^G\op{WF}(\chi)=\overline{\bigcup_{g\in G} \op{Ad}^*(g)\cdot i(\mathfrak{g}/\mathfrak{n})^*}.$$
Let us write down the simplest example. Let $G=\operatorname{SL}(2,\mathbb{R})$, and let
$$N=\left\{\left(\begin{matrix} 1 & x\\ 0 & 1\end{matrix}\right)\Big|\ x\in \mathbb{R}\right\}.$$
The unitary characters of $N$ are parametrized by $i\mathbb{R}$; let us write
$$\chi_{\lambda}\left(\begin{matrix} 1 & x\\ 0 & 1\end{matrix}\right)=e^{\lambda x}$$
for $\lambda\in i\mathbb{R}$. Now, we may form the Hermitian line bundle
$$\mathcal{L}_{\lambda}=G\times_{N} \chi_{\lambda}$$
and the unitary representation $$L^2(G/N,\mathcal{L}_{\lambda})$$ of $G=\op{SL}(2,\mathbb{R})$ for every $\lambda\in i\mathbb{R}$. 
\bigskip

To consider wave front sets, let $\mathfrak{g}=\mathfrak{sl}(2,\mathbb{R})$ denote the Lie algebra of $G=\operatorname{SL}(2,\mathbb{R})$. Introduce coordinates on the Lie algebra 
$$\mathfrak{g}=\left\{X_{x,y,z}=\left(\begin{matrix} x & y-z\\ y+z & -x\end{matrix}\right)\Big|\ x,y,z\in \mathbb{R}\right\}.$$
Notice that the $G$ orbits on $\mathfrak{g}$ are the hyperboloids 
$$x^2+y^2-z^2=c$$
for $c>0$, half of this hyperboloid when $c<0$, and one of three pieces of the cone when $c=0$. Notice that when $x^2+y^2-z^2>0$, the matrix $X_{x,y,z}$ is diagonalizable with real eigenvalues; we call such an element of $\mathfrak{g}$ \emph{hyperbolic} and we denote the set of hyperbolic elements in $\mathfrak{g}$ by $\mathfrak{g}_{\text{hyp}}$. When $x^2+y^2-z^2<0$, the matrix $X_{x,y,z}$ is diagonalizable with purely imaginary eigenvalues; we call such an element of $\mathfrak{g}$ \emph{elliptic} and we denote the set of elliptic elements in $\mathfrak{g}$ by $\mathfrak{g}_{\text{ell}}$. The set of nonzero elliptic elements has two connected components which we denote by $\mathfrak{g}_{\text{ell}}^+$ (the set of nonzero elliptic elements $X_{x,y,z}$ with $z>0$) and $\mathfrak{g}_{\text{ell}}^-$ (the set of nonzero elliptic elements $X_{x,y,z}$ with $z<0$). When $x^2+y^2-z^2=0$, the matrix $X_{x,y,z}$ is nilpotent, and we call such an element of $\mathfrak{g}$ \emph{nilpotent}. We denote by $\mathfrak{g}_{\text{nilp}}$ the set of nilpotent elements in $\mathfrak{g}$.
\bigskip

We may identify $\mathfrak{g}\simeq \mathfrak{g}^*$ via the trace form 
$$X\mapsto \left(Y\mapsto \operatorname{Tr}(XY)\right).$$
This isomorphism is $G$ equivariant (so it takes $G$ orbits on $\mathfrak{g}$ to $G$ orbits on $\mathfrak{g}^*$). After dividing by $i$, we obtain a $G$ invariant identification of $\mathfrak{g}$ with $i\mathfrak{g}^*$. We denote by $i\mathfrak{g}^*_{\text{hyp}}$ (resp. $i\mathfrak{g}^*_{\text{ell}}$, $i(\mathfrak{g}_{\text{ell}}^*)^+$, $i(\mathfrak{g}_{\text{ell}}^*)^-$, $i\mathfrak{g}^*_{\text{nilp}}$) the subset of $i\mathfrak{g}^*$ which corresponds under the above isomorphism to $\mathfrak{g}_{\text{hyp}}$ (resp. $\mathfrak{g}_{\text{ell}}$, $\mathfrak{g}_{\text{ell}}^+$, $\mathfrak{g}_{\text{ell}}^-$, $\mathfrak{g}_{\text{nilp}}$).
\bigskip

By Theorem 1.1, we know
$$\operatorname{WF}(L^2(G/N))=\overline{\operatorname{Ad}^*(G)\cdot i(\mathfrak{g}/\mathfrak{n})^*}.$$
Under the identification, $i\mathfrak{g}^*\cong \mathfrak{g}$, one notes that $i(\mathfrak{g}/\mathfrak{n})^*$ corresponds to 
$$\mathfrak{b}=\left\{\left(\begin{matrix} a & x\\ 0 & -a\end{matrix}\right)\Big|\ a,x\in \mathbb{R}\right\}.$$
Now, all of the elements in $\mathfrak{b}$ are either hyperbolic or nilpotent, and all hyperbolic or nilpotent elements in $\mathfrak{g}$ are conjugate to elements in $\mathfrak{b}$. Thus, if we break up 
$$i\mathfrak{g}^*= i\mathfrak{g}^*_{\text{hyp}}\cup i\mathfrak{g}^*_{\text{nilp}}\cup i\mathfrak{g}^*_{\text{ell}},$$
then 
$$\operatorname{WF}(L^2(G/N))=i\mathfrak{g}^*_{\text{hyp}}\cup i\mathfrak{g}^*_{\text{nilp}}.$$

Let us consider the representation theory side for a moment. As we stated before, all of the irreducible representations of $G=\op{SL}(2,\mathbb{R})$ occurring in the direct integral decomposition of $L^2(G/N,\mathcal{L}_{\lambda})$ also occur in $L^2(G)$, that is, they are tempered. The irreducible, tempered representations of $G$ are as follows. There are the holomorphic discrete series $\sigma_n^+$ for $n\in \mathbb{N}$, the antiholomorphic discrete series $\sigma_n^-$ for $n\in \mathbb{N}$, the spherical unitary principal series $\sigma_{\nu,+}$ for $\nu\in \mathbb{R}_{\geq 0}$, the non-spherical unitary principal series $\sigma_{\nu,-}$ for $\nu\in \mathbb{R}_{>0}$, and the two limits of discrete series $\sigma_0^+$ and $\sigma_0^-$. We have the direct integral decomposition

$$L^2(G/N)\simeq \int^{\oplus}_{\nu\in i\mathbb{R}_{\geq 0}} \sigma_{\nu,+}d\nu\bigoplus \int^{\oplus}_{\nu\in i\mathbb{R}_{>0}}\sigma_{\nu,-}d\nu.$$
Now, one computes from Theorem 1.2 of \cite{HHO16} that
$$\operatorname{WF}\left(\int^{\oplus}_{\nu\in i\mathbb{R}_{\geq 0}} \sigma_{\nu,+}d\nu\bigoplus \int^{\oplus}_{\nu\in i\mathbb{R}_{>0}}\sigma_{\nu,-}d\nu\right)=\overline{i\mathfrak{g}^*_{\text{hyp}}}=i\mathfrak{g}^*_{\text{hyp}}\cup i\mathfrak{g}^*_{\text{nilp}}.$$ 
Of course, we observe that computing the wave front set from the $L^2$ side and the representation theory side yield the same thing. Now, let us consider the more general case $L^2(G/N,\mathcal{L}_{\lambda})$ for $\lambda\neq 0$. We cannot compute the wave front set from the $L^2$ side using Corollary \ref{cor:dense semisimple cor} because $\mathfrak{n}=\operatorname{Lie}(N)$ does not contain a dense subset of semisimple elements (in fact all of the elements in $\mathfrak{n}$ are nilpotent). But, we can still look at the representation theory side. Break up
$$i\mathfrak{n}^*=i\mathfrak{n}^*_+\cup \{0\}\cup i\mathfrak{n}^*_-$$
so that $$i\mathfrak{n}^*_+\subset \overline{i(\mathfrak{g}_{\text{ell}}^*)^+},\ i\mathfrak{n}^*_-\subset \overline{i(\mathfrak{g}_{\text{ell}}^*)^-}.$$

Utilizing the work of Matumoto \cite{Ma92}, Harish-Chandra (unpublished), and Wallach \cite{Wa92} together with knowledge of the wave front sets of irreducible, tempered representations of $\op{SL}(2,\mathbb{R})$ (this knowledge can be derived from work of Rossmann \cite{Ro78}, \cite{Ro80}, \cite{Ro95} and Barbasch-Vogan \cite{BV}; see Section 8.1 of \cite{HHO16} where this example is worked out in detail), we have

$$L^2(G/N,\mathcal{L}_{\lambda})\simeq \int_{\nu\in \mathbb{R}_{\geq 0}} \sigma_{\nu,+} d\nu\oplus \int_{\nu\in \mathbb{R}_{> 0}} \sigma_{\nu,-}d\nu\oplus \sum_{n\in \mathbb{N}} \sigma_n^+.$$ 
if $d\chi_{\lambda}\in i\mathfrak{n}^*_+$ (the fact that the multiplicities are one was shown by Kostant \cite{Ko78}). The analogous formula holds if $d\chi_{\lambda}\in i\mathfrak{n}^*_-$ with $+$ and $-$ swapped. 
\bigskip

Now, applying Theorem 1.2 of \cite{HHO16}, we obtain 
$$\operatorname{WF}(L^2(G/N,\mathcal{L}_{\lambda}))=i\mathfrak{g}^*_{\text{hyp}}\cup i(\mathfrak{g}^*_{\text{ell}})^+\cup i\mathfrak{g}^*_{\text{nilp}}.$$
if $d\chi_{\lambda}\in i\mathfrak{n}^*_+$. Similarly, we obtain 
$$\operatorname{WF}(L^2(G/N,\mathcal{L}_{\lambda}))=i\mathfrak{g}^*_{\text{hyp}}\cup i(\mathfrak{g}^*_{\text{ell}})^-\cup i\mathfrak{g}^*_{\text{nilp}}.$$
if $d\chi_{\lambda}\in i\mathfrak{n}^*_-$. 

In particular, $$\operatorname{WF}(\operatorname{Ind}_N^G \chi_{\lambda})= i\mathfrak{g}^*_{\text{hyp}}\cup i(\mathfrak{g}^*_{\text{ell}})^+\cup i\mathfrak{g}^*_{\text{nilp}}\supsetneq i\mathfrak{g}^*_{\text{hyp}}\cup i\mathfrak{g}^*_{\text{nilp}}$$
$$ =\overline{\operatorname{Ad}^*(G)\cdot i(\mathfrak{g}/\mathfrak{n})^*}=\op{Ind}_N^G \op{WF}(\chi_{\lambda})$$
when $d\chi_{\lambda}\in i\mathfrak{n}^*_+$ and 
$$\operatorname{WF}(\operatorname{Ind}_N^G \chi_{\lambda})= i\mathfrak{g}^*_{\text{hyp}}\cup i(\mathfrak{g}^*_{\text{ell}})^-\cup i\mathfrak{g}^*_{\text{nilp}}\supsetneq i\mathfrak{g}^*_{\text{hyp}}\cup i\mathfrak{g}^*_{\text{nilp}}$$
$$ =\overline{\operatorname{Ad}^*(G)\cdot i(\mathfrak{g}/\mathfrak{n})^*}=\op{Ind}_N^G \op{WF}(\chi_{\lambda})$$
when $d\chi_{\lambda}\in i\mathfrak{n}^*_-$.

Hence, we have a counterexample to the converse to Theorem 1.1 of \cite{HHO16}, and we have demonstrated the necessity of at least some hypothesis in Theorem 1.2 that does not exist in Theorem 1.1.

\appendix
\section{Density bundles}\label{app:dense}
Let $X$ be an $n$-dimensional smooth manifold and let $\mathcal F X\to X$ be the 
frame bundle whose fibers over the base point $x\in X$ consist of all ordered 
bases of the $n$-dimensional vector space $T_x X$. Note that a choice of such a 
basis is equivalent to a choice of a linear isomorphism $b: \mathbb R^n\to T_x X$
and we will denotes points in $\mathcal F X$ by two-tuples $(x,b)$. Given any 
$g\in \operatorname{GL}(n,\mathbb R)$ we can define canonically its right action on
$\mathcal FX$ by the pullback of the isomorphism $b: \mathbb R^n\to T_x X$ with $g$
\[
 (x,b)g:=(x,b\circ g).
\]
This action is free and transitive on the fibers, so $\mathcal FX$ is a principle 
$\operatorname{GL}(n,\mathbb R)$ fiber bundle.

If $U\subset X$ and $V\subset \mathbb R^ n$ are open and 
\[
 \kappa:U\subset X \to V\subset \mathbb R^ n
\]
is a smooth chart, then this chart naturally gives rise to a local section of 
the frame bundle defined by
\begin{equation} \label{eq:frame_bundle_section}
\tau_\kappa:U\to \pi^{-1}_{\mathcal FX}(U),~x\mapsto (x,d\kappa^ {-1}_{|\kappa(x)}). 
\end{equation}

For $\alpha >0$ the map $\operatorname{GL}(n,\mathbb R)\to \operatorname{End}(\mathbb C), g\mapsto |\det(g)|^ {-\alpha}$ 
is a one dimensional representation and we can define the density bundles as the
associated fiber bundles with respect to these representations.
\begin{definition}
For a smooth $n$-dimensional manifold $X$ and $\alpha>0$ we define the 
$\alpha$-density bundle over $X$ as
\[
 \mathcal D^ \alpha := \mathcal F X\times_{|\det(\bullet)|^{-\alpha}} \mathbb C.
\]
As $\mathbb R_{\geq 0}\subset \mathbb C$ is invariant under the action by $|\det(\bullet)|^{-\alpha}$
we can also define the positive $\alpha$ density bundle as
\[
 \mathcal D_{\geq 0}^ \alpha:= \mathcal F X\times_{|\det(\bullet)|^{-\alpha}} \mathbb R_{\geq 0}.
\]
\end{definition}
We will denote elements in the density bundles by equivalence classes $[(x,b),z]\in \mathcal D^ \alpha $
where $(x,b)\in \mathcal F X$, $z\in\mathbb C$ and the equivalence relation is 
given by $((x,b\circ g),z)\sim ((x,b),|\det(g)|^{-\alpha} z)$. 

Note that the density bundles behave nicely under tensor products as we have 
for $\alpha,\beta>0$
\begin{equation}
 \mathcal D^\alpha \otimes \mathcal D^ \beta  \cong \mathcal D^{\alpha+\beta}. 
\end{equation}

Moreover, there is a global absolute value map
\begin{equation}\label{eq:def_abs_global}
 |\quad|_{\mathcal D^\alpha}: \Gamma(\mathcal D^\alpha ) \to\Gamma(D^\alpha_{\geq 0} )
\end{equation}
given by
$$[(x,b),z]\mapsto [(x,b),|z|].$$

Sometimes it is useful to work in coordinates. Given a smooth chart 
$\kappa:U\to V$ of $X$ and using the local sections $\tau_\kappa$
defined in (\ref{eq:frame_bundle_section}) we obtain a local trivialization of 
the density bundles and can thus locally identify sections 
$\Psi: U\to \pi^ {-1}_{\mathcal D^ \alpha  }(U)$
with a function $\Psi_\kappa :V\to \mathbb C$ by the condition
\[
 \Psi(m) = [\tau_\kappa(m) ,\Psi_\kappa(\kappa(m))]
\]
which determines the function $\Psi_\kappa$ uniquely as the right $\operatorname{GL}(n,\mathbb R)$
action on the fiber is free.

If $\kappa':U\to V'$ is another chart then
\begin{eqnarray*}
 \Psi(m) &=& [\tau_\kappa(m), \Psi_\kappa(\kappa(m))] \\
      &=& [(m, d\kappa^ {-1}_{|\kappa(m)}), \Psi_\kappa(\kappa(m))] \\
      &=& [(m, d\kappa'^ {-1}_{|\kappa'(m)}\circ d(\kappa'\circ\kappa^ {-1})_{|\kappa(m)} ), \Psi_\kappa(\kappa(m))] \\
      &=& [\tau_{\kappa'}(m), |\det(d(\kappa'\circ\kappa^ {-1})_{|\kappa(m)} )|^{-\alpha} \cdot\Psi_\kappa(\kappa\circ\kappa'^ {-1}(\kappa'(m))] .
\end{eqnarray*}
Consequently we have for $y\in V'$
\begin{eqnarray}
 \Psi_{\kappa'}(y) &=& |\det(d(\kappa'\circ\kappa^ {-1})_{|\kappa\circ\kappa'^ {-1}(y)} )|^{-\alpha} \cdot\Psi_\kappa(\kappa\circ\kappa'^ {-1}(y)) \nonumber \\
 &=&|\det(d(\kappa\circ\kappa'^ {-1})_{|y} )|^{\alpha} \cdot\Psi_\kappa(\kappa\circ\kappa'^ {-1}(y)). \label{eq:coordinate_transform}
\end{eqnarray}
Note that the same construction associates sections $\rho$ of the positive density 
bundle $\mathcal D^\alpha_{\geq 0}  $ to functions $\rho_\kappa:V\to\mathbb R_{\geq 0}$
and the transformation with respect to coordinate change is also according to (\ref{eq:coordinate_transform})..
We can thus give an alternate definition of the global absolute value map for sections 
on the density bundle
$|\quad|_{\mathcal D^\alpha}: \Gamma(\mathcal D^\alpha ) \to\Gamma(D^\alpha_{\geq 0} )$
by requiring locally for a chart $\kappa:U\to V$ and $x\in V$ that
\[
 (|\Psi|_{\mathcal D^ \alpha})_\kappa (x) := |\Psi_\kappa(x)|.
\]
Note that (\ref{eq:coordinate_transform}) assures that this definition is chart 
independent. It is easy to see that this coordinate definition agrees with the 
definition (\ref{eq:def_abs_global}).

Given a chart $\kappa:U\to V$ and a section $\Psi \in \Gamma(\mathcal D^1 )$ compactly 
supported in $U$, we say that $\Psi$ is \emph{integrable} if and only if
$\Psi_\kappa$ is Lebesgue integrable on $V\subset \mathbb R^n$ and we define
\begin{equation}
 \label{def:density_integral}
 \int_U \Psi:=\int_V \Psi_\kappa(x) d\lambda(x)
\end{equation}
where $d\lambda(x)$ is the usual Lebesgue measure. The behavior of $\Psi_\kappa$ 
under coordinate changes (\ref{eq:coordinate_transform}) assures that this definition
is independent of the choice of charts. The same definition holds for sections in 
the positive density bundle $\rho \in \Gamma(\mathcal D^1_{\geq 0} )$. 
Note that from the definition of the global absolute value map 
(\ref{eq:def_abs_global}) we directly obtain, that $\Psi$ is 
integrable if and only if $|\Psi|_{\mathcal D^1}$ is integrable and that
\begin{equation}\label{eq:integral_inequ}
 \left|\int_U \Psi\right|\leq \int_U |\Psi|_{\mathcal D^1}.
\end{equation}
An arbitrary not necessarily compactly supported section 
$\Psi\in \Gamma(\mathcal D^1 )$ is said to be integrable if for a countable
atlas $(\kappa_i,U_i,V_i)$ and a partition of unity $\chi_i$ subordinate to 
the cover $U_i$ we have that for all $i$, $\chi_i \Psi$ are integrable as compactly
supported sections and that
\[
 \sum\limits_{i} \left(\int_{U_i} |\chi_i\Psi|_{\mathcal D^1}\right)<\infty.
\]
We then define 
\[
 \int_X \Psi:=  \sum\limits_{i} \left(\int_{U_i} \chi_i\Psi\right).
\]
Again one checks, that this definition is independent of the choice of the atlas
and partition of unity using (\ref{eq:coordinate_transform}). As above the same
definition applies for sections in the positive density bundle. Furthermore we 
have that an arbitrary section $\Psi\in \Gamma(\mathcal D^1 )$ is integrable 
if and only if $|\Psi|_{\mathcal D^1}$ is integrable and the inequality (\ref{eq:integral_inequ})
holds. If $\rho_1,\rho_2 \in \mathcal D^1_{\geq 0}$ then we say $\rho_1\leq\rho_2$
if this equality holds fiber wise. We then immediately obtain from the definition 
of the integrals
\begin{equation}\label{eq:integral_inequ_pos}
 \int_X \rho_1\leq\int_X\rho_2.
\end{equation}

Let $\mathcal V\to X$ be a Hermitian vector bundle, where the fibers $\mathcal V_x$
are Hilbert spaces with scalar products $\langle~,~\rangle_{\mathcal V_x}$ respectively
norms $\|\quad\|_{\mathcal V_x}$. By tensoring with the density bundles we can
define $L^1$ and $L^2$-norms of sections in $\mathcal V\otimes \mathcal D^\alpha $ 
($\alpha =1,1/2$) as follows: 
Given a section 
\[
f\in \Gamma(\mathcal V\otimes \mathcal D^\alpha ) , f:x\to v_x\otimes z_x
\]
we can associate a section $\|f\|_{\mathcal V\otimes \mathcal D^\alpha}$ in 
$\mathcal D^\alpha_{\geq 0}$ by setting
\[
 \|f\|_{\mathcal V\otimes \mathcal D^\alpha}: x\mapsto \|v_x\|_{\mathcal V_x} \cdot |z_x|_{\mathcal D^\alpha}.
\]
We now say that the $L^1$-norm of $f\in \Gamma(\mathcal V\otimes \mathcal D^1 )$
(respectively the $L^2$-norm of $f\in \Gamma(\mathcal V\otimes \mathcal D^{1/2} )$)
is defined if $\|f\|_{\mathcal V\otimes \mathcal D^1}$ (respectively 
$(\|f\|_{\mathcal V\otimes \mathcal D^{1/2}})^{\otimes 2}$) is integrable and set
\[
 \|f\|_{L^1} := \int_X \|f\|_{\mathcal V\otimes \mathcal D^1},
\]
respectively
\[
 \|f\|_{L^2} := \sqrt{ \int_X (\|f\|_{\mathcal V\otimes \mathcal D^{1/2}})^{\otimes 2} }.
\]
We define for $p=1,2$
\[
 \mathcal L^p(\mathcal V\otimes \mathcal D^{1/p}):=\left\{ f\in \Gamma(\mathcal V\otimes \mathcal D^{1/p}), \textup{s.t. }\|f\|_{L^p} \textup{ is defined }\right\}
\]
and
\[
 L^p(\mathcal V\otimes \mathcal D^{1/p}) := \mathcal L^p(\mathcal V\otimes \mathcal D^{1/p}) /\{f \in \mathcal L^p(\mathcal V\otimes \mathcal D^{1/p}),~\|f\|_{L^p} =0 \}
\]
then $L^1(\mathcal V\otimes \mathcal D^{1/p})$ becomes a Banach space with norm $\|\quad\|_{L^1}$
and $L^2(\mathcal{V}\otimes \mathcal D^{1/2})$ a Hilbert space with the scalar product
\[
 \langle f_1,f_2\rangle_{L^2} := \int \langle f_1(x),f_2(x)\rangle_{\mathcal V_x\otimes\mathcal D^{1/2}_x}.
\]
Here $\langle f_1(x),f_2(x)\rangle_{\mathcal V_x\otimes\mathcal D^{1/2}_x}$ denotes the 
section in $\mathcal D^1 $ which is assigned to the two sections $f_i:x\mapsto v_i(x)\otimes z_i(x)$, $i=1,2$ by 
\[
 x\mapsto \langle v_1(x),v_2(x)\rangle_{\mathcal V_x} z_1(x)\otimes \overline{z_2}(x).
\]

If the base manifold $X$ is equipped with a smooth left action by a Lie group $G$, 
then this action can be lifted to $\mathcal D^ \alpha$ as follows. Recall that 
for $(x,b)\in \mathcal F_xX$, $b :\mathbb R^ n\to T_xX$ is a linear isomorphism. 
Now for any $g\in G$ the differential of this left action is a linear isomorphism between 
the tangent spaces $dg_{|x}:T_xX\to T_{gx}X$. Thus 
we can define
\[
 l_g:\mathcal F_xX\to\mathcal F_{gx}X,\quad\quad (x,b)\mapsto (gx,dg_{|x}\circ b).
\]
This left action on the frame bundle then leads to a canonical left action on 
the density bundles
\[
 l_g:\mathcal D^ \alpha_x\to\mathcal D^ \alpha_{gx},\quad\quad [(x,b),z]\mapsto [(gx,dg_{|x}\circ b),z].
\]
If $h\in G_x$ lies in the stabilizer of the point $x\in X$, then $l_{h}$ is a 
linear isomorphism of the fiber $\mathcal D^\alpha_x\cong \mathbb C$ and we get a
one dimensional representation $(\sigma_x,\mathcal D^\alpha_x)$ of $G_x$ which is 
in general not unitary. We can even explicitly calculate this representation:
\begin{eqnarray}\label{eq:densityaction}
 \sigma_x(h)[(x,b),z]&=& [(x,dh_{|x}\circ b),z] \nonumber\\
 &=&[(x,b\circ b^ {-1}\circ dh_{|x}\circ b),z]\nonumber \\
 &=&[(x,b),|{\det}_{\mathbb R^ n}( b^ {-1}\circ dh_{|x}\circ b)|^{-\alpha} z] \nonumber\\
 &=&[(x,b),|{\det}_{T_xX}( dh_{|x})|^{-\alpha} z]\label{eq:sigma_x_explicit}.
\end{eqnarray}

\bibliographystyle{amsalpha}
\bibliography{WFSetsIII}

\end{document}